\documentclass[10pt]{article}

\usepackage{amssymb}
\usepackage{amsmath}
\usepackage{theorem}
\usepackage{epsfig}
\usepackage{verbatim}
\usepackage{graphicx}
\usepackage{subfigure}
\usepackage{cancel}
\usepackage{epstopdf}
\usepackage{mathrsfs}  
\usepackage{a4wide}
\usepackage{pdflscape}
\usepackage{xcolor}

\usepackage{color}

\usepackage{amssymb}
\usepackage{amsmath}
\usepackage{theorem}
\usepackage{epsfig}
\usepackage{verbatim}
\usepackage{graphicx}
\usepackage{subfigure}
\usepackage{epstopdf}
\usepackage{hyperref}

\usepackage{color}

\newtheorem{theorem}{Theorem}[section]
\newtheorem{lemma}[theorem]{Lemma}

\newtheorem{prop}[theorem]{Proposition}
\newtheorem{corollary}[theorem]{Corollary}

\newtheorem{rem}[theorem]{Remark}

\newtheorem{question}[theorem]{Question}

\newcommand{\colvec}[1]{\begin{pmatrix} #1 \end{pmatrix}}
\makeatletter
\newcommand*{\rom}[1]{\expandafter\@slowromancap\romannumeral #1@}
\makeatother
\newcommand{\quotes}[1]{``#1''}
\newcommand{\RR}{\mathbb{R}}

\newcommand{\N}{\mathbb{N}}

\newcommand{\ep}{\epsilon}

\newcommand{\pa}{\partial}

\newenvironment{proof}{\begin{trivlist} \item[] {\em Proof:}}{\hfill $\Box$
                       \end{trivlist}}

\newenvironment{proofthm}[1]{\begin{trivlist} \item[] {\em Proof of Theorem \ref{#1}:}}{\hfill $\Box$
                       \end{trivlist}}

\newenvironment{proofprop}[1]{\begin{trivlist} \item[] {\em Proof of Proposition \ref{#1}:}}{\hfill $\Box$
                       \end{trivlist}}

\newenvironment{prooflem}[1]{\begin{trivlist} \item[] {\em Proof of Lemma \ref{#1}:}}{\hfill $\Box$
                       \end{trivlist}}

\makeatletter
\renewcommand*\l@section{\@dottedtocline{1}{0em}{1.5em}}
\renewcommand*\l@subsection{\@dottedtocline{2}{1.5em}{2.3em}}
\renewcommand*\l@subsubsection{\@dottedtocline{3}{3.8em}{3.7em}}
\makeatother

\numberwithin{equation}{section}

\allowdisplaybreaks[4]

\begin{document}
\title{Existence of non-trivial non-concentrated compactly supported stationary solutions of the 2D Euler equation with finite energy}

\author{Javier G\'omez-Serrano, Jaemin Park and Jia Shi}

\maketitle
\begin{abstract}
In this paper, we prove the existence of locally non-radial solutions to the stationary 2D Euler equations with compact support but non-concentrated around one or several points. Our solutions are of patch type, have analytic boundary, finite energy and sign-changing vorticity and are new to the best of our knowledge. The proof relies on a new observation that finite energy, stationary solutions with simply-connected vorticity have compactly supported velocity, and an application of the Nash-Moser iteration procedure.
\end{abstract}

\tableofcontents
\section{Introduction}

Throughout this paper we will work with the two-dimensional incompressible Euler equations in vorticity form. The evolution of the vorticity $\omega$ is given by
\begin{equation}\label{euler}
\begin{cases}
\partial_t \omega+ u \cdot \nabla \omega = 0 \ &\text{ in } \mathbb{R}^2 \times \mathbb{R}_+,  \\
 u(\cdot,t) = -\nabla^{\perp}(-\Delta)^{-1}\omega(\cdot,t) &\text{ in }\mathbb{R}^2,\\
 \omega(\cdot,0) = \omega_0&\text{ in }\mathbb{R}^2,
\end{cases}
\end{equation}
where $\nabla^\perp := (-\partial_{x_2}, \partial_{x_1})$. Note that we can express $u$ as 
$
u(\cdot,t) = \nabla^\perp(\omega(\cdot,t) * \mathcal{N}),
$
where $\mathcal{N}(x) := \frac{1}{2\pi}\ln |x|$ is the Newtonian potential in two dimensions.  

In this paper we will be focusing on constructing non-radial stationary solutions to equations \eqref{euler}. We will work in the \emph{patch} setting, where $\omega(\cdot,t)= \sum \Theta_i 1_{D_i(t)}$ is an sum of indicator functions of bounded sets that move with the fluid, although some of our results translate into the smooth setting as well (where $\omega(\cdot,t)$ is smooth and compactly-supported in $x$). \color{black} See Remark~\ref{remark1}. \color{black} For well-posedness results for patch solutions, see the global well-posedness results  \cite{Bertozzi-Constantin:global-regularity-vortex-patches,Chemin:persistance-structures-fluides-incompressibles}.

\color{black}
Stationary solutions of the Euler equations are an important building block since they might play a role in many different directions: for example understanding turbulence both in 2D \cite{Caglioti-Lions-Marchioro-Pulvirenti:stationary-2d-euler} and in 3D realizing turbulent flows as a superposition of Beltrami flows (which are particular stationary solutions of 3D Euler whose curl is proportional to themselves) \cite{Constantin-Majda:beltrami-spectrum,Dombre-Frisch-Greene-Henon-Mehr-Soward:chaotic-streamlines-abc,Pelz-Yakhot-Orszag-Shtilman-Levich:velocity-vorticity-patterns-turbulence}. In the context of numerical simulation, steady solutions of the 2D Euler equations were used by Chorin \cite{Chorin:numerical-study-slightly-viscous-flow} to perform numerical simulations of the 2D Navier-Stokes equations with small viscosity, approximating the NS solutions as a superposition of steady eddies of constant vorticity that solve the 2D Euler equations. In the context of convex integration \cite{DeLellis-Szekelyhidi:turbulence-geometry-nash-onsager,Buckmaster-Vicol:convex-integration-survey}, Beltrami flows and Mikado flows are classes of stationary solutions of the Euler equations that have been used within an iteration scheme to generate fast oscillating perturbations in order to construct weak solutions.

\color{black}

In \cite{GomezSerrano-Park-Shi-Yao:radial-symmetry-stationary-solutions} we showed that any stationary solution (in both the patch and smooth settings) for which $\omega \geq 0$ and was compactly supported had to be radial. In this paper we want to address the necessity of the hypothesis $\omega \geq 0$ by answering (on the positive) the following question:
\begin{question} Do there exist non-trivial stationary solutions for which $\omega$ changes sign?
\label{question1}
\end{question}

Our first main result immediately gives the answer to the above question.
\begin{theorem}[Corollary~\ref{infinite_energy_solution}]
\label{firsttheorem}
There exist non-radial, sign-changing vortex patch solutions with analytic boundary to the 2D Euler equation~\eqref{euler} whose kinetic energy is infinite, that is, $\int_{\mathbb{R}^2}|\nabla^\perp \omega * \mathcal{N}|^2dx = \infty$.
\end{theorem}

\begin{rem}
In \cite[Theorem A]{GomezSerrano-Park-Shi-Yao:radial-symmetry-stationary-solutions}, it was shown that any non-negative stationary vortex patch must be radially symmetric up to a translation. Theorem~\ref{teoremaestacionarias}  implies that by allowing an arbitrarily small portion of negative vorticity, one can find a non-radial stationary vortex patch. More precisely, for any $\epsilon>0$, one can find a non-radial stationary vortex patch $\omega$ such that $\int_{\RR^2} \omega^-(x)dx <\epsilon$ while $\int_{\RR^2} \omega^+(x)dx$ is uniformly bounded from below, where $\omega^{-}(x) := -\omega(x)1_{\left\{ x\in \RR^2 : \omega(x) < 0 \right\}}$ and $\omega^+(x) := \omega(x)1_{\left\{ x\in \RR^2 : \omega(x) >0 \right\}}$.
 \end{rem}

Our second main theorem concerns the solutions with \textit{finite} kinetic energy. For the radial vorticity $\omega$ with zero-average, $\int_\RR \omega(x)dx=0$, its velocity vanishes outside the support of $\omega$. With this observation, one can easily produce \quotes{globally non-radial} solutions by placing multiple copies of such vorticity so that their supports do not overlap. See Figure~\ref{diagram3}. However the flow on each connected component of the support of the velocity is still circular (around different points). From now on, we say that such a solution is locally radial. More precisely a solution to the 2D Euler equation~\eqref{euler} is locally radial if each connected component of $\omega$ is radial up to a translation and $\int_{\mathcal{C}}\omega dx = 0$ for each connected component $\mathcal{C}$ of $\text{supp}(\omega)$. The next theorem states that there exist more non-trivial stationary solutions. 

\begin{figure}[h!]
\begin{center}
\includegraphics[scale=0.8]{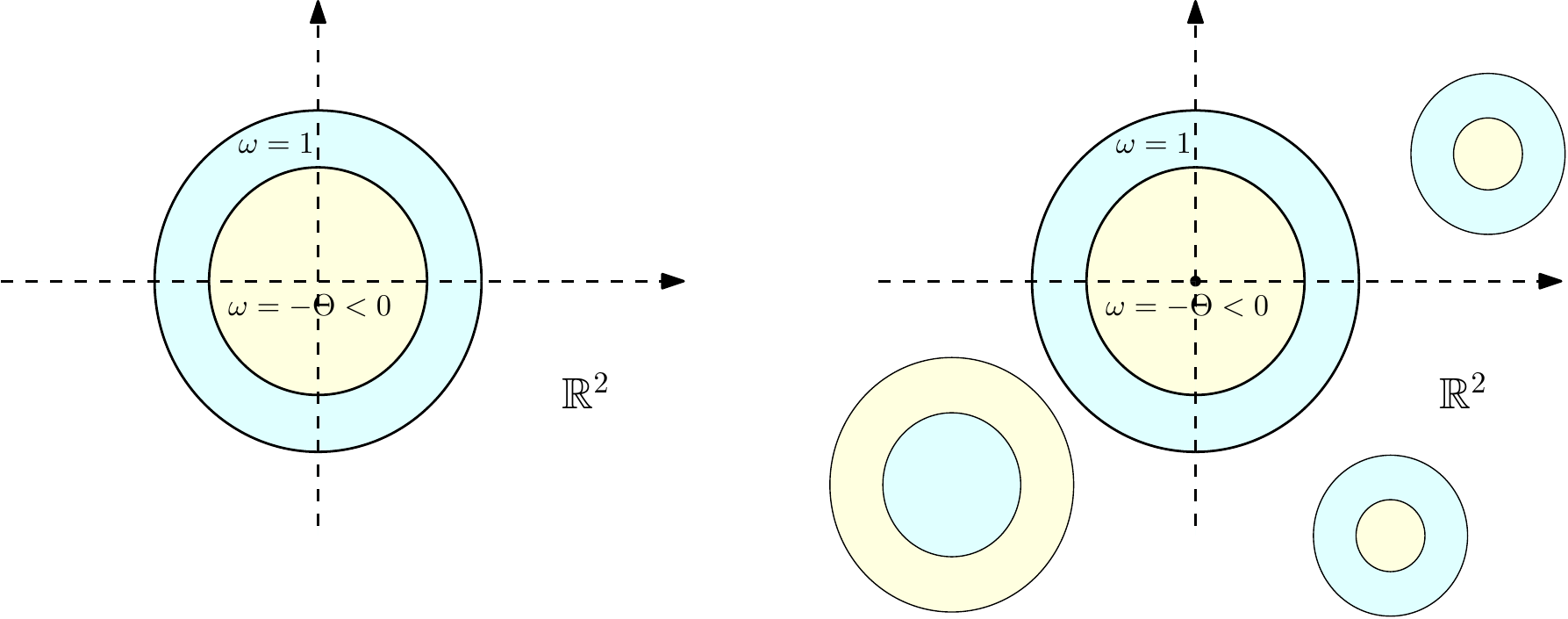}
\caption{One can choose $\Theta$ so that the radial vorticity has zero average (left). Since the support of such velocity and vorticity coincide, one can place such solutions to produce non-radial solutions (right). \label{diagram3}}
\end{center}
\end{figure}

\begin{theorem}[Corollary~\ref{corollary_1}, Theorem~\ref{ofc}]
\label{secondtheorem}
There exist  vortex patch solutions to the 2D Euler equation~\eqref{euler} that are not locally radial, with finite kinetic energy and analytic boundary. Furthermore, the solutions have compactly supported velocity.
\end{theorem}

\begin{rem}\label{remark1}
In this paper, we construct patch-type solutions with compactly supported velocity. This is a consequence of the existence of stationary solutions with finite kinetic energy and our key Lemma~\ref{zero mean}. We note that a smooth stationary solution with finite kinetic energy also has compactly supported velocity (see Remark~\ref{smooth_app}). 
\end{rem}

\color{black}

\subsection{2D Euler rigidity and construction of stationary solutions} 
In this subsection we will summarize some of the history of stationary solutions, mostly focusing on the rigidity (only trivial solutions exist) vs flexibility (non-trivial solutions exist) dichotomy. The first result goes back to Fraenkel \cite[Chapter 4]{Fraenkel:book-maximum-principles-symmetry-elliptic}, who proved that if $D$ is a stationary, simply connected patch, then $D$ must be a disk. The main idea uses the fact that in this setting, the stream function $\psi=1_D*\mathcal{N}$ solves a semilinear elliptic equation $\Delta \psi = g(\psi)$ in $\mathbb{R}^2$ with $g(\psi)=1_{\{\psi<C\}}$, and one can apply the moving plane method developed in \cite{Serrin:symmetry-moving-plane,Gidas-Ni-Nirenberg:symmetry-maximum-principle} using the monotonicity of $g$ to obtain the symmetry of $\psi$. However, this result does not cover the non simply-connected case due to the fact that $\psi$ may take different values on the different parts of the boundary and thus one can not apply moving plane techniques. This was solved by the authors and Yao in \cite{GomezSerrano-Park-Shi-Yao:radial-symmetry-stationary-solutions} using a variational approach that does not require this condition, and generalized to the smooth case as long as the vorticity is non-negative. Fraenkel's result (and method) was generalized to other classes of active scalar equations such as the generalized SQG (where the velocity in \eqref{euler} is given by the perpendicular gradient of the convolution with $\frac{1}{|x|^{\alpha}}$ as opposed to the Newtonian potential) by
 Reichel \cite[Theorem 2]{Reichel:balls-riesz-potentials}, Lu--Zhu \cite{Lu-Zhu:overdetermined-riesz-potential} and Han--Lu--Zhu \cite{Han-Lu-Zhu:characterization-balls-bessel-potentials} in the case of $\alpha \in[0,1)$ and Choksi--Neumayer--Topaloglu \cite{Choksi-Neumayer-Topaloglu:anisotropic-liquid-drop-models} in the case $\alpha \in[0,\frac53)$. In \cite{GomezSerrano-Park-Shi-Yao:radial-symmetry-stationary-solutions} we closed the problem for the full range $\alpha \in [0,2)$.

In the last few years, there has been an emergence of results on rigidity conditions, namely under which hypotheses we can guarantee that the solution has some rigid features in order to ultimately characterize stationary solutions by other geometric properties (such as being a shear or being radial). These are usually referred as `Liouville'' type of results. In the case of 2D Euler, Hamel--Nadirashvili in \cite{Hamel-Nadirashvili:liouville-euler,Hamel-Nadirashvili:shear-flow-euler-strip-halfspace} proved that any stationary solution without a stagnation point must be a shear flow whenever the domain is a strip and also in \cite{Hamel-Nadirashvili:rigidity-euler-annulus} proved the corresponding rigidity (radial symmetry) result whenever the domain is a two-dimensional bounded annulus, an exterior circular domain, a punctured disk or a punctured plane. Constantin--Drivas--Ginsberg \cite{Constantin-Drivas-Ginsberg:rigidity-flexibility-MHD, Constantin-Drivas-Ginsberg:rigidity-flexibility} obtained rigidity and flexibility results for Euler and other equations (such as MHD) in both 2D and 3D. Coti-Zelati--Elgindi--Widmayer \cite{CotiZelati-Elgindi-Widmayer:stationary-kolmogorov-poiseuille} constructed stationary solutions close to the Kolmogorov and Poiseuille flows in $\mathbb{T}^2$. In the case of 2D Navier--Stokes, Koch--Nadirashvili--Seregin--\v{S}ver\'ak also proved a Liouville theorem in \cite{Koch-Nadirashvili-Seregin-Sverak:liouville-navier-stokes}. See also \cite{GomezSerrano-Park-Shi-Yao:rotating-solutions-vortex-sheet,GomezSerrano-Park-Shi-Yao:rotating-solutions-vortex-sheet-rigidity} where together with Yao we proved rigidity and flexibility results for the vortex sheet problem.

We now review some additional results related to the characterization or construction of nontrivial stationary solutions to 2D Euler (flexibility). Nadirashvili \cite{Nadirashvili:stationary-2d-euler}, following Arnold \cite{Arnold:geometrie-differentielle-dimension-infinie,Arnold:apriori-estimate-hydrodynamic-stability,Arnold-Khesin:topological-methods-hydrodynamics} studied the geometry and the stability of stationary solutions. When the problem is posed on a surface, Izosimov--Khesin \cite{Izosimov-Khesin:characterization-steady-solutions-2d-euler} characterized stationary solutions of 2D Euler.  Choffrut--\v{S}ver\'ak \cite{Choffrut-Sverak:local-structure-steady-euler} showed that locally near each stationary smooth solution there exists a manifold of stationary smooth solutions transversal to the foliation, and Choffrut--Sz\'ekelyhidi \cite{Choffrut-Szekelyhihi:weak-solutions-stationary-euler} showed that there is an abundant set of stationary weak ($L^{\infty}$) solutions near a smooth stationary one. Shvydkoy--Luo \cite{Luo-Shvydkoy:2d-homogeneous-euler,Luo-Shvydkoy:addendum-homogeneous-euler} looked at stationary smooth solutions of the form $v = \nabla^{\perp}(r^{\gamma}f(\omega))$, where $(r,\omega)$ are polar coordinates and were able to obtain a classification of them. In a different direction, Turkington \cite{Turkington:stationary-vortices} used variational methods to construct stationary vortex patches of a prescribed area in a bounded domain, imposing that the patch is a characteristic function of the set $\{\Psi > 0\}$, and also studied the asymptotic limit of the patches tending to point vortices. He also studied the case of unbounded domains.  We emphasize that those solutions do not have finite energy, unless the domain is bounded. Long--Wang--Zeng \cite{Long-Wang-Zeng:concentrated-steady-vortex-patches} studied the regularity in the smooth setting (see also \cite{Cao-Wang:nonlinear-stability-patches-bounded-domains}) as well as their stability. For other variational constructions close to point vortices, we would like to mention the work done by Cao--Liu--Wei \cite{Cao-Liu-Wei:regularization-point-vortices}, Cao--Peng--Yan \cite{Cao-Peng:planar-vortex-patch-steady} and Smets--van Schaftingen \cite{Smets-VanSchaftingen:desingularization-vortices-euler}. Musso--Pacard--Wei \cite{Musso-Pacard-Wei:stationary-solutions-euler} constructed nonradial smooth stationary solutions with finite energy but without compact support in $\omega$.
 Our solutions are different from all of these constructions since they are not close to point vortices.

The (nonlinear $L^1$) stability of circular patches was proved by Wan--Pulvirenti \cite{Wan-Pulvirenti:stability-circular-patches} and later Sideris--Vega gave a shorter proof \cite{Sideris-Vega:stability-L1-patches}. See also Beichman--Denisov \cite{Beichman-Denisov:stability-rectangular-strip} for similar results on the strip. {Recently, Choi--Lim \cite{Choi-Lim:stability-monotone-vorticities} generalized the stability results for radial patches to radially symmetric monotone vorticity.}  
Lately, Gavrilov \cite{Gavrilov-stationary-euler-3d,Gavrilov:stationary-euler-helix} managed to construct nontrivial stationary solutions of 3D Euler with compactly supported velocity, which was further simplified and extended to other equations by Constantin--La--Vicol \cite{Constantin-La-Vicol:remarks-gavrilov-stationary}. In \cite{Dominguez-Enciso-PeraltaSalas:piecewise-smooth-stationary-euler}, Dom\'inguez-V\'azquez--Enciso--Peralta-Salas construct a different family of non-localizable, stationary solutions of 3D Euler that are axisymmetric with swirl. Regarding the stability results in 3D, we refer to the work of Choi \cite{Choi:stability-hill-vortex} for solutions near Hill's spherical vortex.

\subsection{Structure of the proofs}

The skeleton of our proof employs bifurcation theory, trying to find a perturbation of trivial solutions (radial vorticity).  A very natural procedure is to frame it using a Crandall-Rabinowitz Theorem approach (see \cite{Castro-Cordoba-GomezSerrano:uniformly-rotating-smooth-euler,Castro-Cordoba-GomezSerrano:global-smooth-solutions-sqg,Castro-Cordoba-GomezSerrano:existence-regularity-vstates-gsqg,Castro-Cordoba-GomezSerrano:analytic-vstates-ellipses,Castro-Lear:traveling-waves-couette,Garcia:Karman-vortex-street,Garcia:vortex-patch-choreography,Garcia-Hmidi-Soler:non-uniform-vstates-euler,Garcia-Hmidi-Mateu:time-periodic-3d-qg,GomezSerrano:stationary-patches,Hassainia-Hmidi:v-states-generalized-sqg,
Hmidi-Mateu-Verdera:rotating-vortex-patch,Hmidi-delaHoz-Mateu-Verdera:doubly-connected-vstates-euler,Hmidi-Mateu-Verdera:rotating-doubly-connected-vortices,Hmidi-Mateu:bifurcation-kirchhoff-ellipses} for applications in the context of fluid mechanics).  In a suitable functional setting, the problem boils down to finding non-trivial zeros of a nonlinear functional of the form (see \eqref{stationary_equation}) :
\[
\mathcal{F}(\Theta,R)=0, \quad F:\mathbb{R}\times H^{k+1}(\mathbb{T})\mapsto H^{k}(\mathbb{T}), \quad k\ge 3, \text{ given that $\mathcal{F}(\Theta,0) = 0 $ for any $\Theta\in \mathbb{R}$}.
\]
Note that the loss of derivatives attributes to the fact that each component of the functional in \eqref{stationary_equation} takes account of the tangent vector of the boundaries.
 An easy observation is that the non-degeneracy of the velocity at the trivial solution (more precisely, the angular velocity $\ne 0$) guarantees that the linearized functional is Fredholm. For a radial vorticity with infinite energy, this can be more explicitly revealed in the linearized operators obtained in Proposition~\ref{Linearized_operator}, where the the angular velocity contributes to the diagonal elements of the matrix $M_n$, which can also be seen in \eqref{derivative_matrix_3}. However, for a radial vorticity whose average vanishes,  the corresponding velocity field, which can be explicitly computed, vanishes outside the support of the vorticity. At such trivial solutions, the linearized functional fails to be Fredholm and one cannot directly apply the Crandall-Rabinowitz Theorem. Indeed, the degeneracy of the velocity on the outmost boundary yields a mismatch of the image spaces of the functional at the nonlinear level and the linear level in terms of the regularity. One possible attempt to overcome this issue is to find two bifurcation curves emanating from negative-average vorticity and positive-average vorticity that are close to each other and show that the two curves merge together forming a loop, using the strategy in \cite{Hmidi-Renault:existence-small-loops-doubly-connected-euler} (see Figure~\ref{diagram1}). It is not difficult to find (in the three-layer setting) values for $\Theta^+$ and $\Theta^-$ and doing numerical continuation one can observe those loops.
  
\begin{figure}[h!]
\begin{center}
\includegraphics[scale=1.1]{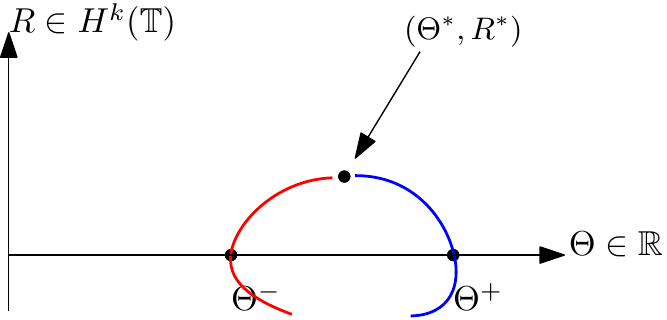}
\caption{One possible strategy: Find two bifurcation curves emanating from negative average and positive average and show that those two curves are connected. If it is possible, then there must be a solution with zero-average  $(\Theta^*,R^*)$ by  continuity. \label{diagram1}}
\end{center}
\end{figure}

 Our key Lemma~\ref{zero mean} shows that the degeneracy of the velocity is not a special case at the trivial solutions, but a generic phenomenon. More precisely, if the vorticity $\omega$ is stationary (not necessarily radial) with zero-average, then $u = 0$ outside the support of $\omega$ as long as $\text{supp}(\omega)$ is simply connected. This gives two crucial implications:
  \begin{itemize}
  \item   A non-trivial stationary solution near a trivial one cannot be obtained by using the implicit function theorem. In Figure~\ref{diagram1}, $D_R\mathcal{F}(\Theta^*,R^*)$, the linearized operator at $(\Theta^*,R^*)$ (if it exists) cannot be an isomorphism.
  
  \item  The mismatch in regularity between the image spaces described above may be treated in the sprit of Nash-Moser scheme with the use of an \textit{approximate inverse}.
  \end{itemize}
 
The first implication shows that one cannot use the Lyapunov-Schmidt reduction which is a crucial tool in the strategy in \cite{Hmidi-Renault:existence-small-loops-doubly-connected-euler} to find a loop of bifurcation curves and even more importantly in the original proof of the Crandall-Rabinowitz theorem. Regarding the second implication, the Nash-Moser scheme has been well-adapted in the context of steady-state solutions \cite{Choffrut-Sverak:local-structure-steady-euler,Iooss-Plotnikov:small-divisor,Iooss-Plotnikov-Toland:standing-waves-infinitely-deep-gravity}, dynamical solutions \cite{Lannes:well-posedness-water-waves,Rodrigo:evolution-sharp-fronts-qg} and more recently, embedded into a KAM scheme in the context of quasiperiodic solutions for the water waves problem (see \cite{Baldi-Berti-Haus-Montalto:quasiperiodic-gravity-waves-finite-depth} and the references therein).  In Section~\ref{Section4}, we will find a bifurcation curve without the use of such a {reduction technique}, and combine it with the Nash-Moser scheme in the following way:  
\begin{itemize}
\item[1)] In Subsection~\ref{main_finite_1}, we slightly modify our functional setting given in Section~\ref{functional_setting} so that $(\Theta,R)$ represents a zero-average vorticity. We are led to find non-trivial zeros of the modified functional (see \eqref{stationarR_equation3} for the definition of $G$) $$G:\mathbb{R}\times H^{k+1}\mapsto H^{k}, \quad \text{given that ${G}(\Theta,0) = 0 $ for any $\Theta\in \mathbb{R}$ }.$$  
See also \eqref{def_G_stream}, where it is shown that each {component} of $G$ is  the tangential derivative of the stream function on each boundary component.
\item[2)] We analyze the linearized operator $D_RG$ to find $\Theta^*\in \mathbb{R}$ such that $D_RG(\Theta^*,0)$ satisfies (Subsection~\ref{Spectral})
\begin{itemize}
\item[$\bullet$] Ker$(D_RG(\Theta^*,0))$ is one-dimensional, that is, Ker$(D_RG(\Theta^*,0)) = \text{span}\left\{ v\right\}$ for some $v\in C^\infty$.
\item[$\bullet$] Im$(D_RG(\Theta^*,0))^\perp$ is one-dimensional.
\item[$\bullet$] $(D_RG(\Theta^*,0))$ satisfies the transversality condition: $\partial_\Theta D_R(G(\Theta^*,0))[v]\notin \text{Im}D_RG(\Theta^*,0)$.
\end{itemize} This indicates that a possible non-trivial solution can be found as a zero of the following functional for sufficiently small $s>0$:
\begin{align}\label{tilde_g_intro}
\tilde{G}_s:H^{k+1}\mapsto H^{k}, \quad \tilde{G}_s(R):=G(\Theta^*+v\cdot R,sv + (I-P)[R]),
\end{align}
where $P:H^{k+1}\mapsto C^\infty$ is the projection to $\text{Ker}(D_RG(\Theta^*,0))$, and $I$ is the identity operator (see \eqref{def_tilde_g}).
\item[3)] We perform Newton's method for the functional $\tilde{G}_s$. The two important ingredients for Newton's method to work are 1) a sufficiently good initial guess and 2) the invertibility of the linearized operator $D\tilde{G}_s$. For the initial guess, we already have $\tilde{G}_s = O(s^2)$, since $v\in \text{Ker}(D_RG(\Theta^*,0))$. 
\item[4)] Regarding the second ingredient, we observe that the operator $D_RG(\Theta,R)$  can be decomposed 
\[ D_RG = a + A \text{ as in \eqref{decomp_11},} 
\] where the operator $a(\Theta,R)$ vanishes if $(\Theta,R)$ corresponds to a stationary solution. This is due to the fact that the velocity on the outmost boundary vanishes by Lemma~\ref{zero mean}. This observation leads us to decompose $D\tilde{G}_s$, which can be immediately computed from \eqref{tilde_g_intro} as
\[
D\tilde{G}_s = \partial_\Theta G\circ P + D_RG\circ (I-P) = \underbrace{ \partial_\Theta G\circ P + A\circ (I-P)}_{=: T_s}  + a\circ (I-P), \quad \text{ see \eqref{T_sdef}}.
\]
Thanks to the transversality condition, the one dimensionality of  $\text{Im}(D_RG(\Theta^*,0))^\perp (=\text{Im}A(\Theta^*,0)^\perp$ since $a(\Theta^*,0)=0$) is compensated by $\partial_\Theta G\circ P$. Indeed, it is the main goal of Subsection~\ref{Analysis_T} to show that $T_s$ is invertible. 
\item[5)] As described in 4) we do not have the invertibility of $D\tilde{G}_s$, however, the invertibility of $T_s$ is enough for the iteration in \eqref{definition_of_approx_sol} to work since $a$ tends to $0$ along the iteration steps towards the solution. More precisely, $T_s$ plays a role of the approximate inverse of $D\tilde{G}_s$. The loss of derivatives occurring at each iteration step is treated by the use of regularizing operator in \eqref{regularizing1} in the spirit of the Nash-Moser scheme.
\end{itemize}

\color{black}
 We note that our finite energy solutions consist of three-layered vortex patches, unlike the infinite energy case (See Subsection~\ref{main_infinite} and \ref{main_finite_1}). From the technical point of view, the two-layered, zero-average trivial solutions do not give a linearized operator that has finite dimensional Kernel. More precisely, one cannot find a pair of parameters $(b,\Theta)$ such that 1) the matrix $M_n$ in Proposition~\ref{Linearized_operator} has zero-determinant and 2) the corresponding $\omega$ determined by $(b,\Theta)$ has zero-average (see Remark~\ref{zero_mean_bifurcation1}). An important aspect of such two-layered stationary vorticity is that its stream function $\psi$ always satisfies
 \begin{align}\label{stream_formulation}
 \Delta \psi = g(\psi), \text{ for some $g:\mathbb{R}\mapsto \mathbb{R}$}.
 \end{align}
We emphasize that the solutions that we construct in this paper \textbf{cannot} be captured by \eqref{stream_formulation} unlike the earlier works, for example, \cite{Smets-VanSchaftingen:desingularization-vortices-euler,CotiZelati-Elgindi-Widmayer:stationary-kolmogorov-poiseuille,Constantin-Drivas-Ginsberg:rigidity-flexibility}. That is, there is no $g$ such that the stream function $\psi$ solves \eqref{stream_formulation}. This feature can been seen from the fact that the trivial solutions from which we bifurcate to obtain finite energy solutions exhibit non-monotone stream functions in the radial direction. 
 
 Lastly, we point out that the desingularization from point vortices is not applicable to find a stationary solution with finite energy, since a steady configuration of point vortices  has zero vortex angular momentum, which together with the finite energy hypothesis implies that the individual circulations have to be all zero (see \cite[Lemma 1.2.1]{ONeil:stationary-point-vortices}).

\subsection{Organization of the paper}
The paper is organized in the following way: Section \ref{functional_setting} sets up the functional framework and the equation that a stationary solution has to satisfy. Section \ref{patchsetting} proves Theorem \ref{firsttheorem}, in the easier case where the finite energy hypothesis is dropped. Finally, section \ref{Section4} proves Theorem \ref{secondtheorem} in the full generality setting using the Nash-Moser scheme. The Appendix contains some technical background results used throughout the proof, as well as basic bifurcation theory and some integrals needed for the spectral study.

\section{Functional equations and functional setting for stationary vortex patches}\label{functional_setting}
Let us consider vorticity $\omega$ of the form $\omega:=\sum_{i=1}^n\Theta_i1_{D_{i}}$ for some $n\in \mathbb{N}$, where $\Theta_i\in \RR$ and $D_i$ is a simply-connected domain for each $i=1,\ldots,n$. Suppose the boundary of $D_i$ is parametrized by a time-dependent $2\pi$-periodic curve $z_i(\cdot,t):\mathbb{T}\mapsto \RR^2$. The evolution equation of the boundary can be written as a system of equations, 

\begin{align}
\label{evolution_patch}
\partial_tz_j(x,t)\cdot\partial_xz_j(x,t)^{\perp}  & = u(z_j(x,t),t) \cdot\partial_xz_j(x,t)^{\perp}.
\end{align}
where $u(\cdot,t):=\nabla^{\perp}\left(\omega(\cdot,t) *\frac{1}{2\pi}\log|\cdot| \right)$ is the velocity vector.  We will look for stationary solutions to $\eqref{evolution_patch}$. Let us assume that each $z_i$ can be written as 
 \begin{align}\label{bd_parametrization}
 z_i(x)=(b_i+R_i(x))(\cos(x),\sin(x)), \quad x\in \mathbb{T},
 \end{align}
  where $b_i$ is a positive constant for each $i$. By plugging these parametrizations into \eqref{evolution_patch}, we are led to solve the following system for $b:=(b_1,\ldots,b_n)$, $\Theta:=(\Theta_1,\ldots,\Theta_n)$ and $R:=(R_1,\ldots,R_n)$:
 \begin{align}\label{stationary_equation}
0 =\mathcal{F}_j(b,\Theta,R):= u(z_j(\cdot)) \cdot z_j'^{\perp} &:= \sum_{i=1}^n \frac{1}{4\pi}\Theta_i u_{i,j}(R)\cdot z_{j}'^{\perp}=\sum_{i=1}^n \frac{1}{4\pi}\Theta_i  \underbrace{\left( u_{i,j}^{\theta}(R)R_j' - u_{i,j}^{r}(R)(b_j+R_j) \right)}_{=:S_{ij}(b_i,b_j,R_i,R_j) }.
\end{align}
where $u_{i,j}(R)(x) := \nabla^{\perp}\left( 1_{D_{i}}*\log|\cdot|^2 \right) (z_j(x))$, which can be thought of as a contribution of the $i$th patch on the $j$th curve, and $u_{i,j}^{\theta}$ and $u_{i,j}^r$ are the angular and the radial components of $u_{i,j}$. More explicitly, we have
\begin{align}\label{sij}
S_{ij}(b_i,b_j,R_i,R_j) & = \int_0^{2\pi} \cos(x-y)((b_i + R_i(y))R_j'(x) - (b_j + R_j(x))R_i'(y)) \nonumber\\
& \times \log((b_j + R_j(x))^2 + (b_i + R_i(y))^2 - 2(b_j + R_j(x))(b_i + R_i(y))\cos(x-y))dy \nonumber\\
& - \int_0^{2\pi} \sin(x-y)((b_i + R_i(y))(b_j + R_j(x)) + R_i'(y)R_j'(x)) \nonumber\\
& \times \log((b_j + R_j(x))^2 + (b_i + R_i(y))^2 - 2(b_j + R_j(x))(b_i + R_i(y))\cos(x-y))dy.
 \end{align}
 It is clear that $S_{ij}(b_i,b_j,0,0) = 0,$ since  $u_{i,j}^{r} = 0$ for  radial patches. Thus we have
 \begin{align}\label{trivial_one}
\mathcal{F}(b,\Theta,0) = 0, \text{ for any $b,\Theta$.}
 \end{align}
 In what follows, we will pick one of $\Theta_i$ as a bifurcation parameter, while the others are fixed.
 
We now proceed to discuss the functional spaces that we will use. In Section \ref{patchsetting} we will work with the following analytic spaces. Following \cite{Castro-Cordoba-GomezSerrano:analytic-vstates-ellipses}, we denote the space of analytic functions in the strip $\left\{ z \in \mathbb{C} : \text{Im}(z) \le c\right\}$ by $\mathcal{C}_w(c)$. For $k \in \mathbb{Z}$, we will consider the following spaces of $\frac{2\pi}{m}$-periodic functions:
\begin{align*}
X^{k,m}_{c} := \left\{ f(x) \in \mathcal{C}_{w}(c), \quad f(x) = \sum_{j=1}^{\infty}a_{jm}\cos(jmx),\quad  \rVert f \rVert_{X^{k,m}_c} < \infty \right\}, \\
Y^{k,m}_{c} :=\left\{ f(x) \in \mathcal{C}_{w}(c), \quad f(x) = \sum_{j=1}^{\infty}a_{jm}\sin(jmx),\quad  \rVert f \rVert_{Y^{k,m}_c}  < \infty \right\},
\end{align*}
where $\rVert f \rVert_{X^{k,m}_c} = \rVert f \rVert_{Y^{k,m}_c}:= \sum_{j=1}^\infty |a_{jm}|^2(1+jm)^{2k}(\cosh(cjm)^2 + \sinh(cjm)^2)$.

In Section \ref{Section4} we will consider the following spaces of $\frac{2\pi}{m}$-periodic functions:
\begin{align*}
X^{k,m} := \left\{ f(x) \in H^k(\mathbb{T}): \quad f(x) = \sum_{j=1}^{\infty}a_{jm}\cos(jmx),\quad  \rVert f \rVert_{X^{k,m}} < \infty \right\}, \\
Y^{k,m} :=\left\{ f(x) \in H^k(\mathbb{T}): \quad f(x) = \sum_{j=1}^{\infty}a_{jm}\sin(jmx),\quad  \rVert f \rVert_{Y^{k,m}}  < \infty \right\},
\end{align*}
where $\rVert f \rVert_{X^{k,m}} = \rVert f \rVert_{Y^{k,m}}:= \rVert f \rVert_{H^k(\mathbb{T})}$.
 
\section{Warm-up: Existence of non-radial stationary vortex patches with infinite energy}\label{patchsetting}
As a warm-up, in this section we aim to show that there exists a non-trivial patch solution with infinite kinetic energy, $\frac{1}{2\pi}\int_{\RR^2} |\nabla \left( \omega * \log|\cdot|\right)|^2 dx = \infty$. Recall that (\cite[Proposition 3.3]{Majda-Bertozzi:vorticity-incompressible-flow}),
\begin{align}\label{mean_energy}
\int_{\RR^2} |\nabla \left( \omega * \log|\cdot|\right)|^2 dx < \infty \iff \int_{\RR^2}\omega(x)dx = 0.
\end{align}
 In our proof, we will find a continuous bifurcation curve, emanating from a two-layered vortex patch whose vorticity does not has zero mean.

\subsection{Main results for infinite energy}\label{main_infinite}
Let us consider two-layer vortex patches, that is, $i\in\left\{ 1,2\right\}$ in the setting in Section~\ref{functional_setting}. For $b,\Theta \in (0,1)$, using the scaling invariance of the equations, we will choose the parameters to be
 \begin{align}\label{parameters1}
 b_1=1,\quad b_2=b,\quad \Theta_1=\Theta,\quad \Theta_2=-1,
 \end{align}
 so that if $R_1=R_2=0$, then 
 \begin{align}\label{vortexpatch}
 \omega = 
 \begin{cases}
 \Theta  & \text{ in }B_1\backslash B_b\\
 \Theta -1  & \text{ in } B_b,
 \end{cases}
 \end{align}
 where $B_r$ denotes the unit disk centered at the origin.
  Note that the case $\Theta=1$ corresponds to an annulus. Later, we will fix $b$ as well, and let $\Theta$ play the role as a  bifurcation parameter. With this setting, the system \eqref{stationary_equation} is equivalent to
 \begin{align}
 \label{stationary_equation2}
 0=\mathcal{F}(\Theta,R)=(F^1(\Theta,R),F^2(\Theta,R)), \quad R=(R_1,R_2),
 \end{align}
where we omit the dependence of $\Theta_2$ and $b_1$, $b_2$ for notational simplicity.

\begin{theorem}
\label{teoremaestacionarias}
Let $k,m$ be such that  $k\geq 3$, $2 \le m \in \mathbb{N}$. Let $b$ satisfy the condition in Lemma \ref{propbstar}.  Then for some $c>0$ and $s_0=s_0(k,m,c,b)>0$, there exist two bifurcation curves $[0,s_0)\ni s\mapsto (\Theta^{\pm}(s),R^{\pm}(s)) \in \RR \times (X^{k,m}_{c} )^2$ such that for each $s \in (0,s_0)$, $(\Theta^{\pm}(s),R^{\pm}(s))$ is a solution of the equation \eqref{stationary_equation2} and $ R^{\pm}(s) \ne 0 \in (X^{k,m}_c)^2$. The bifurcation curve emanates from $(\Theta^{\pm}(0),R^{\pm}(0)) = (\Theta_m^{\pm},0)$, where $\Theta_m^{\pm}$ are defined in Lemma \ref{propbstar}.
\end{theorem}

Theorem~\ref{teoremaestacionarias} immediately implies the existence of non-radial stationary vortex patches with infinite energy.
 
\begin{corollary}\label{infinite_energy_solution}
Let $2\le m\in \N$. Then there is an $m$-fold symmetric stationary patch solution for the 2D Euler equation with analytic boundary and infinite kinetic energy, that is \[\int_{\RR^2} |\nabla \left( \omega * \frac{1}{2\pi}\log|\cdot| \right)|^2 dx = \infty.\] 
\end{corollary}
\begin{proof}
 By Theorem~\ref{teoremaestacionarias}, there are two continuous bifurcation curves  $\Psi^{\pm} : [0,s_0) \mapsto \RR \times \left( X^{k,m}_c\right)^2$ of solutions of \eqref{stationary_equation2}  such that
 
 \begin{align*}
  \Psi^{\pm}(s) = \left( \Theta^{\pm}(s),R(s) \right), \quad F(\Psi^{\pm}(s))=0, \quad \text{ and } \quad \Psi^{\pm}(0)=\left( \Theta_m^{\pm},0\right),
\end{align*}
for some $s_0>0$. From \eqref{evolution_patch} and \eqref{stationary_equation}, it is clear that for each $s\in (0,s_0)$ and for each choice of $\pm$, 

\begin{align}\label{vorticity}
\omega^s(x) :=  
\begin{cases}
\Theta-1 & \text{ if }x\in D_2(s), \\
\Theta & \text{ if }x\in D_1(s)\backslash \overline{D_2(s)}\\
0 & \text { otherwise, }
\end{cases}
\end{align}
is a stationary solution to the Euler equation, where  $D_1(s)$ and $D_2(s)$ are the bounded domains determined by \eqref{bd_parametrization} with $R(s)$. Since $R(s)\in \left( X^{k,m}_c\right)^2$, the boundaries are analytic.  

Now we consider the kinetic energy of the solution.  From \eqref{mean_energy}, it suffices to show that $\int_{\RR^2} \omega^0(x) dx \ne 0$. By the continuity of the bifurcation curve, this immediately implies that $\int_{\RR^2} \omega^s(x) dx \ne 0$ for small $s>0$.  Then it follows from Lemma~\ref{bstar_theta} that $b^2 < \Theta_m^{\pm}$, hence
\begin{align*}
\int_{\RR^2}\omega^0(x)dx=\pi\left(-b^2+\Theta_m^{\pm}\right)>0.
\end{align*}
 This completes the proof.
 \end{proof}

  The rest of this section will be devoted to prove Theorem~\ref{teoremaestacionarias}. The proof will be divided into 5 steps. These steps correspond to check the hypotheses of the Crandall-Rabinowitz theorem~\ref{CRtheorem} for our functional $F$ in  \eqref{stationary_equation2}. The hypotheses in Theorem~\ref{CRtheorem} can be read as follows in our setting:
\begin{enumerate}
\item The functional $\mathcal{F}$ satisfies $$\mathcal{F}(\Theta,R)\,:\, (0,1)\times \{V^{\epsilon}\}\mapsto (Y^{k-1,m}_{c})^2,$$ where $V^{\epsilon}$ is an open neighborhood of 0,
\begin{align*}
V^{\epsilon}=\left\{
(f,g)\in (X^{k,m}_{c})^2\,:\, ||f||_{X^{k,m}_{c}}+||g||_{X^{k,m}_{c}}<\epsilon \right\}
\end{align*}
for some $\epsilon>0$ and $k\geq 3$.
\item $\mathcal{F}(\Theta,0) = 0$ for every $0 < \Theta< 1$.
\item The partial derivatives $\partial_{\Theta} \mathcal{F}$, $D\mathcal{F}$, $\partial_{\Theta} D\mathcal{F}$ exist and are continuous, where $D\mathcal{F}$ is Gateaux derivative of $\mathcal{F}$ with respect to the functional variable $R$.
\item Ker$(D\mathcal{F}(\Theta_m^{\pm},0)) \subset (X^{k,m}_c)^2$ and $(Y^{k-1,m}_{c})^2$/Im($D\mathcal{F}(\Theta_m^{\pm},0)$) are one-dimensional (see Proposition \ref{propbstar} for the definition of $\Theta_{m}^{\pm}$).
\item $\pa_{\Theta} D\mathcal{F}(\Theta_m^{\pm},0)[v_0] \not \in$ Im($D\mathcal{F}(\Theta_m^{\pm},0)$), where $v_0$ is a non-zero element in Ker$(D\mathcal{F}(\Theta_m^{\pm},0))$.
\end{enumerate}

\begin{rem}\label{analyticity_of_S}
We remark that if $i \ne j$ then the functions inside the logarithm in $S_{i,j}$ in \eqref{sij} are uniformly bounded from below in $y$ for all $x$ by a strictly positive constant depending on the parameters. Then we can analytically extend the integrand in $x$ to the strip $|\Im(z)| \leq c$ in such a way that the real part of this extension stays uniformly bounded away from 0 for a small enough $c$. The case $i=j$ can be treated similarily as in \cite[Remark 2.1]{Castro-Cordoba-GomezSerrano:analytic-vstates-ellipses}.

\end{rem}
\subsection{Proof of Theorem~\ref{teoremaestacionarias}}
\subsubsection{Steps 1,2 and 3: Regularity}\label{regularity_step}
In order to check the first three steps, it suffices to check if $S_{ij}$ in \eqref{sij} satisfies the hypotheses, since $F$ is  a linear combination of $S_{ij}$. As mentioned in Remark~\ref{analyticity_of_S}, the case $i\ne j$ is trivial since there is no singularity in the integrand and analytically extended into a strip in $\mathbb{C}$, if $(R_1,R_2)$ is in a sufficiently small neighborhood of $(0,0)$.  For $i=j$, the first three steps with slightly different settings were already done in the literature. For example, step 1 can be done in the same way as in \cite{delaHoz-Hassainia-Hmidi:doubly-connected-vstates-gsqg}. Step 2 follows immediately from \eqref{trivial_one}. Existence and continuity of the Gateaux derivatives for the gSQG equation was done in \cite{GomezSerrano:stationary-patches} and the same proof can be adapted to our setting straightforwardly.
\subsubsection{Step 4: Analysis of the linear part.}
 In this section, we will focus on the spectral study of the Gateaux derivative $D\mathcal{F}(\Theta,0):=D_{R}\mathcal{F}(\Theta,0)$.
\paragraph{Calculation of $D\mathcal{F}$}
We aim to express the Gateaux derivative of $\mathcal{F}$ around $(\Theta,0)$ in the direction $(H(x),h(x))$ in terms of Fourier series.
 \begin{lemma}\label{linearpart1}
Let $S_{ij}$ be defined as in \eqref{sij}. Then:
\begin{align*}
\left.\frac{d}{dt}S_{ij}(b_i,b_j,th_i,th_j)\right|_{t=0} & = \int b_i (h_j'(x)\cos(x-y)-h_i'(y))\log(b_j^2 + b_i^2 - 2b_jb_i\cos(x-y))dy =: \mathcal{L}_1+\mathcal{L}_2.
\end{align*}
\end{lemma}
\begin{proof}
Let $V^{1,ab}_{ij}$ (resp. $V^{2,ab}_{ij}$) be the contribution of the first term (resp. second term) of \eqref{sij} where the first factor contributes with a $t^a$ and the second with $t^b$. We are looking for all combinations such that $a+b = 1$. We start looking at the first summand. We have that:
\begin{align*}
V_{ij}^{1,10} & = \int \cos(x-y)(b_ih_j'(x)-b_jh_i'(y))\log(b_j^2 + b_i^2 - 2b_jb_i\cos(x-y))dy.
\end{align*}
Similarly, for the second one,
\begin{align*}
V_{ij}^{2,01} & = -2\int \sin(x-y)(b_i b_j) \frac{h_j(x)(b_j-b_i\cos(x-y)) + h_i(y)(b_i-b_j\cos(x-y))}{b_j^2+b_i^2-2b_jb_i\cos(x-y)}dy, \\
V_{ij}^{2,10} & = -\int \sin(x-y)(b_ih_j(x) + b_jh_i(y))\log(b_j^2 + b_i^2 - 2b_jb_i\cos(x-y))dy,
\end{align*}
where we have used Lemma \ref{lemaexpansionlog}. Integrating by parts in $V_{ij}^{2,01}$:
\begin{align*}
V_{ij}^{2,01} & = \int ((b_i h_j(x) + b_jh_i(y))\sin(x-y) - h_i'(y)(b_i - b_j\cos(x-y))) \log(b_j^2+b_i^2-2b_jb_i\cos(x-y)) dy,
\end{align*}

Finally, adding all the log terms and the non-log terms together:
\begin{align*}
V_{ij}^{1,10} + V_{ij}^{2,01} + V_{ij}^{2,10} & = \int b_i (h_j'(x)\cos(x-y)-h_i'(y))\log(b_j^2 + b_i^2 - 2b_jb_i\cos(x-y))dy,
\end{align*}
as we wanted to prove.
\end{proof}

\begin{lemma}\label{linearpart2}
Let $h_i = A_i \cos(mx)$ and $r = \frac{\min\{b_i,b_j\}}{\max\{b_i,b_j\}}$. Then:
\begin{align*}
\left.\frac{d}{dt}S_{ij}(b_i,b_j,th_i,th_j)\right|_{t=0}  =
2\pi\sin(mx)m b_i \left(A_j r - A_i \frac{r^m}{m}\right).
\end{align*}
\end{lemma}

\begin{proof}
From Lemma~\ref{linearpart1} and Corollary~\ref{integral_a0}, we have that 
\begin{align*}
\mathcal{L}_{1} & = -m b_i A_j\sin(mx) \mathcal{A}_{0}(r,1), \\
\mathcal{L}_{2} & = m b_i A_i \sin(mx) \mathcal{A}_{0}(r,m),
\end{align*}
Adding the two contributions gives the desired result.
\end{proof}
 Note that the functional $\mathcal{F}$ is a linear combination of $S_{ij}$ (see \eqref{stationary_equation}). Using the above two lemmas, we obtain the following proposition:
\begin{prop}\label{Linearized_operator}
Let $h(x) = \sum_{n}a_n \cos(nx),$ $H(x) = \sum_{n}A_n\cos(nx)$, then we have that:
\begin{align*}
D\mathcal{F}(\Theta,0)[H,h] = \left(\begin{array}{c}U(x) \\ u(x) \end{array}\right),
\end{align*}
where
\begin{align*}
 U(x) = \sum_{n}U_n \sin(nx), \quad u(x) = \sum_{n} u_n \sin(nx),
\end{align*}
and the coefficients satisfy, for any $n$:
\begin{align*}
\left(\begin{array}{c}U_n \\ u_n \end{array}\right)
:=
(-n) M_n(\Theta)
\left(\begin{array}{c}A_n \\ a_n \end{array}\right)
:= 
(-n) \left(
\begin{array}{cc}
 \frac{b^{2}}{2} - \frac{\Theta}{2} + \frac{\Theta}{2n} & - \frac{b^{n+1}}{2n} \\
\Theta\frac{b^{n}}{2n} & -\frac{b}{2n} + \frac{b}{2}(1-\Theta)
\end{array}
\right)
\left(\begin{array}{c}A_n \\ a_n \end{array}\right).
\end{align*}
\end{prop}
\begin{proof}
It follows from \eqref{parameters1}, the definition of $\mathcal{F}$ in \eqref{stationary_equation2}, \eqref{stationary_equation} and Lemma~\ref{linearpart2}.
\end{proof}
\paragraph{One dimensionality of the Kernel of the linear operator.}
Our goal here is to verify the one dimensionality of Ker($D\mathcal{F}(\Theta,0)$) for some $\Theta$. More precisely, we will prove the following proposition:
\begin{prop}\label{onedimensionality}
Fix $2\leq m \in \mathbb{N}$ and take any $b\in (0,b_m)$, where $b_m$ is as in Lemma~\ref{propbstar}.   Then there exist two $\Theta_m^{\pm} \in (0,1)$, such that Ker$(D\mathcal{F}(\Theta_m^{\pm},0))$ is one-dimensional. Furthermore, 
\begin{align*}
\text{Ker}(D\mathcal{F}(\Theta_m^{\pm},0,0)) = \text{span}\left\{ v_0(\Theta_m^{\pm})\cos(mx) := \begin{pmatrix}\frac{1}{2m}b - \frac{b}{2}(1-\Theta_m^{\pm})\\ \frac{\Theta_m^{\pm}}{2m}b^m \end{pmatrix} \cos(mx) \right\} \subset X^{k,m}_c \times X^{k,m}_c.
\end{align*}
\end{prop}
The proof of the above proposition relies on the analysis of the matrix $M_n$ in Lemma~\ref{propbstar} and \ref{propbstar2}, which we will prove below.
\begin{lemma}\label{propbstar}
Let $\Delta_{m}(\Theta)$ be 
\begin{align}\label{determinantm}
 \Delta_{m}(\Theta) := \frac{4m^2}{b}\text{det}(M_{m}(\Theta)) = \left(\Theta b^{2m} + b^{2}m(m(1-\Theta) - 1) + \Theta(1-m)(m(1-\Theta)-1)\right)
\end{align}
Then, for any $m \geq 2$, there exists $0 < b_m < 1$ such that  for any $0<b<b_m<1$, there exists $0<\Theta^{-}_m<\Theta^{+}_m<1$ such that $\Delta_{m}(\Theta_{m}^{\pm}) = 0$. We also have that rk$(M_{m}(\Theta_{m}^{\pm})) = 1$ for those values of $\Theta^{+}_m$, $\Theta^{-}_m$, where $rk(A)$ is the rank of a matrix $A$.
\end{lemma}
\begin{proof}
 For fixed $m$, we study the polynomial $\Delta_{m}(\Theta)$.
 We need to solve
 \begin{align}\label{thetaequation}
     \Delta_{m}(\Theta)=m(m-1)\Theta^2-((m-1)^2+b^2m^2-b^{2m})\Theta+b^2m(m-1)=0.
 \end{align}
 Since $\Delta_m(\Theta)$ is a quadratic function, we only need to show that the discriminant is positive. We have
 \begin{align*}
     &D_m:=((m-1)^2+b^2m^2-b^{2m})^2-4m^2(m-1)^2b^2\\
     &=((m-1)^2+b^2m^2-b^{2m}-2m(m-1)b)((m-1)^2+b^2m^2-b^{2m}+2m(m-1)b)\\
     &=(m-1-bm-b^{m})(m-1-bm+b^{m})((m-1)^2+b^2m^2-b^{2m}+2m(m-1)b)\\
     &=:D_{m,1}\cdot D_{m,2}\cdot D_{m,3}.
 \end{align*}
 Since $m\geq 2$, $0<b<1$, we have $D_{m,3}>1-1=0$. We also have
 \[
 D_{m,2}=m(1-b)+b^m-1=(1-b)[m-\frac{1-b^m}{1-b}]=(1-b)(m-(1+b+b^2+...+b^{m-1}))>0.
 \]
 $D_{m,1}$ is decreasing in $b$ when $b\geq 0$ and $D_{m,1}(0)=m-1>0$, $D_{m,1}(1)=-2$. Let $b_m$ be the only  zero point of $D_{m,1}$ in (0,1). If we take $0<b<b_m$, we have
 \begin{align}\label{d1mnonzero}
 D_{1,m}>0, D_{m}>0. 
 \end{align}
 Hence $\Delta_m(\Theta)$ has two different solutions $\Theta_m^-<\Theta_m^+$.
 Moreover, the matrix does not vanish when $b\neq 0$, implying rk$(M_{m}(\Theta_{m}^{\pm})) = 1$. 
 Now we are left to show $0<\Theta_m^{\pm}<1$. We have \[
\Theta_m^{+}\Theta_m^{+}=b^2<1,
\] and
\[
\Theta_m^{+}+\Theta_m^{-}=\frac{((m-1)^2+b^2m^2-b^{2m})}{m(m-1)}\geq \frac{1-1}{m(m-1)}>0.
\]
Hence, $0<\Theta_m^{-}<1$ and $0<\Theta_m^{+}$. If $\Theta_m^{+}\geq1$, then $\Delta_{m}(1)\leq 0$. However,
\begin{align}\label{deltam1}
    &\Delta_m(1)=b^{2m}-mb^2-(1-m)\\\nonumber
    &=(1-b^2)(m-\frac{1-b^{2m}}{1-b^2})\\\nonumber
    &=(1-b^2)(m-1-b^2-...-b^{2m-2})>0.
\end{align}
Therefore $0<\Theta_m^{\pm}<1$.
\end{proof}

We now show that $\Delta_{jm}(\Theta_{m}^{\pm})\neq 0$ for any $j \neq 1$. 

\begin{lemma}\label{propbstar2}
 Let $j > 1$ and let $\Theta_{m}^{\pm}$ and $\Theta_{jm}^{\pm}$ be defined as in the previous Lemma. Then $\Theta_{jm}^{+}>\Theta_{m}^{+}>\Theta_{m}^{-}>\Theta_{jm}^{-}$. Hence, $M_{jm}(\Theta_{m}^{\pm})$ is non-singular for all $j>1$.
\end{lemma}
\begin{proof}
Since $\Theta_{m}^{\pm}$ solves the equation
\[
\Theta^2-\frac{((m-1)^2+b^2m^2-b^{2m})}{m(m-1)}\Theta+b^2=0.
\]
We only need to show $F(b,m):=\frac{((m-1)^2+b^2m^2-b^{2m})}{m(m-1)}$ is strictly increasing with respect to $m$.
We have
\begin{align*}
    &F(b,m+1)-F(b,m)\\
    &=\frac{(m^2+b^2(m+1)^2-b^{2(m+1)})}{(m+1)m}-\frac{((m-1)^2+b^2m^2-b^{2m})}{m(m-1)}\\
    &=\frac{(m^2+b^2(m+1)^2-b^{2(m+1)})(m-1)-((m-1)^2+b^2m^2-b^{2m})(m+1)}{(m+1)m(m-1)}\\
    &=\frac{b^2((m+1)^2(m-1)-m^2(m+1))+m^2(m-1)-(m+1)(m-1)^2-b^{2m+2}(m-1)+(m+1)b^{2m}}{(m+1)m(m-1)}\\
    &=\frac{-b^2(m+1)+m-1-b^{2m+2}(m-1)+b^{2m}(m+1)}{(m+1)m(m-1)}.
\end{align*}
Thus
\begin{align*}
   &F(b,m+1)-F(b,m)>0\Leftrightarrow -b^2(m+1)+m-1-b^{2m+2}(m-1)+b^{2m}(m+1)>0\\
   &\Leftrightarrow (1+b^2)(-1+b^{2m})-(-1+b^2)(1+b^{2m})m>0\Leftrightarrow -(1+b^2)(b^{2m-2}+b^{2m-4}...+b^2+1)+(1+b^{2m})m>0\\
     &\Leftrightarrow -\sum_{k=1}^{m}b^{2m-2k}-\sum_{k=1}^{m}b^{2k}+(1+b^{2m})m>0 \Leftrightarrow \sum_{k=1}^{m}(1-b^{2k})(1-b^{2m-k})>0.
\end{align*} 
It is easy to show the last inequality since $0<b<1$.
\end{proof}

\begin{proofprop}{onedimensionality}
Let $m \ge 2$  and let $b$, $\Theta_m^{\pm}$ be as defined in Lemma~\ref{propbstar}. Assume that $H(x) = \sum_{j}A_{jm}\cos(jmx)$ and $h(x) = \sum_{j}a_{jm} \cos(jmx)$ satisfy 
\[
DF(\Theta_m^{\pm},0)[H,h] = (0,0).
\]
 Then it follows from Proposition~\ref{Linearized_operator} that
\begin{align*}
M_{jm}(\Theta_m^{\pm})
\begin{pmatrix}
A_{jm}\\
a_{jm}
\end{pmatrix}
= 
\begin{pmatrix}
0\\
0
\end{pmatrix}.
\end{align*}
For all $j>1$, it follows from Lemma~\ref{propbstar2} that $M_{jm}(\Theta_m^{\pm})$ is invertible, thus $A_{jm} = a_{jm} = 0$. For $j=1$, Lemma~\ref{propbstar} tells us that $(A_m,a_m)\in \text{Ker}(M_m(\Theta_m^{\pm})) = \text{span}\left\{ v_0(\Theta_m^{\pm}) := \begin{pmatrix}\frac{1}{2m}b - \frac{b}{2}(1-\Theta_m^{\pm})\\ \frac{\Theta_m^{\pm}}{2m}b^m \end{pmatrix} \right\}$. This finishes the proof.
\end{proofprop}

\paragraph{Codimension of the image of the linear operator.}
 We now characterize the image of $D\mathcal{F}(\Theta_m^{\pm},0)$. We have the following proposition:
\begin{prop}\label{codimension_one}
Let
\begin{align*}
Z = \left\{(Q,q) \in Y^{k-1,m}_c \times Y^{k-1,m}_c, Q(x) = \sum_{j=1}^{\infty}Q_{jm}\sin(jmx), q(x) = \sum_{j=1}^{\infty}q_{jm}\sin(jmx), \right.\\ \left.\exists \lambda_{Q,q} \in \mathbb{R} \text{ s.t.} 
\left(\begin{array}{c}Q_{m} \\ q_{m}\end{array}\right)= \lambda_{Q,q} \left(
\begin{array}{c}
 -\frac{1}{2m}b^{m+1} \\
-\frac{1}{2m}b + \frac{b}{2}(1-\Theta_m^{\pm})
\end{array}
\right)\right\}.
\end{align*}
Then $Z = \text{Im}\left(D\mathcal{F}(\Theta_m^{\pm},0,0)\right)$.
\end{prop}


\begin{proof}
In view of Proposition~\ref{Linearized_operator},  $\text{Im}\left( D\mathcal{F}(\Theta_m^{\pm},0) \right) \subset Z$ is trivial, since $M_{jm}(\Theta_m^{\pm})$ is non-singular for $j>1$, and $\text{Im}(M_m(\Theta_m^{\pm})) = \text{span}\left\{ \begin{pmatrix}
 -\frac{1}{2m}b^{m+1} \\
-\frac{1}{2m}b + \frac{b}{2}(1-\Theta)
\end{pmatrix}
\right\}.$ In order to prove $\text{Im}\left( D\mathcal{F}(\Theta_m^{\pm},0) \right) \supset Z$, we need to check whether the possible preimage satisfies the desired regularity. To do so, we have the following asymptotic lemma:
\begin{lemma}\label{bstar_theta}
For fixed $m\geq 2$ and $b$ defined as in Proposition~\ref{onedimensionality}, we have $b^2-\Theta_m^{\pm}< 0 $ and
\begin{align}
\frac{b}{4j^2m^2}\Delta_{jm}(\Theta_m^{\pm})= \frac{b}{4}\left( b^2-\Theta_m^{\pm} \right)(1-\Theta_m^{\pm})+O\left( \frac{1}{jm} \right), \quad \text{ as }j\to\infty,
\end{align}
Consequently, we have
\begin{align}\label{asymptote}
\text{det}{(M_{jm}(\Theta_m^{\pm}))^{-1}} \lesssim_{m,\Theta} 1, \quad \text{ for sufficiently large }j.
\end{align}
\end{lemma}
\begin{rem}\label{zero_mean_bifurcation1}
As shown in the above lemma, there is no bifurcation curve from the two-layered vortex patch with zero-average. This is due to the fact that the radial vorticity $\omega$ determined by $b$ and $\Theta^{\pm}$ as in \eqref{vortexpatch} satisfies $\int_{\mathbb{R}^2}\omega dx = \pi\left(-b^2+\Theta_m^{\pm}\right) > 0$. Note that if we require $\Theta=b^2$ to ensure $\int \omega dx = 0$, it follows from \eqref{determinantm} that $\Delta_m(b^2)=mb^2(1-b^2)+b^2(b^{2m}-1)$, which does not vanish for any $m\ge 2$ unless $b=0$ or $b=1$. Therefore, for any $0<b<1$, the linearized operator is an isomorphism and the implicit function theorem shows that there cannot be a bifurcation.
\end{rem}
\color{black}
\begin{prooflem}{bstar_theta}
First we show that $b^2< \Theta_m^{\pm}$. By Lemma \ref{propbstar}, $\Theta_m^+ \Theta_m^-=b^2$ and $0<\Theta_m^{\pm}<1$. Thus $\Theta_m^{\pm}>b^2$.  Therefore the first assertion is proved. The second assertion follows directly from \eqref{determinantm} since
\begin{align*}
\frac{b}{4(jm)^2}\Delta_{jm}(\Theta_m^{\pm})=\frac{b}{4(jm)^2}\left((b^2-\Theta_m^{\pm})(1-\Theta_m^{\pm})(jm)^2+O\left( m \right) \right)=\frac{b}{4}\left(b^2-\Theta_m^{\pm}\right)(1-\Theta_m^{\pm})+O\left( \frac{1}{jm} \right).
\end{align*}
Lastly, by choosing $j$ large so that $|\frac{b}{4(jm)^2}\Delta_{jm}(\Theta_m^{\pm})|>\frac{b}{8}(\Theta_m^{\pm}-b^2)(1-\Theta_m^{\pm})>0$, we have
\begin{align*}
|\text{det}(M_{jm}(\Theta_m^{\pm})^{-1})| =  \frac{1}{|\frac{b}{4(jm)^2}\Delta_{jm}(\Theta_m^{\pm})|}\lesssim 1,
\end{align*}
which proves \eqref{asymptote}.
\end{prooflem}
Now for an element $(Q,q)\in Z$, let $(H,h)$ be such that

\begin{align*}
H(x)=\sum_{j=1}^\infty A_{jm}\cos(jmx), \quad h(x)=\sum_{j=1}^\infty a_{jm}\cos(jmx),
\end{align*}
with
\begin{align*}
\begin{pmatrix}
A_m\\
a_m
\end{pmatrix}
=-\frac{1}{m}
\begin{pmatrix}
0\\
\lambda_{Q,q}
\end{pmatrix},
\quad
\begin{pmatrix}
A_{jm}\\
a_{jm}
\end{pmatrix}
=
(-jm)^{-1}M_{jm}(\Theta_m^{\pm})^{-1}
\begin{pmatrix}
Q_{jm}\\
q_{jm}
\end{pmatrix}
\quad \text{ for $j>1$}.
\end{align*}
It is clear from \eqref{Linearized_operator} that $D\mathcal{F}(\Theta_m^{\pm},0,0)(H,h)=(Q,q).$ We will prove that $(H,h) \in X^{k,m}_c \times X^{k,m}_c$.
From Lemma~\ref{bstar_theta} and the fact that $M_{jm}(\Theta_m^{\pm})$ is nonsingular for $j>1$,  it follows that
\begin{align}\label{asymptote2}
|A_{jm}|^2+|a_{jm}|^2 \lesssim (jm)^{-2}\left(|Q_{jm}|^2+|q_{jm}|^2\right) \quad \text{ for all }j>1.
\end{align}
Thus, we obtain
\begin{align*}
\rVert H\rVert_{X^{k,m}_c}^2+\rVert h\rVert_{X^{k,m}_c}^2 &= \sum_{j=1}^{\infty}\left( |A_{jm}|^2+|a_{jm}|^2 \right)(1+jm)^{2k}(\cosh(cjm)^2+\sinh(cjm)^2)\\
 &  \lesssim \frac{1}{m^{2}}\lambda_{Q,q}^{2}(1+m)^{2k}(\cosh(cm)^{2} + \sinh(cm)^{2}) \\
& + \sum_{j=2}^{\infty} (jm)^{-2}\left(|Q_{jm}|^2+|q_{jm}|^2\right)  (1+jm)^{2k}(\cosh(cjm)^{2} + \sinh(cjm)^{2})\\
&\lesssim 1+(\rVert Q\rVert_{Y^{k-1,m}_c}^2+\rVert q\rVert_{Y^{k-1,m}_c}^2)\\
& <\infty.
\end{align*}
This proves that $(H,h)\in X^{k,m}_c\times X^{k,m}_c$, and therefore $Z\subset \text{Im}\left( D\mathcal{F}(\Theta_m^{\pm},0) \right)$.
\end{proof}

\subsubsection{Step 5: Transversality}\label{Step_5_Transversality}
\begin{prop}\label{transv_prop} We have that
\begin{align}\label{transversality_1}
\pa_{\Theta} D\mathcal{F}(\Theta_m^{\pm},0)[v_0] \not \in \text{Im}(D\mathcal{F}(\Theta_m^{\pm},0)),
\end{align}
where $v_0=v_0(\Theta_m^{\pm})$ is as given in Proposition~\ref{onedimensionality}.
\end{prop}
\begin{proof}
  For $h(x) = \sum_{j}a_{jm} \cos(jmx),$ $H(x) = \sum_{j}A_{jm}\cos(jmx)$, we have that (see Proposition~\ref{Linearized_operator}):
\begin{align*}
\partial_\Theta D\mathcal{F}(\Theta_m^{\pm},0)[H,h] = \left(\begin{array}{c}U(x) \\ u(x) \end{array}\right),
\end{align*}
where
\begin{align*}
 U(x) = \sum_{j}U_{jm} \sin(jmx), \quad u(x) = \sum_{j} u_{jm} \sin(jmx),
\end{align*}
and the coefficients satisfy, for any $j$:
\begin{align*}
\left(\begin{array}{c}U_{jm} \\ u_{jm} \end{array}\right)
:=
(-jm) \partial_{\Theta}M_{jm}(\Theta_m^{\pm})
\left(\begin{array}{c}A_{jm} \\ a_{jm} \end{array}\right)
:= 
(-jm) \left(
\begin{array}{cc}
 -\frac12 + \frac{1}{2m} & 0 \\
\frac{b^{m}}{2m} & -\frac{b}{2}
\end{array}
\right)
\left(\begin{array}{c}A_{jm} \\ a_{jm} \end{array}\right).
\end{align*}

Letting 

\begin{align*}
 v_{0}(\Theta_m^{\pm}) = 
\left(
\begin{array}{c}
 \frac{b}{2m} - \frac{b}{2}(1-\Theta_m^{\pm}) \\
\frac{\Theta_m^{\pm}}{2m}b^{m}
\end{array}
\right), \quad
 w(\Theta_m^{\pm}) = 
\left(
\begin{array}{c}
 -\frac{1}{2m}b^{m+1} \\
-\frac{1}{2m}b + \frac{b}{2}(1-\Theta_m^{\pm})
\end{array}
\right), 
\end{align*}

be the generators of Ker$(M_{m}(\Theta_m^{\pm}))$ and Im$(M_{m}(\Theta_m^{\pm}))$ respectively, the transversality condition is equivalent to prove that $w_1(\Theta_m^{\pm})$ and $w(\Theta_m^{\pm})$ are not parallel, where

\begin{align*}
w_{1}(\Theta_m^{\pm}) & =  \pa_{\Theta} M_{m}(\Theta_m^{\pm}) v_{0}(\Theta_m^{\pm}) =
\left(
\begin{array}{c}
\frac{b}{4}\left(\frac{1}{m}-1\right)\left(\frac{1}{m}-(1-\Theta_m^{\pm})\right)\\
\frac{b^{m+1}}{4m}\left(\frac{1}{m}-1\right)
\end{array}
\right)
\end{align*}

This is equivalent to prove that:

\begin{align*}
0 \neq -\frac{b^{2}}{8}\left(\frac{1}{m}-1\right)\left(\frac{1}{m}-(1-\Theta_m^{\pm})\right)^{2} + \frac{b^{2m+2}}{8m^{2}}\left(\frac{1}{m}-1\right) \Leftrightarrow \left(\Theta_m^{\pm} - (1 - \frac{1+b^m}{m})\right)\left(\Theta_m^{\pm} - (1 - \frac{1-b^m}{m})\right)\neq 0.
\end{align*}
We prove it by contradiction. If $\Theta_m^{\pm} = 1 - \frac{1+b^m}{m}$, we have $b^m=m(1-\Theta_m^{\pm})-1$.
Moreover, by \eqref{determinantm},we have
\[
\Theta_m^{\pm}b^{2m}+b^{m+2}m+\Theta_m^{\pm}(1-m)b^m=0.
\]
Hence, 
\begin{align*}
   & \Theta_m^{\pm}b^{m}+b^{2}m+\Theta_m^{\pm}(1-m)=0\\
   &\Rightarrow \Theta_m^{\pm}(m(1-\Theta_m^{\pm})-1)+b^{2}m+\Theta_m^{\pm}(1-m)=0\\
   &\Rightarrow -m(\Theta_m^{\pm})^2+b^2m=0.
\end{align*}
Since $\Theta_m^{\pm}>0, b>0$, we have
\[
\Theta_m^{\pm}=b,
\]
implying a contradiction since $\Theta_m^+ \Theta_m^- = b^2$ and $\Theta_m^+ \neq \Theta_m^-$. If $\Theta_m^{\pm} = 1 - \frac{1-b^m}{m}$, we can follow the same way to get $\Theta_m^{\pm}=b$ and get a contradiction.
\end{proof}

\begin{proofthm}{teoremaestacionarias}
All the hypotheses of the Crandall-Rabinowitz theorem were checked in Propositions~\ref{onedimensionality}, \ref{codimension_one} and \ref{transv_prop}. Therefore the desired result follows immediately.
\end{proofthm}

\section{Existence of non-radial stationary vortex patches with finite energy}

\label{Section4}
 
In this section, we aim to prove that there exist non-trivial patch solutions with finite kinetic energy, $\frac{1}{2\pi}\int_{\RR^2} |\nabla \left( \omega * \log|\cdot|\right)|^2 dx < \infty$. As mentioned in \eqref{mean_energy}, this property is equivalent to $\int_{\RR}\omega dx = 0$. By Remark \ref{zero_mean_bifurcation1}, we can not use two-layer patches and instead we will consider three-layer patches.
\subsection{Main results for finite energy}\label{main_finite_1}
We consider vortex patches with three layers, that is, $i\in \left\{1,2,3\right\}$ in the setting in Section~\ref{functional_setting}. The total vorticity that we consider is of the form $\omega= \sum_{i=1}^3 \Theta_i 1_{D_i}$, where $D_i$ is determined by $\partial D_i = \left\{ (b_i + R_i(x))(\cos x,\sin x) : x \in \mathbb{T} \right\}$.  We will look for a bifurcation curve from the radial one, $\sum_{i=1}^3 \Theta_i 1_{B_{b_i}}$, where $B_r$ denotes the disk with radius $r$ centered at the origin. We have the following parameters and functional variables:

\begin{itemize}
 \item $b_i \in \RR$: the radii of the different layers of the annuli. We will have $1 =: b_1 > b_2 > b_3$.

\item $\Theta_i\in \RR$: the vorticity at the different layers. We will choose $\Theta_1 := 1$.

\item $R:=(R_1,R_2,R_3) \in (X^{k,m})^3$, for some $3\le k\in \mathbb{N}$ : the functional variables that determine the boundaries. 
\end{itemize}

In the rest of this section, we will fix $m$, $b_2$ and $\Theta_2$ so that for $2 \le m \in \mathbb{N}$,
\begin{align}\label{parameter_3}
0< b_2<\left(\frac{1}{2}\right)^{\frac{1}{2m}} \quad \text{ and }\quad \frac{m({b_2}^2-1)}{(1-{b_2}^{2m})b_2^2}< \Theta_2< \min\bigg\{2\frac{{b_2}^{2m-2}(b_2^2-1)m}{1-b_2^{2m}},\ \frac{-1}{b_2^2}\bigg\}.
\end{align}
 Given $R$, $\Theta_i$, $b_1$ and $b_2$, we choose $b_3$  so that
 \begin{align}\label{Theta b relation}
\int_{\mathbb{R}^2}\omega(x)dx = \sum_{i=1}^3\int_0^{2\pi}\Theta_i(b_i+R_i(x))^2dx=0. 
\end{align}
Since $b_1$, $b_2$ and $\Theta_1$, $\Theta_2$ are fixed constants, \eqref{Theta b relation} implies that $b_3$ is a function of $\Theta_3$ and $R$, more precisely, 
\begin{align}\label{def_b3}
b_3 &= b_3(\Theta_3,R) \nonumber \\
& = \sqrt{-\frac{1}{2\pi\Theta_3} \left( \Theta_1\int_0^{2\pi} (b_1 + R_1(x))^2dx + \Theta_2 \int_0^{2\pi} (b_2 + R_2(x))^2dx + \Theta_3\int_0^{2\pi}R_3(x)^2dx \right)  }
\end{align}
If $\Theta_3,b_3(\Theta_3,R)\ne 0$, then its derivative with respect to $R$ is given by
\begin{align}\label{b_der_R}
Db_3(\Theta_3,R)[h] & := \frac{d}{dt}b_3(\Theta_3,R+th)\bigg|_{t=0} \nonumber\\
& = -\frac{1}{2\pi \Theta_3 b_3(\Theta_3,R)}\left( \Theta_1\int_{\mathbb{T}}R_1(x)h_1(x)dx  +  \Theta_2\int_{\mathbb{T}}R_2(x)h_2(x)dx  + \Theta_3\int_{\mathbb{T}}R_3(x)h_3(x)dx\right),
\end{align}
where we used $\int_\mathbb{T}h_i(x)dx = 0$ for $h_i\in X^{k,m}$.

Note that  for sufficiently small $\|R_i\|_{L^\infty}$ and $|\Theta_3-\Theta^*_{3,m}|$, where $\Theta^*_{3,m}$ is as defined in Lemma~\ref{kernel}, we can choose $b_3=b_3(\Theta_3,R_i)$ so that \eqref{Theta b relation} is compatible with $b_2>b_3>0$ (see Lemma~\ref{kernel}).  Therefore, a 4-tuple $(\Theta_3,R_1,R_2,R_3)=:(\Theta_3,R)$ uniquely determines $\omega= \sum_{i=1}^3 \Theta_i 1_{D_i}$ such that the boundary of the $i$th patch surrounds the $j$th patch if $i < j$ and $\int_{\RR}\omega dx = 0$. In the proof, $\Theta_3$ will play the role of the bifurcation parameter and we will look for a bifurcation from $(\Theta^*_{3,m},0)\in \RR\times (X^{k,m})^3$.


With this setting, the system \eqref{stationary_equation} is equivalent to
\begin{align}\label{stationarR_equation3}
 0=G(\Theta_3, R) :=(G_1(\Theta_3, R), G_2(\Theta_3, R), G_3(\Theta_3, R)),
 \end{align}
 where
  \begin{align}\label{stationarR_equation4}
G_j(\Theta_3, R):=\mathcal{F}_j(b(\Theta_3,R),\Theta_3,R)), \text{ and }b(\Theta_3,R):=(b_1,b_2,b_3(\Theta_3,R)), \quad j=1,2,3.
  \end{align}
Now, we are ready to state the main theorem of this section:

\begin{theorem}
\label{teoremaestacionarias2}
 Let $k\geq 3$ and  $2 \le m\in \mathbb{N}$, $\Theta_1 = b_1= 1$ and $b_2$ and $\Theta_2$ as in \eqref{parameter_3}. Then for some $s_0=s_0(m,k,b_2,\Theta_2) > 0$, there exists a bifurcation curve $[0,s_0) \ni s\mapsto (\Theta_3(s),R(s)) \in \RR \times (X^{k,m})^3$  such that for each $s\in (0,s_0)$, $(\Theta_3(s),R(s))$ is a solution of the equation \eqref{stationarR_equation3} and $(R(s)) \ne 0\in (X^{k,m})^3$. The bifurcation curve emanates from $(\Theta_3(0),R(0)) = (\Theta^*_{3,m},0)$, where $\Theta^*_{3,m}$ is defined in Lemma~\ref{kernel}.
 
\end{theorem}

Theorem~\ref{teoremaestacionarias2} immediately implies the existence of non-radial stationary vortex patches with finite kinetic energy.
\begin{corollary}\label{finite_energy_solution}
Let $2\le m\in \N$ and $k\ge 3$. Then there is an $m$-fold symmetric stationary patch solution of the 2D Euler equation with $H^k$-regular boundary and finite kinetic energy, that is 
\[ \int_{\RR^2} \left|\nabla \left( \omega * \frac{1}{2\pi}\log|\cdot| \right)\right|^2 dx < \infty.
\] 
\end{corollary}
\begin{proof}
By the definition of $b_3$ in \eqref{Theta b relation}, each $\omega=\sum_{i=1}^3 \Theta_i 1_{D_i}$ which is determined by $(\Theta_3(s),R(s))$ for $s\in (0,s_0)$, satisfies $\int_\RR \omega dx = 0$. This is equivalent to $\int_{\RR^2} |\nabla \left( \omega * \frac{1}{2\pi}\log|\cdot| \right)|^2 dx < \infty$, (see \eqref{mean_energy}).
\end{proof}

 The existence of the bifurcation curves will be proved by means of a Nash-Moser iteration scheme (Theorem~\ref{theorem1}). The proof of Theorem~\ref{teoremaestacionarias2} will be accomplished in Subsection~\ref{checking_subsection} by checking the hypotheses of Theorem~\ref{theorem1}.

\subsection{Compactly supported velocity}
In this subsection, we digress briefly to observe an interesting consequence of Theorem~\ref{teoremaestacionarias2}. Thanks to a simple maximum principle lemma, it can be shown that each stationary solution on the bifurcation curves has compactly supported velocity. 

\begin{lemma}(the key Lemma)\label{zero mean} Assume that $\omega\in L^{1}\cap L^{\infty}(\RR^n)$ for $n\ge 2$ is compactly supported and let $\Omega$ be the unbounded connected component of $\text{supp}(\omega)^c$. We additionally assume that $\int_{\RR^n}\omega dx = 0$. Then for $f := \omega * \mathcal{N}$, where
\begin{align*}
\mathcal{N}(x) =
\begin{cases}
\frac{1}{2\pi}\log|x| & \text{ if }n=2 \\
\frac{1}{n(2-n)V_n}|x|^{2-n}, \quad V_n:=\frac{\pi^{\frac{n}{2}}}{\Gamma(\frac{n}{2}+1)} & \text{ if } n>2,
\end{cases}
\end{align*}
it holds that
\begin{align*}
\sup_{x\in \Omega}f(x)=\max_{x\in \partial \Omega}f(x) \quad \text{and} \quad \inf_{x\in \Omega}f(x)=\min_{x\in \partial \Omega}f(x).
\end{align*}
Consequently, if $f$ is constant on $\partial \Omega$, then $f$ is constant in $\Omega$.
\end{lemma}
\begin{rem}  The above lemma does not hold without the assumption $\int_{\RR^n} \omega dx = 0$. For example, $f:=1_{B_1} * \mathcal{N}$ is harmonic in $\Omega:=B_1^c$, while $\max_{\Omega}f$ is unbounded in  $\Omega^c$.
\end{rem}

\begin{proof}
It suffices to prove the maximum part since we can apply the argument to $-f$ for the minimum part. The proof is classical but we present a proof for the sake of completeness.

  Let $M:=\sup_{x\in \Omega}f(x)$. Then $A:=\left\{ x\in \Omega : f(x)=M\right\}$ is relatively closed in  $\Omega$, since $f$ is continuous. Furthermore, since $f$ is harmonic in $\Omega$, the mean value property yields that $A$ is open. Therefore $A$ must be either $\emptyset$ or $\Omega$ since $\Omega$ is connected. If $A=\Omega$, then the result follows trivially. Now,  let us suppose that $A=\emptyset$. Towards a  contradiction, assume that $M > \max_{x\in \partial \Omega}f(x)$. Since $f$ is bounded in $\RR^n$ (this property still holds when $n=2$, thanks to $\int_{\RR^2}\omega dx =0$), therefore the only possible case is that
 \begin{align}\label{maximum}
 \lim_{|x|\to \infty}f(x)=M.
 \end{align}
 Let us consider $\phi(r):=\frac{1}{|\partial B_r|}\int_{|x|=r}f(x)d\sigma(x)$. For any sufficiently large $r$ such that $\Omega^c \subset B_r$, where $B_r$ is the ball centered at the origin with radius r, it follows that
 \begin{align*}
 \frac{d}{dr}\phi(r) & = \frac{d}{dr}\frac{1}{|S_{n-1}|}\int_{|x|=1}f(rx)d\sigma(x) = \frac{1}{|\partial B_r|}\int_{|x|<r} \omega (x)dx=0,
 \end{align*}
 Furthermore, \eqref{maximum} yields that
 \begin{align*}
 \lim_{r\to\infty}\phi(r)=M,
 \end{align*}
 hence, we have $\phi(r)=M$ for all sufficiently large $r$. For such $r>0$, we have
 \begin{align*}
 M=\frac{1}{|\partial B_r|}\int_{|x|=r}f(x)d\sigma\le \max_{x\in \partial B_r}f(x)\le M.
 \end{align*}
 This implies $A$ cannot be empty, which is a contradiction. Hence $M=\max_{x\in \partial \Omega}f(x)$.
\end{proof}

 \begin{corollary}\label{corollary_1}
 Let $2\le m\in \N$ and $k\ge 3$. There exists an m-fold symmetric stationary patch solution of the 2D Euler equation with $H^k$-regular boundary and compactly supported velocity.
 \end{corollary}
 \begin{proof}
 Let $\omega = \sum_{i=1}^3\Theta_i1_{D_i}$ be the vorticity determined by $(\Theta_3(s),R(s))$ for $s\in (0,s_0)$ in Theorem~\ref{teoremaestacionarias2}. From \eqref{evolution_patch} and \eqref{stationary_equation}, it follows that its stream function $f:=\frac{1}{2\pi}\omega * \log|\cdot|$ is constant on the outmost boundary $\partial D_1$. However, Lemma~\ref{zero mean} implies that $\sup_{D_1^c} f = \inf_{D_1^c}f$, therefore, $f$ is constant in $D_1^c$. This proves that $\text{supp}(\nabla^\perp f) \subset D_1$.
 \end{proof}

\begin{rem}\label{smooth_app}
In this paper we focus on patch type solutions $\omega=\sum_{i=1}^n\Theta_i 1_{D_i}$. Lemma~\ref{zero mean} is still applicable for smooth $\omega$ as well as long as the boundary of $\Omega:=\text{supp}(\omega)$ can be approximated by regular level curves of $\omega$ (See Figure~\ref{diagram2}). This is due to the fact that the stream function of the smooth stationary $\omega$ must be constant on each regular level set of $\omega$ (see \cite[Section 1]{GomezSerrano-Park-Shi-Yao:radial-symmetry-stationary-solutions}), hence the stream function has to be constant on each connected component of  $\partial \Omega$. If $\partial \Omega$ has only one connected component, then the velocity vanishes in the unbounded component of $\Omega^c$.
\end{rem}

\begin{figure}[h!]
\begin{center}
\includegraphics[scale=0.9]{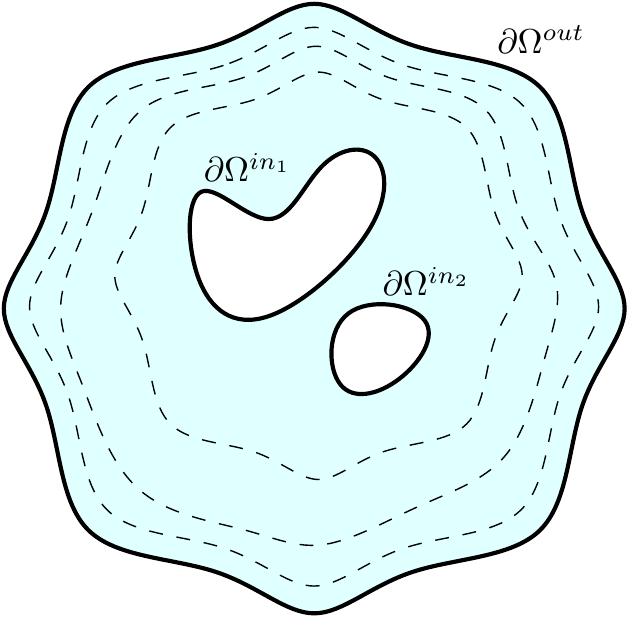}
\caption{Illustration for smooth stationary $\omega$ where $\Omega:=\text{supp}(\omega)$ is colored in blue, whose boundary is not necessarily connected. The dashed lines are regular level sets of $\omega$, which converges to the outermost boundary. In such case, the stream function is constant on the outermost boundary $\partial \Omega^{out}$, thus Lemma~\ref{zero mean} implies that the velocity vanishes in the unbounded part of $\Omega^c$.  \label{diagram2}}
\end{center}
\end{figure}

\color{black}


\subsection{Nash-Moser theorem}
We first prove a bifurcation theorem using the Nash-Moser scheme under some assumptions, which will turn out to be satisfied by our nonlinear functional. We follow the ideas from Berti \cite{Berti:nash-moser-tutorial}.

  Let $2\le m\in \N$ fixed. We denote 
\[ 
X^{k}:=X^{k,m}, \quad  Y^k := Y^{k,m},\quad C^\infty:=C^\infty(\mathbb{T}), \quad   R:=(R_1,R_2,R_3)\in \left( X^{k+1}\right)^3 \cap (C^{\infty})^3, \quad  \partial_{\Theta}:=\partial_{\Theta_3},\]  for simplicity. Furthermore, for a Banach space $X$ and an element $R\in X$, we denote the norm of $R$ by $|R|_{X}:=\rVert R \rVert_{X}$. In addition, we use the notation $A\lesssim_{a,b} B$ if there exists a constant $C=C(a,b)>0$ depending on some variables $a,b$ such that $A\le CB$. We also use $c_0,c_1,....$ to denote universal constants that may vary from line to line.

\begin{theorem}\label{theorem1} Assume that there exists $\Theta_3^* \in \RR$ and  an open neighborhood  $I\times V^{3}$ of $(\Theta^*_3,0)\in \RR\times \left( X^{3}\right)^3$ such that for each $2 \le k \in \mathbb{N}$,  $G : I \times \left(X^{k+1} \right)^3 \mapsto \left(Y^{k}\right)^3$ satisfies the following: For $(\Theta_3,R)\in I\times \left(V^3 \cap (C^\infty)^3\right)$
\begin{itemize}
\item[(a)] (Existence of a curve of trivial solutions) $G(\Theta_3,0) = 0$ for all $\Theta_3 \in I$.

\item[(b)] (Regularity) It holds that
\begin{align}
&|G(\Theta_3,R)|_{(Y^{k})^3} \lesssim_{k} 1 + |R|_{(X^{k+1})^3}, \label{lineargrowth1}\\
&\left|\partial_\Theta G(\Theta_3,R)\right|_{(Y^{k})^3} \lesssim_{k} 1+\left| R \right|_{(X^{k+1})^3}\label{lineargrowth2}\\
&| D^2 G(\Theta_3,R)[h,h]|_{(Y^{k})^3} \lesssim_{k} (1+| R |_{(X^{k+3})^3})|h|_{(X^{k+1})^3}^2\label{D2G}\\
&| \partial_\Theta D G(\Theta_3,R)[h]|_{(Y^{k})^3} \lesssim_{k} (1+| R |_{(X^{k+3})^3})| h |_{(X^{k+1})^3}\label{dtDG}\\
&|\partial_\Theta D^2G(\Theta^*_3, R)[h,h]|_{(Y^{k})^3} \lesssim_{k} (1+| R |_{(X^{k+3})^3})| h |_{(X^{k+1})^3}^2\label{dtDG2}\\
&| \partial_{\Theta\Theta} D G(\Theta_3,R)[h]|_{(Y^{k})^3} \lesssim_{k} (1+| R |_{(X^{k+3})^3})| h |_{(X^{k+1})^3}\label{dttDG}
\end{align}

\item[(c)] (Decomposition of $DG$) $DG(\Theta_3,R)\in \mathcal{L}((X^{k+1})^3;(Y^k)^3)$  has the following decomposition:
\begin{align*}
DG(\Theta_3,R)[h] = a(\Theta_3,R)[h] + A(\Theta_3,R)[h],
\end{align*}
such that $a(\Theta_3,R)\in \mathcal{L}((X^{k+1})^3,(Y^{k})^3)$, $A(\Theta_3,R) \in \mathcal{L}((X^{k+1})^3,Y^{k+1}\times (Y^{k})^2)$. Also, there exists $0 < \eta <1$ such that if $\rVert R \rVert_{(H^{k+2})^3}\le \eta$, then
\begin{align}
&{|a(\Theta_3,R)[h]|_{(Y^{k})^3} \lesssim_{k}  |G(\Theta_3,R)|_{(Y^{k})^3}|h|_{(X^{k+1})^3} }\label{approxinverse1}, \\
&|A(\Theta_3,R)[h]|_{Y^{k+1}\times (Y^{k})^2} \lesssim_{k} (1+| R |_{(X^{k+3})^3}) |h|_{(X^{k+1})^{3}},\label{approxinverse2}
\end{align}
\item[($\tilde{c}$-1)]   $A:\RR\times (X^{k+3})^3\mapsto \mathcal{L}((X^{k+1})^3,Y^{k+1}\times (Y^{k})^2)$ is Lipschitz continuous. That is, if $$(\Theta_3^1,R^1),\ (\Theta_3^2,R^2)\in I\times (V^3\cap (C^\infty)^3),$$ and $\rVert R^1\rVert_{(H^{k+3}(\mathbb{T}))^2}, \rVert R^2 \rVert_{(H^{k+3}(\mathbb{T}))^2} \le 1$,
it holds that 
\begin{align}\label{NM_lip}
|A(\Theta_3^1,R^1)[h] - A(\Theta_3^2,R^2)|_{(Y^{k})^3} \lesssim_k \left( |\Theta_3^1 - \Theta_3^2| + |R^1-R^2|_{(X^{k+3})^3}\right) |h|_{(X^{k+1})^3}.
\end{align}
\item[($\tilde{c}$-2)] (Tame estimates) There exists $0 < \eta < 1$ such that if $\rVert R \rVert_{(H^{k+3}(\mathbb{T}))^3} \le \eta$, and $A(\Theta_3,R)[h] =z$ for some  $z\in (C^\infty)^3$ and $h \in \text{Ker}(A(\Theta^*_3,0))^{\perp}$,  then $h\in (C^\infty)^3$. Furthermore, for any even $\sigma \in \mathbb{N}\cup \left\{ 0 \right\}$,  it holds that
\begin{align}\label{tame1}
\left| h \right|_{(X^{k+1+\sigma})^3} \lesssim_{k,\sigma} (1 + \left| R \right|_{(X^{k+4+\sigma})^3} )\left| z \right|_{Y^{k+1}\times (Y^{k})^{2}} + \left| z \right|_{Y^{k+1+\sigma}\times (Y^{k+\sigma})^{2}}.
\end{align}

\item[(d)] (Fredholm index zero) There exist non-zero vectors $ v$ and $w$ such that $v$ and $w$ are supported on the $m$-th Fourier mode and 
\begin{align*}
\text{Ker}(A(\Theta_{3}^{*},0)) = \text{span}\left\{ v\right\}, \quad \text{Im}(A(\Theta_{3}^{*},0))^{\perp} = \text{span}\left\{w\right\}, 
\end{align*}
\item[(e)] (Transversality) $\partial_{\Theta} A(\Theta_{3}^{*},0)[v] \notin \text{Im}(A(\Theta_{3}^{*},0))$.
\end{itemize}
Then, for any $k_0\ge 2$,  there exist a constant $s_0=s_0(k_0)>0$ and a curve $[0,s_0) \ni s\mapsto (\Theta_3(s),R(s))\in I\times (X^{k_0+1})^3$ such that $G(\Theta_3,R)=0$ and $R(s)\ne 0$ for $s>0$. The curve emanates from $(\Theta_3(0),R(0)) = (\Theta^*_3,0)$.
\end{theorem}
\begin{rem}
Note that the evenness of $\sigma$ for the tame estimate in $(\tilde{c}-2)$ is simply because $X^k$ is a space of even functions and any odd order derivatives of $h\in X^k$ are even. 
\end{rem}
The rest of this section is devoted to prove Theorem~\ref{theorem1}.

Towards the proof, let $k_0 \ge 2$ be fixed. We define the projections $P:(X^{k_0+1})^3 \mapsto \text{Ker}(A(\Theta_3^*, 0))$ and $Q: (Y^{k_0})^3\mapsto \text{Im}(A(\Theta_{3}^{*}, 0))^{\perp}$ by
\begin{align}\label{projections}
PR:=\left( v\cdot R\right) v \quad \text{ and } \quad QR:=(w\cdot R)w,
\end{align}
where $(f\cdot g)$ denotes the usual  $L^2$ inner product. Note that from the assumptions $(d)$ and $(e)$, we make an ansatz that for sufficiently small $s>0$, the bifurcation curve $s\mapsto(\Theta_3(s),R(s))$ can be written as
\[
(\Theta_3(s),R(s)) = \left(\Theta^*_3 + v\cdot \tilde{R}(s), sv+(I-P)\tilde{R}(s) \right),
\]
for some $\tilde{R}(s) \in (X^{k_0+1})^3$ such that $|(I-P)\tilde{R}|_{(X^{k_0+1})^3} = o(s)$. From this ansatz, we define a family of functionals $\tilde{G_s}: (X^{k_0+1})^3 \mapsto  (Y^{k_0})^3$ by
\begin{align}\label{def_tilde_g}
\tilde{G}_s(R) := G(\Theta^*_3 + v\cdot R, sv + (I-P)R),
\end{align}
and look for  $R\in (X^{k_0+1})^3$ such that $\tilde{G}_s(R) = 0$ and $|(I-P)R|_{(X^{k_0+1})^3} = o(s)$ for  sufficiently small $s>0$. This will be achieved by Newton's method, where the first approximate solution is $R=0$. To perform Newton's method, we need to study the linearized operator $D\tilde{G}_s$ at each approximate solution $R \ne 0$: For $h\in (X^{k_0+1})^3$, which can be directly computed from \eqref{def_tilde_g},
\begin{align}
D\tilde{G}_s(R)[h]&  := D_R\tilde{G}_s(R)[h]\nonumber\\
& = \partial_{\Theta}G(\Theta_{3}^{*}+v\cdot R,sv+(I-P)R)(v\cdot h) + DG(\Theta_{3}^{*}+v\cdot R, sv+(I-P)R)[(I-P)h]\nonumber\\
& =: T_s(R)[h] + a(\Theta_{3}^{*}+v\cdot R, sv+(I-P)R)[(I-P)h],\label{dGanda}
\end{align}
where 
\begin{align}\label{T_sdef}
T_s(R)[h] := \partial_{\Theta}G(\Theta_{3}^{*}+v\cdot R,sv+(I-P)R)(v\cdot h) + A(\Theta_{3}^{*}+v\cdot R, sv+(I-P)R)[(I-P)h].
\end{align}
The above decomposition of $D\tilde{G}_s$ into $T_s + a$ follows from the assumption $(c)$. Recall that Newton's method relies on the \quotes{invertibility} of $T_s(R)$ for small $s>0$. However, as we will see in the next lemma, $T_s$ is not fully invertible between $(X^{k_0+1})^3$ and $(Y^{k_0})^3$, because of its loss of derivatives. This is a motivation of adapting a Nash-Moser scheme in our proof.  However, \eqref{approxinverse1} suggests that the inverse of $T_s$ is a good approximate right inverse of $D\tilde{G}_s$. In the next subsection, we will focus on the properties of $T_s(R)$.

\subsubsection{Analysis of ${T}_s$}\label{Analysis_T}
 We will look for a solution $R$ to $\tilde{G}_s (R) = 0$  in $(X^{k_0+1})^3$ for small $s>0$.  In each step of Newton's method (Nash-Moser iteration), we will regularize the approximate solution $R_n$ (see \eqref{definition_of_approx_sol}). The theorem will be achieved by proving that $R_n$ converges in $(X^{k_0+1})^3$. However, we will also obtain boundedness of the sequence $R_n$ in the higher norms, which is necessary because of the extra regularity conditions as in (c), $(\tilde{c}-1)$ and $(\tilde{c}-2)$.  For this reason, we will establish several lemmas assuming that an approximate solution $R$ is more regular then $(X^{k_0+1})^3$, which will turn out to be true at the end of the proof.

\begin{lemma}\label{approxinv}
Let  $0 < \epsilon < 1 $ and $2\le k_0\in \mathbb{N}$ be fixed. There exist positive constants  $s_0(\epsilon,k_0),\  c_0(\epsilon,k_0)$ and $0<\delta(\epsilon,k_0)<1$ such that for each $0< s  < s_0$, the following holds: If 
\begin{align}
&|PR|_{(X^{k_0+2})^3} \leq  s^{\epsilon},\label{assumptionfory1}\\
&|(I-P)R|_{(X^{k_0+2})^3}\leq  s^{1+\epsilon},\label{assumptionfory2}\\
&|R|_{(H^{k_0+4})^3}\leq \delta,\label{assumptionfory3}
\end{align}
then
\begin{itemize}
\item[(A)] For all $t\in[0,1]$, 
\begin{align}\label{containedinV}
(\Theta^*_3 + v\cdot R, sv + (I-P)R)\in I\times V^3,
\end{align}
where $V^{3}$ is as in Theorem \ref{theorem1}. Therefore, 
\begin{align*}
T_s(R)[h] := \partial_{\Theta}G(\Theta_{3}^{*}+v\cdot R,sv+(I-P)R)(v\cdot h) + A(\Theta_{3}^{*}+v\cdot R, sv+(I-P)R)[(I-P)h],
\end{align*}
is well-defined.
\item[(B)] $T_s(R):(X^{k_0+1})^3\mapsto Y^{k_0+1}\times(Y^{k_0})^2$ is an isomorphism and 
\begin{align}
&|T_s(R)[h]|_{Y^{k_0+1}\times(Y^{k_0})^2} \le c_0 |h|_{(X^{k_0+1})^3},\label{Tbound}\\
&|T_s(R)^{-1}[z]|_{(X^{k_0+1})^3} \le \frac{c_0}{s}|z|_{Y^{k_0+1}\times(Y^{k_0})^2}\label{Tinvbound0},\\
&\left|\left(D\tilde{G}_s(R)\circ T_s(R)^{-1} - I\right)[z]\right|_{(Y^{k_0})^3} \le c_0  |\tilde{G}_s(R)|_{(Y^{k_0})^3}|T_s(R)^{-1}[z]|_{(X^{k_0+1})^3},\label{Approximate_inv}
\end{align}
for some $c_0=c_0(\epsilon,k_0)>0$.

If $R\in C^{\infty}$, then for any even $\sigma\in \mathbb{N}\cup \left\{ 0 \right\}$ and $z\in X^{k_0+1+\sigma}\times (X^{k_0+\sigma})^{2}$, we have $T_s(R)^{-1}[z] \in \left( X^{k_0+1+\sigma}\right)^3$. Also, we have that 
\begin{align}\label{highernorm_inversion}
\left| T_s(R)^{-1}[z] \right|_{\left( X^{k_0+1+\sigma}\right)^3} \le \frac{c_0}{s} \left( (1 +  |R|_{(X^{k_0+4+\sigma})^3})|z|_{Y^{k_0+1}\times (Y^{k_0})^{2}} + |z|_{Y^{k_0+1+\sigma}\times (Y^{k_0+\sigma})^{2}}\right),
\end{align}
where  $c_0$ may depend on not only $\epsilon,k_0$ but also $\sigma$.
\item[(C)]Furthermore, we can choose $c_0$ so that when $R=0$,  the following hold:
\begin{align}
|P(T_s(0)^{-1}[z])|_{(X^{k_0+1})^{3}} \le \frac{c_0}{s}|z|_{Y^{k_0+1}\times(Y^{k_0})^{2}},\label{Tinvbound1}\\
|(I-P)(T_s(0)^{-1}[z])|_{(X^{k_0+1})^{3}} \le c_0 |z|_{Y^{k_0+1}\times(Y^{k_0})^{2}}.\label{Tinvbound2}
\end{align}
\end{itemize}\end{lemma}

We  will frequently use 
\begin{align}\label{crudebound}
|T_s(R)^{-1}z|_{\left(X^{k_0+1}\right)^3}  \le \frac{c_0}{s}|z|_{(Y^{k_0+1})^3},
\end{align}
 which is more crude than \eqref{Tinvbound0}.
\begin{proof}
Let us fix $\epsilon>0$ and $k_0\ge2$. We will show that if \eqref{assumptionfory1}-\eqref{assumptionfory3} hold for sufficiently small $s>0$ and  for some small $\delta>0$ depending on $\epsilon$, then  (A),(B) and (C) hold for some $c_0>0$.

\textbf{Proof of (A).} To see (A), notice that \eqref{containedinV} is trivial. From the assumptions (b) and (c) in Theorem~\ref{theorem1}, the linear operator $T_s(R)$ is well-defined.
 
\textbf{Proof of (B).} The estimate \eqref{Tbound}, follows from \eqref{T_sdef}, \eqref{approxinverse2}, \eqref{lineargrowth2} and
\[
| sv + (I-P)R |_{\left(H^{k_0+3}\right)^3} \lesssim  1 + \delta \lesssim 1,
\]
where the last inequality follows from \eqref{assumptionfory3}. Before proving \eqref{Tinvbound0}, we first prove that
\begin{align}\label{novanishingQ}
|Q\partial_\Theta G(\Theta_{3}^{*}+v\cdot R, sv+(I-P)R) |_{X^{k_0+1}\times (X^{k_0})^2}\ge cs,
\end{align}
for some $c=c(\epsilon)>0$. The above inequality is a consequence of the transversality condition $(e)$ in Theorem~\ref{theorem1}. In what follows, $c,c_1,c_2,\ldots$ denote positive constants that may change from line to line and depend only on $\epsilon$ but not on $s$.

By the regularity assumptions  of $G$ in (b) in Theorem~\ref{theorem1}, we estimate the quantity in \eqref{novanishingQ} using a Taylor expansion up to linear order: For a fixed $R$, we set $f(s,y):=Q\partial_\Theta G(\Theta_{3}^{*}+v\cdot R, sv+(I-P)y)$, so that
\begin{align}\label{fisQG}
f(s,R) = Q\partial_\Theta G(\Theta_{3}^{*}+v\cdot R, sv+(I-P)R).
\end{align}
 Using the fundamental theorem of calculus, we can write
\begin{align*}
f(s,R) &:=  \int_0^{1}\frac{d}{dt}\left( f(s,tR)\right) dt + f(s,0)\\
 & = \int_0^{1}\frac{d}{dt}\left( f(s,tR)\right) dt  + f(0,0) + \int_0^{s}\frac{d}{dt}f(t,0)dt.
\end{align*}
In terms of $G$, \eqref{fisQG} and the above equality give us that
\begin{align}\label{QGestimate}
|Q\partial_\Theta G(\Theta_{3}^{*}&+v\cdot R, sv+(I-P)R) |_{Y^{k_0+1}\times (Y^{k_0})^2} =|f(s,R)|_{Y^{k_0+1}\times (Y^{k_0})^2}\nonumber\\
& \ge -{\sup_{0\leq t\leq 1}\|Q\partial_\Theta DG(\Theta_{3}^{*}+v\cdot R, sv+t((I-P)R))\|_{\mathcal{L}((X^{k_0+2})^3,(Y^{k_0+1})^3)}|(I-P)R|_{\left(X^{k_0+2}\right)^3}} \nonumber\\
&\ - {|Q\partial_\Theta G(\Theta_{3}^{*}+v\cdot R,0)|_{Y^{k_0+1}\times (Y^{k_0})^2}}\nonumber\\
& \ +{\bigg|\int_{0}^{s} Q\partial_\Theta DG(\Theta_{3}^{*}+v\cdot R, tv)[v] dt \bigg|_{Y^{k_0+1}\times (Y^{k_0})^2}}\nonumber\\
& =: -L_1 - L_2 + L_3,
\end{align}
where the first term after the inequality follows from b) in Theorem~\ref{theorem1} and 
\begin{align*}
\int_0^1 \left| \frac{d}{dt}\left( f(s,tR)\right) \right|_{Y^{k_0+1}\times (Y^{k_0})^2} dt  & \le \int_0^1 \left| \frac{d}{dt}\left( f(s,tR)\right)\right|_{(Y^{k_0+1})^3} dt \\
& \le \sup_{0\le t \le 1}\left| D_Rf(s,tR)[R] \right|_{(Y^{k_0+1})^3},
\end{align*}
 and the rest follows simply from the triangular inequality. Then the regularity assumption (b) in Theorem~\ref{theorem1} yields that (recall $\rVert R \rVert_{(X^{k_0+4})^3}<\delta < 1$,)
\begin{align}\label{l1}
L_1 \le c_1 |(I-P)R|_{\left(X^{k_0+2}\right)^3} \le c_1 s^{1+\ep},
\end{align}
where the last inequality follows from \eqref{assumptionfory2}. For $L_2$, we use (a) in Theorem~\ref{theorem1} to obtain 
\begin{align}\label{l2}
L_2 = 0.
\end{align}
To estimate $L_3$, we compute
\begin{align}\label{L_3estimate112}
L_3  &\ge \bigg| \int_{0}^{s}Q\partial_{\Theta} DG(\Theta_{3}^{*},tv)[v]dt\bigg|_{Y^{k_0+1}\times (Y^{k_0})^2} - \int_{0}^{s}\sup_{u\in(0,|v\cdot R|)}\bigg|Q\partial_{\Theta\Theta}DG(\Theta_{3}^{*}+u,tv)[v]\bigg|_{Y^{k_0+1}\times (Y^{k_0})^2}|v\cdot R|dt.
\end{align}
For the first integral, we can write
\begin{align*}
Q\partial_\Theta DG(\Theta^*_3,tv)[v] &= Q\partial_\Theta DG(\Theta^*_3,0)[v] + \int_0^{t}Q\partial_\Theta D^2G(\Theta^*_3, uv)[v,v]du\\
& = QA(\Theta^*_3,0)[v] + \int_0^{t}Q\partial_\Theta D^2G(\Theta^*_3, uv)[v,v]du,
\end{align*}
where the last equality follows from (a) and (c) in Theorem~\ref{theorem1}, which shows $a(\Theta^*_3,0)=0$. For the integral term,
using \eqref{dtDG2}, we have 
\[
\left| Q\partial_\Theta D^2 G(\Theta^*_3,uv)[v,v]\right|_{Y^{k_0+1}\times (Y^{k_0})^2} \lesssim \left| Q\partial_\Theta D^2 G(\Theta^*_3,uv)[v,v]\right|_{(Y^{k_0+1})^3} \lesssim (1+\rVert R \rVert_{(X^{k_0+4})^3})\rVert v \rVert_{(X^{k_0+2})^3} \le c,
\] where the last inequality follows from \eqref{assumptionfory3}. Therefore
\begin{align*}
\bigg| \int_{0}^{s}Q\partial_{\Theta} DG(\Theta_{3}^{*},tv)[v]dt\bigg|_{X^{k_0+1}\times (X^{k_0})^2} &\ge s \left| QA(\Theta^*_3,0)[v] \right|_{X^{k_0+1}\times (X^{k_0})^2} - \int_0^{s}\int_0^t c dudt \\
& \ge s \left| Q\partial_\Theta DG(\Theta^*_3,0)[v] \right|_{X^{k_0+1}\times (X^{k_0})^2} - c_2 s^2 \\
& \ge c_3 s - c_2 s^2,
\end{align*}
where the last inequality follows from the transversality (e).  For the second integral in \eqref{L_3estimate112}, we have that for sufficiently small $s>0$,
\[
\sup_{u\in (0,|v\cdot R|),t\in (0,s)}\left| \partial_{\Theta\Theta}DG(\Theta^*_3 + u,tv)[v] \right|_{Y^{k_0+1}\times (Y^{k_0})^2} \lesssim \sup_{t\in (0,s)}(1+|  tv |_{(X^{k_0+4})^3})| v |_{(X^{k_0+2})^3}\le c_1,
\]
for sufficiently small $s$, which follows from \eqref{dttDG}.  
Hence using \eqref{assumptionfory1}, we obtain from \eqref{L_3estimate112} that
\begin{align}\label{l3}
L_3 \ge  c_3 s - c_2s^2 - c_1 s^{1+\epsilon}
\end{align}
Thus the claim \eqref{novanishingQ} follows from \eqref{QGestimate} \eqref{l1}, \eqref{l2} and \eqref{l3} for small $s$, depending on $\epsilon$. 

 Towards the proof of \eqref{Tinvbound0}, we will consider how to  invert $A(\Theta_3^*+v\cdot R,sv+(I-P)R)[(I-P)]$.  Since the dimension of $\text{Im}(DG(\Theta_3^*,0)^\perp)$ is one and  $w$ is the basis of $\text{Im}(DG(\Theta_{3}^{*},0))^{\perp}$, the above claim \eqref{novanishingQ} proves that
\begin{align}\label{novanishingQ2}
\left| w\cdot \partial_\Theta G(\Theta_{3}^{*}+v\cdot R, sv+(I-P)R) \right| > c s.
\end{align}
To simplify the notations,
we denote
\begin{align} 
&\tilde{A}[h] := A(\Theta_{3}^{*}+v\cdot R, sv+(I-P)R)[(I-P)h],\label{Atilde_def1}\\
&\partial_\Theta G:=\partial_{\Theta}G(\Theta^*+v\cdot R,sv+(I-P)R).\nonumber
\end{align}
We denote by $Q_1$ the projection from $Y^{k_0+1}\times \left( Y^{k_0}\right)^2$ into $\text{Im}(\tilde{A})^\perp$. Since the norm of $R$ is small, we expect that the functional structure of $\tilde{A}$ should be similar to the structure of $A(\Theta^*_3,0)$. Indeed, by continuity hypotheses \eqref{NM_lip}, \eqref{assumptionfory3} and Lemma~\ref{functional_stability}, where we can think of $\tilde{A}$ as a linear map from $\text{Ker}(A(\Theta^*_3,0))$, we can choose $\delta$ small enough so that if \eqref{assumptionfory3} holds, then there exists $ 0\ne w_1\in Y^{k_0+1}\times (Y^{k_0})^2$ such that
\begin{align}
&\text{Im}(\tilde{A})^{\perp}=\text{span}\left\{w_1\right\} \label{Atilde1}\\
&|w - w_1|_{Y^{k_0+1}\times (Y^{k_0})^2} \le c\delta, \label{Atilde_2} \\
&|\tilde{A}|_{\mathcal{L}(\text{Ker}\left(A(\Theta_{3}^{*},0)\right)^{\perp}, \text{Im}(\tilde{A}))} \le c, \label{Atilde_3} \\
&|\tilde{A}^{-1}|_{\mathcal{L}( \text{Im}(\tilde{A}),\text{Ker}\left(A(\Theta_{3}^{*},0)\right)^{\perp})} \le c.\label{Atilde_4}
\end{align}
 Now, we claim that we can further restrict $\delta$ if necessary so that
\begin{align}\label{claimforw1}
\left| w_1\cdot \partial_\Theta G \right| > c s.
\end{align}
In fact, note that \eqref{assumptionfory1}-\eqref{assumptionfory2} imply that $R\in (X^{k_0+2})^3$, thus by \eqref{lineargrowth2}, we have $\partial_\Theta G\in  (Y^{k_0+1})^3\subset Y^{k_0+1}\times (Y^{k_0})^2$. Also we have
\begin{align*}
\partial_\Theta G &= \int_0^{1} \partial_t\left( \partial_\Theta G(\Theta^*_3+v\cdot R, t(sv + (I-P)R))\right)dt \\
& = \int_0^{1}\partial_\Theta DG(\Theta^*_3+v\cdot R,t(sv+(I-P)R))[sv+(I-P)R]dt,
\end{align*}
thus 
\begin{align}\label{claimforw2}
\left| \partial_\Theta G\right|_{Y^{k_0+1}\times (Y^{k_0})^2} \le \left| \partial_\Theta G\right|_{(Y^{k_0+1})^3} \le c |sv+(I-P)R|_{(X^{k_0+2})^2} \le cs,
\end{align}
where the second inequality follows from \eqref{lineargrowth2} and the last inequality follows from \eqref{assumptionfory2}. Then it follows that (recall that $f\cdot g$ is the dot product in $L^2$ space),
  \begin{align*}
  |\partial_\Theta G\cdot w_1|  & \geq |\partial_{\Theta}G\cdot w| -| \partial_{\Theta}G \cdot (w_1 - w)| \\
  &\ge c_ 1s - |\partial_{\Theta}G\cdot (w_1-w)|\\
  &\ge c_1s -cs|w-w_1|_{X^{k_0+1}\times (X^{k_0})^2}\\
  &\ge c_1s - c\delta s,
  \end{align*}
  where we used \eqref{novanishingQ2}, \eqref{claimforw2} and \eqref{Atilde_2} for the second, third and the fourth inequality, respectively. Hence, assuming $\delta$ is small, we have \eqref{claimforw1}.
  
 To prove \eqref{Tinvbound0}, pick an arbitrary $z\in Y^{k_0+1}\times(Y^{k_0})^{2}$. There exists a unique $\eta\in \text{Ker}\left(A(\Theta_{3}^{*},0)\right)$ and a unique $h\in \text{Ker}\left(A(\Theta_{3}^{*},0)\right)^{\perp}$ such that
\begin{align}
 & (v\cdot \eta) Q_1\partial_\Theta G = Q_1z \label{uniqueeta}\\
 &(v\cdot \eta)(I-Q_1)\partial_\Theta G  + \tilde{A}[h] = (I-Q_1)z.\label{uniqueh}
 \end{align}
 In fact, there exists a unique $\eta$ in \eqref{uniqueeta} thanks to \eqref{claimforw1} and that $Q_1$ is the projection to a one-dimensional space spanned by $w_1$. Once $\eta$ is fixed, the existence and the uniqueness of $h$ in \eqref{uniqueh} follows from \eqref{Atilde_4}.
   Once $\eta$ and $h$ are determined,  it is clear that 
 \begin{align}\label{tsinverse}
  T_s(R)[h+\eta] = (v\cdot (h+\eta))\partial_\Theta G  + \tilde{A}[\eta+h] = (v\cdot \eta)\partial_\Theta G + \tilde{A}[h] = z,
  \end{align}
  where we used $v\cdot h = 0$ and $\tilde{A}[\eta] = \tilde{A}[(I-P)\eta] = 0$, which follows from the definition of $\tilde{A}$ in \eqref{Atilde_def1}.
  Therefore $T(R)^{-1}z = h+\eta$. Furthermore we have from \eqref{uniqueeta} and \eqref{uniqueh} that 
  \begin{align}
 & |\eta|_{(X^{k_0+1})^3} \le c \frac{|z\cdot w_1|}{|\partial_{\Theta}G \cdot w_1|}\label{bound1}\\
&  |h|_{(X^{k_0+1})^3} \le |\tilde{A}^{-1}(I-Q_1)\partial_{\Theta}G (v\cdot \eta)|_{{(X^{k_0+1})^3}} +| \tilde{A}^{-1}(I-Q_1)z|_{(X^{k_0+1})^3}\label{bound2}.
  \end{align}
  
  
Using \eqref{claimforw1} and \eqref{Atilde_4}, we obtain
  \begin{align}
  &|\eta|_{(X^{k_0+1})^3} \le c  \frac{|z|_{ Y^{k_0+1}\times(Y^{k_0})^{2}}}{s},\label{bound_3}\\
  &|h|_{(X^{k_0+1})^3} \le c |\partial_{\Theta}G|_{Y^{k_0+1}\times (Y^{k_0})^2}\frac{|z|_{ Y^{k_0+1}\times(Y^{k_0})^{2}}}{s} +c |z|_{ Y^{k_0+1}\times(Y^{k_0})^{2}} \le c \frac{|z|_{ Y^{k_0+1}\times(Y^{k_0})^{2}}}{s},\label{bound3}
  \end{align}
  where we used \eqref{claimforw2}. This proves \eqref{Tinvbound0}.
  
    To show \eqref{Approximate_inv}, we compute 
  \begin{align*}
\left(D\tilde{G}_s(R)\circ T_s(R)^{-1}-I\right)[z] =  a(\Theta_3^{*}+v\cdot R,sv+(I-P)R)\circ T(R)^{-1}[z],
\end{align*}
which follows from \eqref{dGanda}. This implies
\begin{align*}
|\left(D\tilde{G}_s(R)\circ T_s(R)^{-1}-I\right)[z]|_{\left(X^{k_0}\right)^3} & \le c|a(\Theta_3^{*}-v\cdot R,sv+(I-P)R)\circ T(R)^{-1}[z]|_{\left(X^{k_0}\right)^3}\\
&\le c |G(\Theta_3^{*}+v\cdot R, sv+(I-P)R)|_{\left(Y^{k_0}\right)^3}|T_s(R)^{-1}[z]|_{\left(X^{k_0+1}\right)^3}\\
& \le c |\tilde{G}_s(R)|_{\left(Y^{k_0}\right)^3}|T_s(R)^{-1}[z]|_{\left(X^{k_0+1}\right)^3},
\end{align*}
where we used \eqref{approxinverse1} to get the second inequality and used the definition of $\tilde{G}_s(R)$ to obtain the last inequality.  

In order to prove \eqref{highernorm_inversion}, we improve on the estimates in \eqref{bound1} and \eqref{bound2}. For \eqref{bound1}, recall that $\eta$ is in a one-dimensional space, span$\left\{ v \right\}$ and $v$ is supported on the $m$-th Fourier mode. Thus, 
\begin{align}\label{higher_1}
|\eta|_{(X^{k_0+1 + \sigma})^3} \le c |\eta|_{(X^{k_0+1})^3} \le c  \frac{|z|_{ Y^{k_0+1}\times(Y^{k_0})^{2}}}{s},
\end{align}
where $c$ may depend on $\sigma$.  
Furthermore, $(\tilde{c}$-2) in Theorem~\ref{theorem1} and \eqref{uniqueh} give us that
\begin{align}
|h|_{(X^{k_0+1+\sigma})^3} & \le c (1+  \left| R \right|_{(X^{k_0+4+\sigma})^3}) \left( \left| (v\cdot \eta)\partial_\Theta G \right|_{Y^{k_0+1}\times (Y^{k_0})^2} + \left| z \right|_{Y^{k_0+1}\times (Y^{k_0})^2} \right) \nonumber \\
 & \ +  c\left( \left| (v\cdot \eta) \partial_\Theta G\right|_{Y^{k_0+1+\sigma}\times (Y^{k_0+\sigma})^2} + \left| z \right|_{Y^{k_0+1+\sigma}\times (Y^{k_0+\sigma})^2} \right). \label{higher_estimate}
\end{align}
 Using \eqref{claimforw2} and \eqref{bound_3}, we have
 \begin{align*}
 \left| (v\cdot \eta)\partial_\Theta G \right|_{Y^{k_0+1}\times (Y^{k_0})^2} \le  \left| \partial_\Theta G \right|_{{Y^{k_0+1}\times (Y^{k_0})^2}}|v\cdot \eta| \le c\left| z \right|_{Y^{k_0+1}\times (Y^{k_0})^2}.
 \end{align*}
For the higher norm of $\partial_\Theta G$, we use \eqref{lineargrowth2} and and \eqref{bound_3} and obtain
\begin{align*}
\left |(v\cdot \eta)\partial_\Theta G\right|_{Y^{k_0+1+\sigma}\times (Y^{k_0+\sigma})^2} &\le c \left| \partial_\Theta G \right|_{(Y^{k_0+1+\sigma})^3}|v\cdot \eta| \\
& \le \frac{c}{s}\left( 1 + \left| R \right|_{(X^{k_0+2+\sigma})^3}\right)\left| z \right|_{Y^{k_0+1}\times (Y^{k_0})^2}.
\end{align*}
Therefore, \eqref{higher_estimate} gives us
\begin{align*}
|h|_{(X^{k_0+1+\sigma})^3}  \le \frac{c}{s}\left( (1 + \left| R \right|_{(X^{k_0+4+\sigma})^3} )\left| z \right|_{Y^{k_0+1}\times (Y^{k_0})^2}  + \left| z \right|_{Y^{k_0+1+\sigma}\times (Y^{k_0+\sigma})^2 }  \right).
\end{align*}
Combining with \eqref{higher_1}, we obtain \eqref{highernorm_inversion}.

 \textbf{Proof of (C).} Finally, if $R=0$, then we have $|\partial_\Theta G|_{Y^{k_0+1}\times (Y^{k_0})^2} = |\partial_\Theta G(\Theta_3,sv)|_{{Y^{k_0+1}\times (Y^{k_0})^2}}\le c s$. Therefore we can improve \eqref{bound3} and obtain
  \begin{align*}
  |h|_{(X^{k_0+1})^3} \le c |z|_{ Y^{k_0+1}\times(Y^{k_0})^{2}},
  \end{align*}
  which implies \eqref{Tinvbound2}. \eqref{Tinvbound0} follows trivially from \eqref{Tinvbound1}. This completes the proof.
\end{proof}

\begin{lemma}\label{conditionsforH}
Let  $\tilde{G}_s$ be as in \eqref{def_tilde_g}. Then there exists an open set $V^3$ near $0\in (X^3)^3$ such that for all $R,\tilde{R}\in V^3$ and $k\ge 2$,   
\begin{itemize}
\item[(a')] (Initial value) $|\tilde{G}_s(0)|_{(X^{k})^3} \lesssim_k s^2.$
\item[(b')] (Taylor estimate) If $\rVert R \rVert_{(X^{k+3})^3}, \rVert \tilde{R} \rVert_{(X^{k+3})^3} \le \eta$ for some $0 < \eta=\eta(k) < 1$, then 
\begin{align*}
\begin{cases}
|D\tilde{G}_s(R)[h]|_{(Y^{k})^3} \lesssim_{k} |h|_{(X^{k+1})^3}\\
|\tilde{G}_s(\tilde{R})-\tilde{G}_s(R)-D\tilde{G}_s(R)[\tilde{R}-R]|_{(Y^{k})^3} \lesssim_{k} |\tilde{R}-R|_{(X^{k+1})^3}^2\\
|\tilde{G}_s(\tilde{R})-\tilde{G}_s(R)-D\tilde{G}_s(R)[\tilde{R}-R]-\frac{1}{2}D^2\tilde{G}_s[\tilde{R}-R,\tilde{R}-R]|_{(Y^{k})^3} \lesssim_{k} |\tilde{R}-R|_{(X^{k+1})^3}^3. 
\end{cases}
\end{align*}

\end{itemize}
\end{lemma}
\begin{proof}
(a') follows from the fact that $v \in \text{Ker}(DG(\Theta_3^{*},0))$ and $G(\Theta_3^*,0)=0$. (b') is due to (b)  in Theorem~\ref{theorem1}.
\end{proof}

\subsubsection{Nash-Moser iteration}
For $\beta,N>0$ and $2\le k\in \mathbb{N}$, we consider a regularizing operator $S(N): (X^{k})^3\mapsto (C^{\infty})^3$ such that
\begin{align}\label{regularizing1}
\begin{cases}
|S(N)R|_{(X^{k+\beta})^3} \lesssim_{k,\beta} N^{\beta} |R|_{(X^{k})^3} & \forall \ R\in (H^{k})^3\\
|(I-S(N))R|_{(H^{k})^3} \lesssim_{k,\beta} N^{-\beta} |R|_{(X^{k+\beta})^3}  & \forall \ R\in (H^{k+\beta})^3.
\end{cases}
\end{align}
Note that we can choose $S$ so that $PS(N) = S(N)P$, since $v$ is supported on the $m$-th Fourier mode (see (d) in Theorem~\ref{theorem1}).

 For  $N_n>1$ and even integer $\beta>0$, which will be chosen later, we set
\begin{align}\label{definition_of_approx_sol}
\begin{cases}
R_{n+1} = R_n - S(N_n)T_s(R_n)^{-1}[\tilde{G}_s(R_n)] & \text{ for }n\ge0\\
R_0 = 0,
\end{cases}
\end{align}
 and
 \begin{align}\label{def_of_sequences}
 \begin{cases}
 a_n:=|R_{n+1}-R_n|_{(X^{k_0+1})^3}\\
 a'_n:=|R_{n+1}-R_n|_{(X^{k_0+2})^3}\\
a''_n=|R_{n+1}-R_n|_{(X^{k_0+4})^3}\\
 b_n:=|\tilde{G}_s(R_n)|_{(Y^{k_0})^3}\\
 C_n:=|T_s(R_n)^{-1}[\tilde{G}_s(R_n)]|_{(X^{k_0+1+\beta})^3}.
  \end{cases}
\end{align}

Note that our goal is to show that $\sum_{n=0}^\infty a_n < \infty$, which implies that $R_n$ converges in $(X^{k_0+1})^3$. In order to prove the assumptions in \eqref{assumptionfory1}-\eqref{assumptionfory3}, we will need the boundedness of $a_n'$ and $a_n''$ as well.
  
 \begin{lemma}\label{iteration1} Let $\epsilon>0$ and $k_0\ge2$ fixed and let $s_0$ and $\delta$ be defined as in Lemma~\ref{approxinv}. We also assume that $\delta$ is even smaller if necessary, so that $\delta < \eta(k_0)$ where $\eta(k_0)$ is as in (b') in Lemma~\ref{conditionsforH}.   For any $n\ge 0$ such that  $R_n$ satisfies \eqref{assumptionfory1}, \eqref{assumptionfory2}  and \eqref{assumptionfory3}  for some $s \in (0,s_0)$, then
  we have
 \begin{align}
&a_{0}\lesssim_{\ep,k_0}  s, \label{recursion3}\\
& a_{n}\lesssim_{\ep,k_0} \frac{1}{s}N_nb_{n} &\text{ for }n\ge 1, \label{recursion1}\\
&b_{0} \lesssim_{\ep,k_0,\beta} s^2 \label{recursion221}\\
&b_{1} \lesssim_{\ep,k_0,\beta} (N_0^3+1)s^{3} + N_0^{-\beta}C_{0}. \label{recursion4}\\
&b_{n+1} \lesssim_{\ep,k_0,\beta} a_{n}^2 + \frac{1}{s}N_{n}b_{n}^2 + N_{n}^{-\beta}(1+b_{n})C_{n} & \text{ for }n\ge 1,\label{recursion2}\\
&a'_n \lesssim_{\ep,k_0} \frac{1}{s}N_n^2 b_n & \text{ for }n\ge 0, \label{recursion5}\\
&a''_n\lesssim_{\ep,k_0}\frac{1}{s}N_n^4 b_n & \text{ for }n\ge 0. \label{recursion6}
 \end{align}
 Furthermore, we have
 \begin{align}
 |PR_1|_{(X^{k_0+2})^3}\lesssim_{\ep,k_0} N_0 s \label{firststep1}\\
 |(I-P)R_1|_{(X^{k_0+2})^3} \lesssim_{\ep,k_0}  N_0 s^2. \label{firststep2}
 \end{align}
 \end{lemma}
 
\begin{proof}
\textbf{Proof of \eqref{recursion3}, \eqref{recursion1}, \eqref{firststep1} and \eqref{firststep2}.} We first claim that
\begin{align}
&|PR_1|_{(X^{k_0+1})^3} \lesssim s\label{claim1}\\
&|(I-P)R_1|_{(X^{k_0+1})^3} \lesssim s^2, \label{claim2}
\end{align}
Let us prove \eqref{claim1} first. Thanks to (a') in Lemma~\ref{conditionsforH} and \eqref{Tinvbound1}, we have
\begin{align*}
|PR_1|_{(X^{k_0+1})^3} = |S(N_0)PT_s(0)^{-1}[\tilde{G}_s(0)]|_{(X^{k_0+1})^3} \lesssim  |PT_s(0)^{-1}[\tilde{G}_s(0)]|_{(X^{k_0+1})^3} \lesssim \frac{1}{s}|\tilde{G}_s(0)|_{(Y^{k_0}+1)^3} \lesssim s,
\end{align*}
where we used that $S(N_0)$ and $P$ commute to obtain the first equality. To prove \eqref{claim2}, we use \eqref{Tinvbound2} and compute
\begin{align*}
|(I-P)R_1|_{(X^{k_0+1})^3}& \lesssim |S(N_0)(I-P)T_s(0)^{-1}[\tilde{G}_s(0)]|_{(X^{k_0+1})^3} \\
& \lesssim  |(I-P)T_s(0)^{-1}[\tilde{G}_s(0)]|_{(X^{k_0+1})^3} \\
& \lesssim |\tilde{G}_s(0)|_{(Y^{k_0+1})^3} \\
&\lesssim s^2,
\end{align*}
which proves \eqref{claim2}.  With the claim,  \eqref{recursion3} follows immediately. \eqref{recursion1}   \eqref{firststep1} and \eqref{firststep2} can be proved in exactly the same way as above using \eqref{regularizing1}.

\textbf{Proof of \eqref{recursion221}, \eqref{recursion4} and \eqref{recursion2}.}
Note that Lemma~\ref{conditionsforH} immediately implies \eqref{recursion221}.  Now, let us prove \eqref{recursion4} and \eqref{recursion2}. We have
\begin{align}\label{bnestimate1}
b_{n+1}:=|\tilde{G}_{s}(R_{n+1})|_{(Y^{k_0})^3} & \lesssim {|\tilde{G}_s(R_{n+1})-\tilde{G}_s(R_{n})-D\tilde{G}_s(R_{n})[R_{n+1}-R_{n}]|_{(Y^{k_0})^3}} \nonumber\\
& \  +{|\tilde{G}_s(R_{n})+D\tilde{G}_s(R_{n})[R_{n+1}-R_{n}]|_{(Y^{k_0})^3}}\nonumber\\
& =: J_1 + J_2,
\end{align}
We claim that
\begin{align}\label{claim3}
\begin{cases}
J_1 \lesssim  s^3 & \text{ if }n=0 \\
J_1 \lesssim a_{n}^2 =|R_{n+1}-R_{n}|_{(X^{k_0+1})^3}^2 & \text{ if }n\ge 1.
\end{cases}
\end{align} 
If $n=0$, we have
\begin{align*}
J_1 & \lesssim \bigg|\tilde{G}_{s}({R_1})-\tilde{G}_s(0)-D\tilde{G}_s(0)[R_1]-\frac{1}{2}D^2\tilde{G}_s(0)\left[R_1,R_1\right]\bigg|_{(Y^{k_0})^3} + \frac{1}{2}|D^2\tilde{G}_s(0)[R_1,R_1]|_{(Y^{k_0})^3}\\
&\lesssim |R_1|^3_{(X^{k_0+1})^3} + \frac{1}{2}|D^2\tilde{G}_s(0)[R_1,R_1]|_{(Y^{k_0})^3} \\
&\lesssim s^3 + \frac{1}{2}|D^2\tilde{G}_s(0)[R_1,R_1]|_{(Y^{k_0})^3},
\end{align*}
where we used (b') in Lemma~\ref{conditionsforH} in the second inequality and  \eqref{claim1} and \eqref{claim2} in the last inequality.
To estimate $D^2\tilde{G}$, we recall the definition of $\tilde{G}_s$ and compute
\begin{align*}
D^2\tilde{G}_s(0)[R_1,R_1] & = \left(\frac{d}{dt}\right)^2 G(\Theta_3^*+tv\cdot R_1, sv+(I-P)tR_1)\bigg|_{t=0}  \\
& = \partial_{\Theta \Theta}G(\Theta_3^*,sv)(v\cdot R_1)^2 + 2(v\cdot R_1)\partial_{\Theta}DG(\Theta_3^*,sv)[(I-P)[R_1] ]\\
& \  + D^2G(\Theta_3^*,sv)\left[ (I-P)R_1,(I-P)R_1\right].
\end{align*}
Since $\partial_{\Theta \Theta}G(\Theta_3^*,0) = 0$, therefore
\[
\left| \partial_{\Theta \Theta}G(\Theta^*_3,sv) \right|_{(Y^{k_0})^3} = \left| \int_0^{s}\frac{d}{dt}\left( \partial_{\Theta\Theta} G(\Theta^*_3,tv)\right)dt\right|_{(Y^{k_0})^3} \lesssim s,
\]
which implies $|\partial_{\Theta \Theta}G(\Theta_3^*,sv)(v\cdot R_1)^2|_{(Y^{k_0})^3} \lesssim s^3$, since $|v\cdot R_1| \lesssim |PR_1|_{(X^{k_0+1})^3}\lesssim s$, which follows from \eqref{claim1}. Also it follows from (b), \eqref{claim1} and \eqref{claim2} that
\begin{align*}
&|2(v\cdot R_1)\partial_{\Theta}DG(\Theta_3^*,sv)[(I-P)[R_1] ]|_{(Y^{k_0})^3} \lesssim s^3,\\
&|D^2G(\Theta_3^*,sv)[ (I-P)R_1,(I-P)R_1]|_{(Y^{k_0})^3} \lesssim s^4,
\end{align*}
therefore, $|D^2\tilde{G}_s(0)[R_1,R_1] |_{(Y^{k_0})^3} \lesssim s^3 $.
Thus, we have $J_1 \lesssim  s^3$, which proves \eqref{claim3} for $n=0$. If $n\ge 1$, the claim in \eqref{claim3} follows immediately from (b') in Lemma~\ref{conditionsforH}.\\

 In order to estimate $J_2$ in \eqref{bnestimate1}, we have that for $n\ge 0$,
\begin{align*}
J_2 & \lesssim {|(I-D\tilde{G}_s(R_{n})\circ T_s(R_{n})^{-1})[\tilde{G}_s(R_{n})]|_{(Y^{k_0})^3}} \\
& \  +{|D\tilde{G}_s(R_{n})\left(I-S(N_{n})\right)T_s(R_{n})^{-1}[\tilde{G}_s(R_{n})]|_{(Y^{k_0})^3}}\\
& =: J_{21} + J_{22}.
\end{align*}
Then it follows from \eqref{regularizing1} and \eqref{Approximate_inv} that
\begin{align*}
J_{21} &\lesssim |\tilde{G}_s(R_{n})|_{(Y^{k_0})^3}|T_s(R_{n})^{-1}[\tilde{G}_{s}(R_{n})]|_{(X^{k_0+1})^3} \\
& \lesssim |\tilde{G}_s(R_{n})|_{(Y^{k_0})^3}|S(N_{n})T_s(R_{n})^{-1}[\tilde{G}_s(R_{n})]|_{(X^{k_0+1})^3} \\
& \ + |\tilde{G}_s(R_{n})|_{(Y^{k_0})^3}|\left(I-S(N_{n})\right)T_s(R_{n})^{-1}[\tilde{G}_s(R_{n})]|_{(X^{k_0+1})^3} \\
& \lesssim N_{n}|\tilde{G}_s(R_{n})|_{(Y^{k_0})^3}|T_s(R_{n})^{-1}[\tilde{G}_s(R_{n})]|_{(X^{k_0})^3}\\
&  \ + N_{n}^{-\beta}|\tilde{G}_{s}(R_{n})|_{(Y^{k_0})^3}|T_s(R_{n})^{-1}[\tilde{G}_s(R_{n})]|_{(X^{k_0+1+\beta})^3} \\
&\lesssim \frac{1}{s}N_{n}|\tilde{G}_{s}(R_{n})|_{(Y^{k_0})^3}^2+N_{n}^{-\beta}|\tilde{G}_{s}(R_{n})|_{(Y^{k_0})^3}|T_s(R_{n})^{-1}[\tilde{G}_s(R_{n})]|_{(X^{k_0+1+\beta})^3}\\
&\lesssim \frac{1}{s}N_{n}b_{n}^2+N_{n}^{-\beta}b_{n}C_{n},
\end{align*}
where the fourth inequality follows from \eqref{crudebound}.
Also we have
\begin{align*}
J_{22}\lesssim |\left( I- S(N_{n})\right)T_s(R_{n})^{-1}[\tilde{G}_{s}(R_{n})]|_{(X^{k_0+1})^3} \lesssim N_{n}^{-\beta}C_{n}.
\end{align*}
Hence we have for $n\ge 0$ that
\begin{align}\label{J2estimate1}
J_2 \lesssim \frac{1}{s}N_{n}b_{n}^2 + N_{n}^{-\beta}(1+b_{n})C_{n}.
\end{align}
Thus with \eqref{bnestimate1}, \eqref{claim3}, \eqref{J2estimate1} and (a') in lemma \ref{conditionsforH}, we have that
\begin{align}\label{bnrecursion}
\begin{cases}
b_{1} \lesssim (1+N_0)s^{3} + N_0^{-\beta}C_{0} \\
b_{n+1} \lesssim a_{n}^2 + \frac{1}{s}N_{n}b_{n}^2 + N_{n}^{-\beta}(1+b_{n})C_{n} & \text{ for }n\ge 1.
\end{cases}
\end{align}
\textbf{Proof of \eqref{recursion5} and \eqref{recursion6}.}
Finally, for \eqref{recursion5} and \eqref{recursion6}, we use \eqref{crudebound} to compute
\begin{align}
&|R_{n+1}-R_{n}|_{(X^{k_0+2})^3} = | S(N_n)T_s(R_{n})^{-1}[\tilde{G}_s(R_n)]|_{\left(X^{k_0+2}\right)^3} \lesssim N_n^2\frac{1}{s}b_n, \label{highnormestimate1} \\
&|R_{n+1}-R_{n}|_{(X^{k_0+4})^3} = | S(N_n)T_s(R_{n})^{-1}[\tilde{G}_s(R_n)]|_{\left(X^{k_0+4}\right)^3} \lesssim N_n^4\frac{1}{s}b_n. \label{highnormestimate}
\end{align}
This finishes the proof.
\end{proof}
Let us now derive a recursive formula for $C_n$.
\begin{lemma}\label{iteration2}
For $j=0,\ldots,n-1$ assume that $R_j$ satisfies the same assumptions as in Lemma~\ref{iteration1}, that is, given $\epsilon>0$ and $k_0\ge2$,  $R_j$ satisfies  \eqref{assumptionfory1}, \eqref{assumptionfory2} and \eqref{assumptionfory3} and the assumptions of Lemma~\ref{conditionsforH} for $s\in (0,s_0(\epsilon,k_0))$ and $\delta=\delta(\ep,k_0)>0$. Then,
\begin{align*}
C_{n} \lesssim_{\ep,k_0,\beta} \frac{1}{s}\left( 1 + n\sup_{j=0,\ldots n-1}N_{j}^{4} C_{j} \right), \quad \text{ for $n\ge 1$ and  }\quad C_0\lesssim_{\ep,k_0,\beta} s.
\end{align*}

 \end{lemma}
\begin{proof}
It follows from (a') in Lemma~\ref{conditionsforH} and \eqref{highernorm_inversion} in Lemma~\ref{approxinv} that
\begin{align*}
C_0 \lesssim \frac{1}{s}|\tilde{G}_s(0)|_{Y^{k_0+1+\beta}\times (Y^{k_0+\beta})^{2}}\lesssim \frac{1}{s}|\tilde{G}_s(0)|_{(Y^{k_0+1+\beta})^3}\lesssim s.
\end{align*}
Now let us assume that $n\ge 1$. It follows from the definition of $C_n$ in \eqref{def_of_sequences} and \eqref{highernorm_inversion} that (recall that $\beta$ is even)
\begin{align*}
C_n \lesssim \frac{1}{s} \left( (1+|R_n|_{(X^{k_0+4+\beta})^3})|\tilde{G}_s(R_n)|_{Y^{k_0+1}\times (Y^{k_0})^{2}} + |\tilde{G}_s(R_n)|_{Y^{k_0+1+\beta}\times (Y^{k_0+\beta})^{2}}\right).
\end{align*}
Using \eqref{lineargrowth1}, we have
\begin{align*}
&|\tilde{G}_s(R_n)|_{Y^{k_0+1}\times (Y^{k_0})^{2}} \lesssim |\tilde{G}_s(R_n)|_{(Y^{k_0+1})^3} \lesssim 1+ | R_n |_{(X^{k_0+2})^3}\lesssim 1,\\
&|\tilde{G}_s(R_n)|_{Y^{k_0+1+\beta}\times (Y^{k_0+\beta})^{2}} \lesssim 1+ | R_n |_{(X^{k_0+2+\beta})^3}.
\end{align*}
therefore,  we have
\begin{align*}
C_n \lesssim \frac{1}{s} (1 + |R_n|_{(X^{k_0+4+\beta})^3}).
\end{align*}
For $R_n$, we have
\begin{align*}
|R_n|_{(X^{k_0+4+\beta})^3}  &\lesssim \sum_{j=0}^{n-1}\left|R_{j+1} - R_{j}\right|_{(X^{k_0+4+\beta})^3} \\
& \lesssim \sum_{j=0}^{n-1}\left|S(N_j)T_s(R_j)^{-1}[\tilde{G}_s(R_j)]\right|_{(X^{k_0+4+\beta})^3} \\
& \lesssim \sum_{j=0}^{n-1}N_j^3 C_j\\
& \lesssim n\sup_{j=0,\ldots,n-1}N_{j}^3C_{j}.
\end{align*}
Hence, we obtain the desired result.
\end{proof}

Now we are ready to prove the main theorem of this section. 
\begin{proofthm}{theorem1}
We fix $k_0\ge2$ and pick  $\tilde{\epsilon},\ \tilde{s}>0$ so that
\begin{align}\label{epsilonpick}
\sum_{k=1}^{\infty}s^{-1+2\left(\frac{67}{64} \right)^{k}} \le s^{1+ \tilde{\epsilon}} \text{ for all $s \in (0,\tilde{s})$.} 
\end{align}
And we choose 
\begin{align}
\ \beta:= 816\quad \text{ and }\quad \epsilon:= \min\left\{ \frac{1}{4}, \frac{\tilde{\epsilon}}{2} \right\} >0, \label{param1}\end{align}
 and let $s_0=s_0(\epsilon,k_0),$ $
\delta=\delta(\epsilon,k_0)$ be as in Lemma~\ref{approxinv}. As before, we can also assume $\delta$ is small enough so that $\delta <\eta(k_0)$ where $\eta$ is as in (b') in Lemma~\ref{conditionsforH}. Since $k_0$, $\epsilon$ and $\beta$ are fixed, by Lemma~\ref{iteration1} and \ref{iteration2}, we can find a constant $K > 0$  such that as long as $R_n$ satisfies \eqref{assumptionfory1},\eqref{assumptionfory2} and \eqref{assumptionfory3}, for some $s\in (0,s_0)$, it holds that for any sequence of positive numbers $N_n$, and 
\begin{align}
&a_0 \le Ks \label{a0formula}\\
& |PR_1|_{(X^{k_0+2})^3}\le K N_0 s \label{firststep11}\\
& |(I-P)R_1|_{(X^{k_0+2})^3} \le K  N_0 s^2. \label{firststep12}\\
&a_n \le \frac{K}{s}N_nb_n \quad  \text{for $n\ge 1$,} \label{a1formula}\\
&b_0 \le Ks^2, \label{b0formula} \\
&b_1 \le K\left( (N_0^3 + 1)s^3 + N_0^{-\beta}C_0 \right) \label{b1formula}\\
&b_{n+1} \le K(a_n^2  + \frac{1}{s}N_nb_n^2 + N_n^{-\beta}(1+b_n)C_n) \quad \text{ for $n\ge 1$}, \label{bnformula}\\
&a'_n \le \frac{K}{s}N_n^2 b_n  \quad \text{ for $n\ge 0$}, \label{a'0formula}\\
&a''_n \le \frac{K}{s}N_n^4b_n \quad \text{ for $n\ge 0$}, \label{a''0formula}\\
&C_n \le  \frac{K}{s}\left(1 + n \sup_{j=0,\ldots n-1}\left|N_{j}^{4} C_{j}\right| \right) \quad \text{ for $n\ge 1$,} \label{cnformula}\\
& C_0 \le Ks.\label{c0formula}
\end{align} 

In what follows, we assume without loss of generality that 
\begin{align}\label{kanddelta}
K  > 1 , \quad \text{ and } \quad \delta<1.
\end{align}

For such $K$, we can find $s^*$ such that for all $s\in (0,s^*)$ (each of them can be easily verified for small enough $s>0$),
\begin{align}
&2Kn<s^{-2\left(\frac{67}{64}\right)^{n-1}+1} \quad  \text{ for all $n\ge 1$},\label{condfors1} \\
& 4K^3s^{\frac{7}{16}}< \frac{1}{2}, \label{condfors2} \\
& Ks^{\frac{\tilde{\epsilon}}{2}} \le \frac{1}{2}, \label{condfors4}\\
& s<1, \label{condfors3}\\
& K^2s^{\frac{3}{4}} \le \frac{1}{2}\delta, \label{condfors5} \\
& K\sum_{k=1}^{\infty}s^{-1+\frac{127}{64}\left( \frac{67}{64}\right)^k}  \le \frac{1}{2}\delta, \label{condfors6} \\
& K(n+2)  \le s^{1-2\left( \frac{67}{64}\right)^{n}} \text{ for  all $n\ge 0$}, \label{condfors7} \\
& 3K^3 \le s^{2+\left(-\frac{139}{32} + \frac{143}{64}\cdot\frac{67}{64} \right)\left(\frac{67}{64} \right)^{n}} \text{ for all $n\ge 1$}, \label{condfors9}\\
& 3K \le s^{1+\left(-\frac{141}{32}+\frac{143}{64}\left( \frac{67}{64}\right) \right)\left( \frac{67}{64}\right)^{n}} \text{ for all $n\ge 1$}, \label{condfors10}\\
& 6K \le s^{\left( - \frac{816}{16}+\frac{768}{16}+\frac{143}{64}\frac{67}{64}\right)\left(\frac{67}{64} \right)^{n}}  \text{ for all $n\ge 1$}. \label{condfors11}
\end{align}
Lastly, we fix $s\in (0,\min\left\{s_0,s^*,\tilde{s} \right\})$ and 
\begin{align}\label{param2}
N_n := s^{-\frac{1}{16}\left(\frac{67}{64} \right)^{n}} \quad \text{ for $n\ge 0$.}
\end{align}

 We claim that for all $n\ge0$,
\begin{itemize}
\item[$(P1)_n$]: $R_{n+1}$ satisfies \eqref{assumptionfory1},\eqref{assumptionfory2} and \eqref{assumptionfory3}.
\item[$(P2)_n$]: $C_{n} \le N_n^{768}$.
\item[$(P3)_n$]: $b_{n+1} \le s^{\frac{143}{64}\left(\frac{67}{64}\right)^{n+1}}$.
\item[$(P4)_n$]: $a_{n+1} \le Ks^{-1 + \frac{139}{64}\left( \frac{67}{64}\right)^{n+1}}$.
\end{itemize}
Once we have the above claims, then $(P1)_n$  justifies all the recurrence formulae above for all $n \in \mathbb{N}$, which follows from Lemma~\ref{iteration1} and \ref{iteration2}. Also, it is clear from \eqref{a1formula} and $(P4)_n$ that
\begin{align*}
\sum_{n=0}^{\infty}a_n \le  KN_0s + K\sum_{n=1}^{\infty}  s^{-1 + \frac{139}{64}\left( \frac{67}{64}\right)^{n}} < \infty.
\end{align*}
Therefore $R_n$ is a Cauchy sequence in $(X^{k_0+1})^3$, and we can find a limit $R_\infty:=\lim_{n\to\infty}R_n$. Then, $(P3)_n$ and the continuity of the functional $R\mapsto \tilde{G}_s$ implies that $\tilde{G}_s(R_\infty) = 0$. From the definition of $\tilde{G}_s$ in \eqref{def_tilde_g}, this implies the existence of a solution $R = R(s)\ne 0$ for each $s\in (0,\min\left\{s_0,s^*,\tilde{s} \right\})$ and finishes the proof.

 Now we prove the claim $(Pi)_n$ for $i=1,\ldots,4$. We will follow  the usual induction argument.

\textbf{Initial step.} In the initial case, we assume that $n=0$.  We first prove $(P1)_0$. It follows from \eqref{firststep11},\eqref{firststep12} and \eqref{param2} that 
\begin{align}
&|PR_1|_{(X^{k_0+2})^3}\le Ks^{\frac{15}{16}} \le s^{\frac{1}{4}} \le s^{\epsilon} \label{P1claim1}\\
&|(I-P)R_1|_{(X^{k_0+2})^3} \le K  s^{1+\frac{15}{16}} \le s^{1+ \frac{1}{4}} \le s^{1+\epsilon}, \label{P1claim2}
\end{align}
where the two  second inequalities follow from \eqref{condfors2} ($Ks^{\frac{15}{16}}\le 4K^3s^{\frac{7}{16}}s^{\frac{1}{2}}<s^{\frac{1}{4}}$) and the last inequalities follow from \eqref{condfors3}  and \eqref{param1}. Furthermore, it follows from \eqref{a''0formula}, \eqref{b0formula} and \eqref{param2} that
\begin{align}\label{cknorm1}
|R_1|_{(X^{k_0+4})^3} = a''_0 \le   \frac{K}{s} \cdot \left( s^{-\frac{1}{16}}\right)^4 \cdot \left( K s^2\right) \le K^2s^{\frac{3}{4}} \le \frac{1}{2}\delta,
\end{align}
where the last inequality follows from \eqref{condfors5}. Therefore  $(P1)_0$ holds for $n = 0$. 
  $(P2)_0$ follows immediately from \eqref{c0formula} and \eqref{param2}. In order to prove $(P3)_0$, note that thanks to \eqref{b1formula}, it is enough to show that 
\[
K\left( (N_0^3 + 1)s^3 + N_0^{-\beta}C_0 \right) \le s^{\frac{143}{64}\cdot\frac{67}{64}},
\]
in other words, 
\[K\left( s^{\frac{45}{16}}+s^3 + Ks^{\frac{816}{16}}s \right) \le  s^{\frac{143}{64}\cdot\frac{67}{64}},\]
where we used \eqref{param1}, \eqref{c0formula} and \eqref{param2}. By \eqref{condfors3}, $s^{\frac{45}{16}}$ is the largest value among the three in the parentheses, hence it is sufficient to show that $3K^2 < s^{\frac{143}{64}\cdot \frac{67}{64} - \frac{45}{16}}.$ Since $\frac{143}{64}\cdot \frac{67}{64} - \frac{45}{16} >-\frac{1}{2}$, it is enough to show that  $3K^3 \le s^{-\frac{1}{2}}$, which follows from \eqref{condfors2}.  $(P4)_0$ follows from \eqref{a1formula}, \eqref{param2} and $(P3)_0$.

\textbf{Induction step.}
In this step, we assume that $(Pi)_{k}$ is true for all $0 \le k \le n_0$ and aim to prove $(Pi)_{n_0+1}$. Let us prove  $(P1)_{n_0+1}$ first.  It follows from \eqref{a'0formula},  \eqref{param2} and $(P3)_{k}$ for $k=0,\ldots,n_0$, that 
\[
\sum_{k=1}^{n_0+1} a'_k \le K \sum_{k=1}^{n_0+1} s^{-1}\left( s^{-\frac{1}{16}\left(\frac{67}{64}\right)^k}\right)^2 \cdot s^{\frac{143}{64}\left( \frac{67}{64}\right)^{k}} \le K\sum_{k=1}^{n_0+1}s^{-1+2\left( \frac{67}{64}\right)^{k}} \le K s^{1+\tilde{\epsilon}},
\]
where the last inequality follows from our choice on $\tilde{\epsilon}$ and $s\in \tilde{s}$ in \eqref{epsilonpick}. Hence, we have
\begin{align*}
\sum_{k=1}^{n_0+1} a'_k  \le \left( K s^{\frac{\tilde{\epsilon}}{2}}\right)s^{1+\frac{\tilde{\epsilon}}{2}} \le \frac{1}{2}s^{1+ \frac{\tilde{\epsilon}}{2}}\le  \frac{1}{2} s^{1+\epsilon},
\end{align*}
where the second inequality follows from \eqref{condfors4} and the last inequality follows from \eqref{param1} and \eqref{condfors3}. Therefore we obtain
\[
|PR_{n_0+2}|_{(X^{k_0+2})^3} \le |PR_{1}|_{(X^{k_0+2})^3} + \sum_{k=1}^{n_0+1}a'_k \le Ks^{\frac{15}{16}} + \frac{1}{2}s^{1+\epsilon}\le  \left( K s^{\frac{7}{16}}\right) s^{\frac{1}{2}} + \left( \frac{1}{2}s\right)s^{\epsilon} \le s^{\epsilon},
\]
where the second inequality follows from \eqref{P1claim1} and the last inequality follows from $\epsilon < \frac{1}{2}$ and $Ks^{\frac{7}{16}}\le \frac{1}{2}$, which can be deduced from \eqref{condfors4} and  \eqref{param1}. Using \eqref{P1claim2}, instead of \eqref{P1claim1}, one can easily obtain 
\[
|(I-P)R_{n_0+2}|_{(X^{k_0+2})^3} \le s^{1+\epsilon}.
\]
To prove \eqref{assumptionfory3} for $R_{n_0+2}$, we compute
\begin{align}\label{a''nestimate}
|R_{n_0+2}|_{(H^{k_0+4})^3} &\le a''_0 +  \sum_{k=1}^{n_0+1}a''_k \le \frac{1}{2}\delta + K \sum_{k=1}^{n_0+1} s^{-1}\left( s^{-\frac{1}{16}\left(\frac{67}{64}\right)^k}\right)^4 \cdot s^{\frac{143}{64}\left( \frac{67}{64}\right)^{k}} \nonumber\\
&  \le \frac{1}{2}\delta +  K\sum_{k=1}^{\infty}s^{-1+\frac{127}{64}\left( \frac{67}{64}\right)^k} \le \delta,
\end{align}
where the second inequality follows from \eqref{cknorm1}, \eqref{a''0formula} and $(P3)_k$, for $0\le k\le n_0$, and the last inequality follows from \eqref{condfors6}. This proves $(P1)_{n_0+1}$. 

 We turn to $(P2)_{n_0+1}$. Using \eqref{cnformula} and $(P2)_{n_0}$, we have
 \[
 C_{n_0+1} \le \frac{K}{s}(2+n_0)N_{n_0}^4\sup_{j=0,\ldots,n_0}C_{j} \le \frac{K}{s}(2+n_0)N_{n_0}^{772},
 \]
 where the last inequality follows from $(P2)_{n_0}$. Hence,  it suffices to show that
 \[
  \frac{K}{s}(2+n_0)N_{n_0}^{772} \le (N_{n_0+1})^{768}. 
 \]
 Plugging \eqref{param2}, this is equivalent to
\[
K(n_0+2) \le s^{1+\left( \frac{67}{64}\right)^{n_0}\left(\frac{772}{16}-\frac{768}{16}\cdot \frac{67}{64} \right)} = s^{1-2\left( \frac{67}{64}\right)^{n_0}},
\] 
which is true thanks to our choice of $s$ in \eqref{condfors7}. This proves $(P2)_{n_0+1}$. 

For $(P3)_{n_0+1}$, thanks to \eqref{bnformula}, it is enough to show that
\[
K\left( a_{n_0+1}^2 + \frac{1}{s}N_{n_0+1}b_{n_0+1}^2 + N_{n_0+1}^{-\beta}(1+b_{n_0+1})C_{n_0+1}) \right) \le s^{\frac{143}{64}\left( \frac{67}{64}\right)^{n_0+2}}.
\]
  Using $(P4)_{n_0}$, \eqref{param2}, $(P3)_{n_0}$, \eqref{param1}, $(P2)_{n_0+1}$ and \eqref{condfors3}, which implies $b_{n_0+1}\le 1$ , we only need to show that
  \begin{align}\label{p3indunction}
  K\left( \underbrace{\left( K s^{-1+\frac{139}{64}\left( \frac{67}{64}\right)^{n_0+1}} \right)^2 }_{=:A_1} + \underbrace{ s^{-1}s^{-\frac{1}{16}\left(\frac{67}{64}\right)^{n_0+1}}s^{\frac{143}{32}\left(\frac{67}{64}\right)^{n_0+1}}}_{=:A_2} + \underbrace{ 2s^{\frac{816}{16}\left(\frac{67}{64}\right)^{n_0+1}}s^{-\frac{768}{16}\left(\frac{67}{64}\right)^{n_0+1}}}_{=:A_3}\right) \le s^{\frac{143}{64}\left(\frac{67}{64}\right)^{n_0+2}}.
  \end{align}
  We will show that $KA_i \le \frac{1}{3}s^{\frac{143}{64}\left( \frac{67}{64}\right)^{n_0+2}}$ for $i=1,2,3$. For $A_1$, it suffices to show that
  \[
  K^3s^{-2+\frac{139}{32}\left(\frac{67}{64}\right)^{n_0+1}} \le \frac{1}{3}s^{\frac{143}{64}\left(\frac{67}{64}\right)^{n_0+2}},
  \]
  equivalently,
  \[
  3K^3 \le s^{2+\left(-\frac{139}{32} + \frac{143}{64}\cdot\frac{67}{64} \right)\left(\frac{67}{64} \right)^{n_0+1}},
  \]
and this follows from our choice of $s$ in \eqref{condfors9}. For $A_2$, we need to show that
\[Ks^{-1+\frac{141}{32}\left( \frac{67}{64}\right)^{n_0+1}} \le \frac{1}{3}s^{\frac{143}{64}\left( \frac{67}{64}\right)^{n_0+2}},
\]
eqvalently,
\[
3K \le s^{1+\left(-\frac{141}{32}+\frac{143}{64}\left( \frac{67}{64}\right) \right)\left( \frac{67}{64}\right)^{n_0+1}},
\]
and this follows from \eqref{condfors10}.  For $A_3$, it is enough to show that 
\[
2K s^{\frac{816}{16}\left(\frac{67}{64}\right)^{n_0+1}} \cdot s^{-\frac{768}{16}\left( \frac{67}{64}\right)^{n_0+1}} \le \frac{1}{3}s^{\frac{143}{64}\left( \frac{67}{64}\right)^{n_0+2}},
\]
equivalently,
\[
6K \le s^{\left( - \frac{816}{16}+\frac{768}{16}+\frac{143}{64}\frac{67}{64}\right)\left(\frac{67}{64} \right)^{n_0+1}},
\]
which follows from \eqref{condfors11}. This proves $(P3)_{n_0+1}$.
 
  Lastly, $(P4)_{n_0+1}$ can be proved by \eqref{a1formula} and $(P3)_{n_0 + 1}$, that is,
  \[
  a_{n_0+2} \le Ks^{-1}N_{n_0+2}b_{n_0+2} \le K s^{-1}s^{\left(-\frac{1}{16} + \frac{143}{64}\right)\left( \frac{67}{64}\right)^{n_0+2}} =  Ks^{-1 + \frac{139}{64}\left( \frac{67}{64}\right)^{n_0+2}},
  \]
  which finishes the proof.
\end{proofthm}

\subsection{Estimates on the velocity}\label{velocity_estimates}
In Subsection~\ref{checking_subsection}, we will check whether our functional $G$ in \eqref{stationarR_equation3} satisfies the hypotheses in Theorem~\ref{theorem1}. In this section, we will derive some useful estimates on the velocity vector generated by each patch. Recall that
given $b_1,b_2>0$, $b_1\ne b_2$, and $R_1,R_2\in C^{\infty}(\mathbb{T})$, we denote for $i,j=1,2$,
\begin{align}
&R = (R_1,R_2) \in (C^\infty(\mathbb{T}))^2,\label{R_2}\\
&z_i(R)(x) = (b_i+R_i(x))(\cos x,\sin x), \label{z_2}\\
&u_{i,j}(R)(x):=\begin{pmatrix} u^1_{i,j}(R) \\ u^2_{i,j}(R) \end{pmatrix} =\int_{\mathbb{T}}\log(|z_j(R)(x)-z_i(R)(y)|^2)z_i(R)'^{\perp}(y)dy.\label{u_def}
\end{align}

\begin{prop}\label{growth_velocity}
For $k\ge 2$, there exists $\epsilon=\epsilon(k,b_1,b_2) >0$ such that if $\rVert R \rVert_{(H^{3}(\mathbb{T}))^2} \le \epsilon$, then 
\begin{itemize}
\item[(A)] $u_{i,j}:(H^{k+1}(\mathbb{T}))^2 \mapsto (H^{k}(\mathbb{T}))^2, \text{ and }\rVert u_{i,j} \rVert_{(H^{k}(\mathbb{T}))^2}  \lesssim_{k,b_1,b_2} 1 + \rVert R \rVert_{(H^{k+1}(\mathbb{T}))^2}.$
\item[(B)] The Gateaux derivative $Du_{i,j}(R):(H^{k+1}(\mathbb{T}))^2 \mapsto (H^{k+1}(\mathbb{T}))^2$ exists and 
\[
\rVert Du_{i,j}(R)[h]\rVert_{(H^{k+1}(\mathbb{T}))^2} \lesssim_{k,b_1,b_2} \rVert h \rVert_{(H^{k+1}(\mathbb{T}))^2} + \rVert R \rVert_{(H^{k+3}(\mathbb{T}))^2}\rVert h \rVert_{(H^1(\mathbb{T}))^2}, \text{ for }h\in (C^{\infty}(\mathbb{T}))^2.
\]
\item[(C)] The map $R\mapsto Du_{{i,j}}(R)$ is Lipschitz continuous, in the sense that for $R,r,h \in (C^\infty(\mathbb{T}))^2$ such that $
\rVert R \rVert_{(H^{3}(\mathbb{T}))^2}, \rVert r \rVert_{(H^{3}(\mathbb{T}))^2} \le \epsilon,$ and $\rVert R\rVert_{(H^{k+3}(\mathbb{T}))^2}, \rVert r \rVert_{(H^{k+3}(\mathbb{T}))^2} \le 1$,
\[
\rVert Du_{i,j}(R)[h]-Du_{i,j}(r)[h]\rVert_{(H^{k+1}(\mathbb{T}))^2} \lesssim_{k,b_1,b_2} \rVert R-r\rVert_{(H^2(\mathbb{T}))^2}\rVert h \rVert_{(H^{k+1}(\mathbb{T}))^2} + \rVert R - r \rVert_{(H^{k+3}(\mathbb{T}))^2}\rVert h \rVert_{(H^{1}(\mathbb{T}))^2}. 
\]

\item[(D)] For $\sigma\in \mathbb{N}\cup \left\{ 0 \right\}$, there exists a linear operator $T_{i,j}^\sigma(R)$ such that for $h\in \left(C^\infty(\mathbb{T})\right)^2$, 
\[
\left(\frac{d}{dx}\right)^{\sigma} \left( Du_{i,j}(R)[h]\right) = Du_{i,j}(R)[h^{(\sigma)}] + T_{i,j}^\sigma(R)[h],
\]
where $T_{i,j}^\sigma(R)$ satisfies
\[
\rVert T_{i,j}^\sigma(R)[h]\rVert_{(H^{k+1}(\mathbb{T}))^2} \lesssim_{k} (1+ \rVert R \rVert_{(H^{k+4+\sigma}(\mathbb{T}))^2})\rVert h\rVert_{(L^2(\mathbb{T}))^2} + (1+\rVert R \rVert_{(H^4(\mathbb{T}))^2})\rVert h^{(k+\sigma)} \rVert_{(L^2(\mathbb{T}))^2}.
\]
\end{itemize}
\end{prop}

\begin{rem}
It is well known that roughly speaking, if $\partial D$ is $C^{k+\alpha}$-regular for some $k\ge1,$ $\alpha>0$ then the velocity $\nabla^{\perp}(1_D*\log|x|)$ is also $C^{k+\alpha}$-regular up to the boundary.  In $(A)$ in the above Proposition, we do not aim to prove the optimal regularity since it is not necessary in the proof of the main theorem.
\end{rem}
\color{black}
The proof of the proposition will be given after Lemmas~\ref{Acondition} and \ref{BCcondition}. We will deal with the case $i=j$ only. If $i\ne j$, then the integrand in $u_{i,j}$ has no singularity thus the result follows straightforwardly in a similar manner. Hence, by abuse of notation, we denote for $b>0$ and $R\in C^{\infty}(\mathbb{T})$, 
\begin{align}
&z(R)(x) := (b+R(x))(\cos x, \sin x),\label{single1}\\
&u(R)(x):=\begin{pmatrix}u^1(R) \\ u^2(R) \end{pmatrix} :=\int_{\mathbb{T}}\log\left(|z(R)(x) - z(R)(y)|^2 \right)z(R)'^{\perp}(y)dy.\label{single2}
\end{align}

Let us write $u(R)$ as
\begin{align}\label{f_def}
u(R)(x) & =  \int_{\mathbb{T}} \log\left(2-2\cos(x-y)\right) z(R)'^{\perp}(y)dy +  \int_{\mathbb{T}} \log\left(\frac{|z(R)(x) -z(R)(y)|^2}{2-2\cos(x-y)}\right) z(R)'^{\perp}(y)dy\nonumber\\
& =  \int_{\mathbb{T}} \log\left(2-2\cos(x-y)\right) z(R)'^{\perp}(y)dy + \int_{\mathbb{T}}K(R)(x,y)z(R)'(y)^{\perp}dy  \nonumber\\
& =: u_{L}(R)(x) + u_{N}(R)(x),
\end{align}
where 
\begin{align}
&K(R)(x,y) = F(R(x),R(y),J(R)(x,y)),\nonumber\\
&F(u,v,w) := \log(b^2 + b(u+v)+uv+w^2),\label{F_smooth}\\
&J(R)(x,y) = \frac{R(x)-R(y)}{2\sin(\frac{x-y}{2})}.\nonumber
\end{align}
If $\rVert R \rVert_{H^{3}(\mathbb{T})} \le \epsilon$ for sufficiently small $\epsilon>0$,  it follows from Lemmas~\ref{appendix_lem_1} and \ref{ponce_kato} that 
\begin{align}\label{f_norm1}
\rVert J(R) \rVert_{H^{2}(\mathbb{T}^2)} \lesssim_{b} \rVert R \rVert_{H^{3}(\mathbb{T})} \lesssim \epsilon. 
\end{align}
 Therefore, for sufficiently small $\epsilon>0$, Lemmas~\ref{composition}  and~\ref{appendix_lem_1} imply that for $l\ge 2$,
\begin{align}
&\rVert  K(R) \rVert_{H^{l}(\mathbb{T}^2)} \lesssim_{l,b} 1 + \rVert R \rVert_{H^l(\mathbb{T})} + \rVert J(R)\rVert_{H^{l}(\mathbb{T}^2)} \lesssim_{l,b} 1+ \rVert R \rVert_{H^{l+1}(\mathbb{T})}, \label{Kestimate},
\end{align}

The next lemma will be used to prove the growth condition of the velocity in (A) in Proposition~\ref{growth_velocity}.
\begin{lemma}\label{Acondition}
For $k\ge 2$, there exists $\epsilon=\epsilon(k,b) >0$ such that if $R\in C^{\infty}(\mathbb{T})$ and $\rVert R \rVert_{H^{3}(\mathbb{T})} \le\epsilon$, then
\[
\rVert u \rVert_{(H^{k}(\mathbb{T}))^2}  \lesssim_{k,b} 1 + \rVert R \rVert_{H^{k+1}(\mathbb{T})}.
\]
\end{lemma}
\begin{proof}
For $u_L$, then it follows from Lemma~\ref{T1estimate} and the fact that $R\mapsto u_L(R)$ is linear that
\[
\rVert u_L(R)\rVert_{(H^{k}(\mathbb{T}))^2} \lesssim \rVert z(R)' \rVert_{(H^{k-1}(\mathbb{T}))^2} \lesssim 1 + \rVert R \rVert_{H^{k}(\mathbb{T})}.
\]
Now, let us consider the nonlinear part. We have 
\[
\rVert u_N \rVert_{(H^{k}(\mathbb{T}))^2} \lesssim \rVert K(R) \rVert_{H^{k}(\mathbb{T}^2)}\rVert z(R)'\rVert_{(L^\infty(\mathbb{T}))^2} \lesssim 1+\rVert R \rVert_{H^{k+1}(\mathbb{T})},
\]
where the last inequality follows from \eqref{Kestimate}.
\end{proof}
Now we turn to (B), (C) and (D) in Proposition~\ref{growth_velocity}. We will use the following notations:
\begin{align}
DK(R)[h] := \frac{d}{dt}K(R+th)|_{t=0} &= \underbrace{\partial_uF(R(x),R(y),J(R)(x,y))}_{=:\partial_u F(R)(x,y)}h(x) \nonumber\\
&\ +  \underbrace{\partial_vF(R(x),R(y),J(R)(x,y))}_{=:\partial_vF(R)(x,y)}h(y) \nonumber\\
& \ + \underbrace{\partial_wF(R(x),R(y),J(R)(x,y))}_{=:\partial_wF(R)(x,y)}J(h)(x,y),\label{DK_def} 
\end{align}
where $J(h):=\frac{h(x)-h(y)}{2\sin\left(\frac{x-y}{2} \right)}$ and
\begin{align}
Dz(R)'^\perp[h] :=\frac{d}{dt}z(R+th)'^\perp|_{t=0} =  -h'(x)(\sin x,\cos x) -h(x)(\cos x, \sin x).  \label{Dz_def}
\end{align}
With the above notations, we can write the derivative of the nonlinear term as
\begin{align}\label{Du_Nestimate1}
Du_N(R)[h] & = \int_{\mathbb{T}}K(R)(x,y)Dz(R)'^\perp[h](y)dy + \int_{\mathbb{T}}DK(R)[h](x,y)z(R)'^\perp(y)dy  \nonumber\\
& =: Du_N^1(R)[h] + Du_N^2(R)[h].
\end{align}
Again, Lemma~\ref{composition} and \eqref{f_norm1} yield that if $\rVert R \rVert_{H^3(\mathbb{T})}\le \epsilon$ for sufficiently small $\epsilon>0$, then for $l\ge 2$,
\begin{align}\label{Kernal_norm1}
\rVert \partial_uF(R) \rVert_{(H^{l}(\mathbb{T}))^2},\rVert \partial_vF(R) \rVert_{(H^{l}(\mathbb{T}))^2},\rVert \partial_wF(R) \rVert_{(H^{l}(\mathbb{T}))^2}& \lesssim_l 1+\rVert R \rVert_{H^{l}(\mathbb{T})} + \rVert J(R) \rVert_{H^{l}(\mathbb{T}^2)}\nonumber \\
&  \lesssim_l 1 + \rVert R \rVert_{H^{l+1}(\mathbb{T})}.
\end{align}
Also, if $R_1,R_2 \in C^\infty({\mathbb{T}})$ and $\rVert R_1\rVert_{H^{3}(\mathbb{T})},\rVert R_2\rVert_{H^{3}(\mathbb{T})} \le \epsilon$, then from Lemma~\ref{composition}, it follows that
\begin{align}\label{Kernal_norm1_lip}
\rVert \partial_uF(R_1) - \partial_uF(R_2) \rVert_{(H^{l}(\mathbb{T}))^2},\rVert \partial_vF(R_1) - \partial_vF(R_2) \rVert_{(H^{l}(\mathbb{T}))^2},\rVert \partial_wF(R_1) - \partial_wF(R_2) \rVert_{(H^{l}(\mathbb{T}))^2}  \lesssim_l \rVert R_1-R_2 \rVert_{H^{l+1}(\mathbb{T})}.
\end{align}
For ${\partial_wF^\#(R)}(x) := \lim_{y\to x}\partial_wF(R)(x,y) = \partial_wF(R(x),R(x),R'(x))$, we have
\begin{align}
&\rVert {\partial_wF^\#(R)} \rVert_{H^{l}(\mathbb{T})}\lesssim_l 1 + \rVert R \rVert_{H^{l+1}(\mathbb{T})}, \label{Kernal_norm3}\\
& \rVert {\partial_wF^\#(R)} \rVert_{L^{\infty}(\mathbb{T})} \lesssim 1 + \rVert R \rVert_{H^2(\mathbb{T})} \lesssim 1.\label{Kernal_norm5}\\
&\rVert {\partial_wF^\#(R_1)}- {\partial_wF^\#(R_2)} \rVert_{H^{l}(\mathbb{T})}\lesssim_l  \rVert R_1-R_2 \rVert_{H^{l+1}(\mathbb{T})}, \label{Kernal_norm3_lip}\\
& \rVert {\partial_wF^\#(R_1)}- {\partial_wF^\#(R_2)} \rVert_{L^{\infty}(\mathbb{T})} \lesssim \rVert R_1-R_2 \rVert_{H^2(\mathbb{T})}.\label{Kernal_norm5_lip}
\end{align}
We decompose $Du_N^2(R)[h]$ into
\begin{align}Du_N^2(R)[h] &= \int_{\mathbb{T}}\underbrace{\partial_uF(R)(x,y)z(R)'^\perp}_{=:K_1(R)(x,y)} h(x)dy +\int_{\mathbb{T}}\underbrace{\partial_vF(R)(x,y)z(R)'^\perp}_{=:K_2(R)(x,y)} h(y)dy \nonumber\\
& \ +\int_{\mathbb{T}}\underbrace{\partial_wF(R)(x,y)z(R)'^\perp}_{=:K_3(R)(x,y)} J(h)(x,y)dy\label{def_K_3}\\
& =: Du_{N1}^2(R)[h]+Du_{N2}^2(R)[h]+Du_{N3}^2(R)[h]. \label{Du_N_decomp}
\end{align}
Then it follows from Lemma~\ref{ponce_kato}, \eqref{Kernal_norm1}, \eqref{Kernal_norm3} and \eqref{Kernal_norm5}  that for $l\ge 2$,
\begin{align}
&\rVert K_i(R) \rVert_{H^{l}(\mathbb{T}^2)} \lesssim_l  1 + \rVert R \rVert_{H^{l+1}(\mathbb{T})}, \text{ for }i=1,2,3, \label{Kernal_norm2} \\
&\rVert {K_3}^\#(R) \rVert_{H^{l}(\mathbb{T})} \lesssim 1+\rVert R \rVert_{H^{l+1}(\mathbb{T})}, \text{ where }{K_3}^\#(R)(x):=K_3(R)(x,x),\label{Kernal_norm4}\\
&\rVert {K_3}^\#(R) \rVert_{L^\infty(\mathbb{T})} \lesssim \rVert  {K_3}^\#(R) \rVert_{H^1(\mathbb{T})} \lesssim  \rVert  {K_3}^\#(R) \rVert_{H^2(\mathbb{T})} \lesssim 1+ \rVert R \rVert_{H^3({\mathbb{T}})} \lesssim 1,\label{Kernal_norm6}\\
&\rVert K_i(R_1)-K_2(R_2) \rVert_{H^{l}(\mathbb{T}^2)}  \lesssim_l  \rVert R_1-R_2 \rVert_{H^{l+1}(\mathbb{T})}, \text{ for }i=1,2,3, \label{Kernal_norm2_lip} \\
&\rVert {K_3}^\#(R_1)-{K_3}^\#(R_2) \rVert_{H^{l}(\mathbb{T})} \lesssim \rVert R_1-R_2 \rVert_{H^{l+1}(\mathbb{T})},\label{Kernal_norm4_lip}\\
&\rVert {K_3}^\#(R_1)-{K_3}^\#(R_2) \rVert_{L^\infty(\mathbb{T})} \lesssim  \rVert R_1-R_2 \rVert_{H^2(\mathbb{T})}.\label{Kernal_norm6_lip}
\end{align}
Furthermore, it follows from the definitions of $K_3$ and $\partial_w F$ in \eqref{def_K_3}, \eqref{DK_def} and \eqref{F_smooth} that
\begin{align}\label{K_3_nabla}
\rVert \nabla K_3(R)(x,y) \rVert_{L^\infty(\mathbb{T}^2)} \lesssim  1 + \rVert R''\rVert_{L^\infty(\mathbb{T})} + \rVert \nabla J(R) \rVert_{L^\infty(\mathbb{T}^2)} \lesssim 1 + \rVert R'' \rVert_{L^\infty(\mathbb{T})} \lesssim 1.
\end{align}

\begin{lemma}\label{BCcondition}
For $k\ge 2$, there exists $\epsilon=\epsilon(k,b) >0$ such that if $R\in C^{\infty}(\mathbb{T})$ and $\rVert R \rVert_{H^{3}(\mathbb{T})} \le\epsilon$, the Gateaux derivative $Du(R):H^{k+1}(\mathbb{T}) \mapsto (H^{k+1}(\mathbb{T}))^2$ exists and 
\begin{align}\label{tame_1}
\rVert Du(R)[h]\rVert_{(H^{k+1}(\mathbb{T}))^2} \lesssim_{k,b} \rVert h \rVert_{(H^{k+1}(\mathbb{T}))^2} + \rVert R \rVert_{H^{k+3}(\mathbb{T})}\rVert h \rVert_{H^1(\mathbb{T})}, \text{ for }h\in C^{\infty}(\mathbb{T}).
\end{align}
Furthermore, if $R_1,R_2\in C^\infty(\mathbb{T})$,  $\rVert R_1 \rVert_{H^3(\mathbb{T})},\rVert R_2 \rVert_{H^3(\mathbb{T})}\le \epsilon$, and $\rVert R_1\rVert_{H^{k+3}(\mathbb{T})}, \rVert R_2 \rVert_{H^{k+3}(\mathbb{T})} \le 1$, then
\begin{align}\label{tame2}
\rVert Du(R_1)-Du(R_2)[h]\rVert_{(H^{k+1}(\mathbb{T}))^2} \lesssim_{k,b} \rVert R_1 -R_2 \rVert_{H^2(\mathbb{T})}\rVert h \rVert_{H^{k+1}(\mathbb{T})} + \rVert R_1-R_2 \rVert_{H^{k+3}(\mathbb{T})}\rVert h \rVert_{H^1(\mathbb{T})},
\end{align}
$ \text{ for }h\in C^{\infty}(\mathbb{T}).$ 

Lastly,  for $\sigma\in \mathbb{N}\cup \left\{ 0 \right\}$, there exists a linear operator $T_\sigma(R)$ such that for $h\in C^\infty(\mathbb{T})$, 
\begin{align}\label{tame3}
\left(\frac{d}{dx}\right)^{\sigma} \left( Du(R)[h]\right) = Du(R)[h^{(\sigma)}] + T^\sigma(R)[h],
\end{align}
where $T^\sigma(R)$ satisfies
\begin{align}\label{tame4}
\rVert T^\sigma(R)[h]\rVert_{(H^{k+1}(\mathbb{T}))^2} \lesssim_{k} ((1+ \rVert R \rVert_{H^{k+4+\sigma}(\mathbb{T})})\rVert h\rVert_{L^2(\mathbb{T})} + (1+\rVert R \rVert_{H^4(\mathbb{T})})\rVert h^{(k+\sigma)} \rVert_{L^2(\mathbb{T})}.
\end{align}
\end{lemma}
\begin{proof}
We omit the proof of \eqref{tame2} since it can be proved exactly same way as \eqref{tame_1} with \eqref{lip_estimate_F} in Lemma~\ref{composition}. In order to prove \eqref{tame_1}, we deal with $Du_L$. Since it is linear, we have
\[
Du_L(R)[h](x) := \int_{\mathbb{T}}\log(2-2\cos(x-y)) Dz(R)'^\perp[h](y) dy.
\]
Hence, it follows from Lemma~\ref{T1estimate} that
\begin{align}\label{Du_Lestimate}
\rVert Du_L(R)[h]\rVert_{(H^{k+1}(\mathbb{T}))^2} \lesssim \rVert Dz(R)'^\perp[h]\rVert_{(H^{k}(\mathbb{T}))^2} \lesssim \rVert h \rVert_{H^{k+1}(\mathbb{T})}.
\end{align}
Now, we estimate $Du_N$. Recalling the decomposition in \eqref{Du_Nestimate1} we will estimate $Du_N^1$ and $Du_N^2$ separately. For  $Du_N^1$, it follows from \eqref{Kestimate} and \eqref{Dz_def} that
\begin{align}\label{Du_n1estimate}
\rVert Du_N^1(R)[h] \rVert_{(H^{k+1}(\mathbb{T}))^2} \lesssim \rVert K(R)\rVert_{H^{k+1}(\mathbb{T}^2)}\rVert Dz(R)'^\perp \rVert_{(L^2(\mathbb{T}))^2} \lesssim (1+\rVert R\rVert_{H^{k+2}(\mathbb{T})})\rVert h \rVert_{H^1(\mathbb{T})}.
\end{align}
For $Du_N^2$, we recall the decomposition in \eqref{Du_N_decomp}, and use \eqref{Kernal_norm2} and Lemma~\ref{GN} to obtain
\[
\rVert Du_{N1}^2(R)[h]\rVert_{(H^{k+1}(\mathbb{T}))^2},\rVert Du_{N2}^2(R)[h]\rVert_{(H^{k+1}(\mathbb{T}))^2} \lesssim \rVert h \rVert_{H^{k+1}(\mathbb{T})} + \rVert R \rVert_{H^{k+2}(\mathbb{T})}\rVert h \rVert_{L^\infty(\mathbb{T})}.
\]
For $Du_{N3}^2(R)[h]$, we use \eqref{Kernal_norm2}, \eqref{Kernal_norm4}, \eqref{Kernal_norm6}, \eqref{K_3_nabla}  and Lemma~\ref{J_linear} to show that
\[
\rVert Du_{N3}^2(R)[h]\rVert_{(H^{k+1}(\mathbb{T}))^2} \lesssim  \rVert R \rVert_{H^{k+3}(\mathbb{T})}\rVert h\rVert_{H^1(\mathbb{T})} + \rVert h \rVert_{H^{k+1}(\mathbb{T})}.
\]
Therefore $\rVert Du_N^2(R)[h]\rVert_{(H^{k+1}(\mathbb{T}))^2} \lesssim \rVert h \rVert_{H^{k+1}(\mathbb{T})} +\rVert R \rVert_{H^{k+3}(\mathbb{T})}\rVert h \rVert_{H^{1}(\mathbb{T})}$. With \eqref{Du_Lestimate}, \eqref{Du_n1estimate} and \eqref{Du_Nestimate1}, the desired result follows.

Now, we turn to \eqref{tame3} and \eqref{tame4}. Recall from \eqref{f_def},  \eqref{Du_Nestimate1} and \eqref{Du_N_decomp} that
\begin{align*}
Du(R)[h] &= \int_{\mathbb{T}}\log(2-2\cos(y))h(x-y)dy + \int_{\mathbb{T}}K(R)(x,y)Dz(R)'^\perp[h](y)dy \\
& \ + \int_{\mathbb{T}}K_1(R)(x,y)z(R)'^\perp(y) h(x)dy + \int_{\mathbb{T}}K_2(R)(x,y)z(R)'^\perp(y) h(y)dy \\
& \ + \int_{\mathbb{T}}K_3(R)(x,y)z(R)'^\perp(y) J(h)(x,y)dy\\
& = I_1(R)[h]+I_2(R)[h]+I_3(R)[h]+I_4(R)[h]+I_5(R)[h]
\end{align*}
where $K,K_1,K_2,K_3$ are of the form $H(R(x),R(y),J(R)(x,y))$ for some smooth function $H:\RR^3 \mapsto \RR$. It suffices to show that for each $i=1,\ldots,5$ and $\sigma\in \mathbb{N}\cup \left\{ 0 \right\}$,
\[
\left(\frac{d}{dx}\right)^{\sigma} I_i(R)[h] = I_i(R)[h^{\sigma}] + T_i^{\sigma}(R)[h],
\]
for some $T_i^\sigma(R)$ such that 
\begin{align}\label{tame5}
\rVert T_i^\sigma(R)[h]\rVert_{(H^{k+1}(\mathbb{T}))^2} \lesssim (1+ \rVert R \rVert_{H^{k+4+\sigma}(\mathbb{T})})\rVert h\rVert_{L^2(\mathbb{T})} + (1+\rVert R \rVert_{H^4(\mathbb{T})})\rVert h^{(k+\sigma)} \rVert_{L^2(\mathbb{T})}.
\end{align} We only deal with $I_5$ since the other terms can be done in the same way.  We also assume $\sigma\ge 1$ since $\sigma=0$ follows trivially. 
Let $\tilde{K}_3(R)(x,y):=K_3(R)(x,x-y)z(R)'^\perp(x-y)$. Then it follows from \eqref{ponce_kato}, \eqref{Kernal_norm2}, \eqref{Kernal_norm4} and \eqref{Kernal_norm6} that for $l\ge2$,
\begin{align}
&\rVert \tilde{K}_3(R) \rVert_{H^{l}(\mathbb{T}^2)} \lesssim_l  1 + \rVert R \rVert_{H^{l+1}(\mathbb{T})}, \label{Kernal_norm12} \\
&\rVert {\tilde{K}_3}^\#(R) \rVert_{H^{l}(\mathbb{T})} \lesssim 1+\rVert R \rVert_{H^{l+1}(\mathbb{T})}, \text{ where }{\tilde{K}_3}^\#(R)(x):=\tilde{K}_3(R)(x,x),\label{Kernal_norm14}.
\end{align}
  
  Using the change of variables, $y\mapsto x-y$, we have
 \begin{align}
 \left( \frac{d}{dx}\right)^{\sigma}I_5(R)[h] &= \int_{\mathbb{T}}\tilde{K}_3(R)(x,y)J(h^{(\sigma)})(x,x-y)dy + \sum_{p+q=\sigma,q\le \sigma-1}C_{p,q,\sigma}\int_{\mathbb{T}}(\partial_x)^{p}\tilde{K}_3(R)(x,y)J(h^{(q)})(x,x-y)dy,\nonumber\\
 & =: I_5(R)[h^{(\sigma)}](x) + \sum_{p+q=\sigma,q\le \sigma-1}C_{p,q,\sigma}\underbrace{\int_{\mathbb{T}}(\partial_x)^{p}\tilde{K}_3(R)(x,x-y)J(h^{(q)})(x,y)dy}_{=:T_{5}^{\sigma,p,q}(R)[h](x)}\label{tame6}\\
 & =: I_5(R)[h^{(\sigma)}](x) + T_{5}^{\sigma}(R)[h](x),\label{tame7}
 \end{align}
 where  $C_{p,q,\sigma}$ is some constant and  we used $\partial_{x}\left( J(h)(x,x-y)\right) = J(h')(x,x-y)$. It suffices to show that $T_5^{\sigma,p,q}(R)[h]$ satisfies \eqref{tame5}. Let $\tilde{K}_3^*(R)(x,y) := (\partial_x)^p\tilde{K}_3(R)(x,x-y)$.
 Then it follows from \eqref{Kernal_norm12} and \eqref{Kernal_norm14} that for $l\ge2$,
\begin{align}
&\rVert \tilde{K}^*_3(R) \rVert_{(H^{l}(\mathbb{T}^2))^2} \lesssim_l  1 + \rVert R \rVert_{H^{l+p+1}(\mathbb{T})}, \label{Kernal_norm22} \\
&\rVert ({\tilde{K}^*_3})^\#(R) \rVert_{(H^{l}(\mathbb{T}))^2} \lesssim 1+\rVert R \rVert_{H^{l+p+1}(\mathbb{T})}, \text{ where }({\tilde{K}^*_3})^\#(R)(x):=\tilde{K}^*_3(R)(x,x),\label{Kernal_norm24}\\
&\rVert ({\tilde{K}^*_3})^\#(R) \rVert_{(L^\infty(\mathbb{T}))^2}  \lesssim \rVert \tilde{K}_3(R) \rVert_{(W^{p,\infty}(\mathbb{T}^2))^2} \lesssim \rVert \tilde{K}_3(R) \rVert_{(H^{p+2}(\mathbb{T}^2))^2}.\label{Kernal_norm26}
\end{align}
Furthermore, it follows similarly as \eqref{K_3_nabla} that
\begin{align}\label{K_3_nabla1}
\rVert \nabla \tilde{K}^*_3(R)\rVert_{L^\infty(\mathbb{T}^2)}\lesssim \rVert \nabla^{(p+1)}\tilde{K}_3(R)\rVert_{L^\infty(\mathbb{T}^2)} \lesssim 1+ \rVert R^{(p+2)} \rVert_{L^\infty(\mathbb{T})} \lesssim 1+ \rVert R \rVert_{H^{p+3}(\mathbb{T})}.
\end{align}
Hence it follows from Lemma~\ref{J_linear} that
\begin{align*}
\rVert T^{\sigma,p,q}_5(R)[h]\rVert_{H^{k+1}(\mathbb{T})} &  \lesssim \left( \rVert \tilde{K}^*_3(R) \rVert_{(H^{k+2}(\mathbb{T}^2))^2} + \rVert  ({\tilde{K}^*_3})^\#(R) \rVert_{(H^{k+1}(\mathbb{T}^2))^2} \right)\rVert h^{(q)}\rVert_{H^1(\mathbb{T})} \\
& + \left(\rVert  ({\tilde{K}^*_3})^\#(R) \rVert_{(L^{\infty}(\mathbb{T}^2))^2} + \rVert \nabla\tilde{K}^*_3(R) \rVert_{(L^{\infty}(\mathbb{T}^2))^2}\right)\rVert h^{(q)}\rVert_{H^{k+1}(\mathbb{T})} \\ 
& \lesssim (1+\rVert R \rVert_{H^{k+3+p}(\mathbb{T})})\rVert h^{(q)}\rVert_{H^1(\mathbb{T})} + (1+\rVert R \rVert_{H^{p+3}(\mathbb{T})})\rVert h^{(q)}\rVert_{H^{k+1}(\mathbb{T})}\\
& \lesssim (1+\rVert R \rVert_{H^{k+3+p}(\mathbb{T})})\rVert h\rVert_{H^{1+q}(\mathbb{T})} + (1+\rVert R \rVert_{H^{p+3}(\mathbb{T})})\rVert h\rVert_{H^{k+1+q}(\mathbb{T})}
\end{align*}
where the second last inequality follows from  \eqref{K_3_nabla1}.
 Now,  recalling that $\rVert R \rVert_{H^3(\mathbb{T})}\le \epsilon$ and plugging the following inequalities into the above estimate,
 \[
 \rVert R \rVert_{H^{l}(\mathbb{T})} \lesssim \rVert R \rVert_{L^2(\mathbb{T})} + \rVert R^{(l)}\rVert_{L^2(\mathbb{T})}, \quad \rVert h \rVert_{H^{l}(\mathbb{T})} \lesssim \rVert h \rVert_{L^2(\mathbb{T})} + \rVert h^{(l)}\rVert_{L^2(\mathbb{T})},  \text{ for }l\ge 0,
 \]
 we obtain
 \begin{align}\label{T5sigma}
 \rVert T^{\sigma,p,q}_5(R)[h]\rVert_{H^{k+1}(\mathbb{T})} & \lesssim (1 + \rVert R^{(k+p+3)}\rVert_{L^2(\mathbb{T})})\left( \rVert h \rVert_{L^2(\mathbb{T})} + \rVert h^{(1+q)} \rVert_{L^2(\mathbb{T})} \right) \nonumber\\
 & \ + (1 + \rVert R^{(p+3)} \rVert_{L^2(\mathbb{T})})(\rVert h \rVert_{L^2(\mathbb{T})}+\rVert h^{(k+1+q)} \rVert_{L^2(\mathbb{T})})\nonumber\\
 & =: L_1 + L_2.
 \end{align}
 For $L_1$, note that $k+p+3 \le k+3 + \sigma$ and $q\le \sigma-1$, hence we have
 \begin{align}\label{L1estimate1}
 L_1 \lesssim (1 + \rVert R^{(k+\sigma+3)}\rVert_{L^2(\mathbb{T})})\rVert h \rVert_{L^2(\mathbb{T})} + \rVert h^{(\sigma)}\rVert_{L^2(\mathbb{T})} +  \rVert R^{(k+p+3)}\rVert_{L^2(\mathbb{T})} \rVert h^{(1+q)} \rVert_{L^2(\mathbb{T})}.
 \end{align}
Using the interpolation inequality in Lemma~\ref{GNinterpolation}, we have
\begin{align*}
&\rVert R^{(k+p+3)}\rVert_{L^2(\mathbb{T})} = \rVert (R^{(4)})^{(k+p-1)}\rVert_{L^2(\mathbb{T})} \lesssim \rVert R^{(4)}\rVert_{L^2(\mathbb{T})}^{1-\frac{k+p-1}{k+\sigma}}\rVert (R^{(4)})^{(k + \sigma)} \rVert_{L^2({\mathbb{T})}}^{\frac{k+p-1}{k+\sigma}} \\
& \rVert h^{(1+q)}\rVert_{L^2(\mathbb{T})} \lesssim \rVert h \rVert_{L^2(\mathbb{T})}^{1-\frac{1+q}{k+\sigma}}\rVert h^{(k+\sigma)}\rVert_{L^2(\mathbb{T})}^{\frac{1+q}{k+\sigma}}.
\end{align*}
Thus, we have
\begin{align*}
\rVert R^{(k+p+3)}\rVert_{L^2(\mathbb{T})} \rVert h^{(1+q)} \rVert_{L^2(\mathbb{T})} & \lesssim\left(\rVert R^{(4)}\rVert_{L^2(\mathbb{T})}\rVert h^{(k+\sigma)}\rVert_{L^2(\mathbb{T})} \right)^{\frac{1+q}{k+\sigma}} \left( \rVert (R^{(4)})^{(k + \sigma)} \rVert_{L^2({\mathbb{T})}}\rVert h \rVert_{L^2(\mathbb{T})}  \right)^{\frac{k+p-1}{k+\sigma}} \\
& \lesssim \rVert R^{(4)}\rVert_{L^2(\mathbb{T})}\rVert h^{(k+\sigma)}\rVert_{L^2(\mathbb{T})} +  \rVert R^{(k+4+\sigma)}\rVert_{L^2({\mathbb{T})}}\rVert h \rVert_{L^2(\mathbb{T})},
\end{align*}
where we used $p+q = \sigma$ and Young's inequality. Plugging this inequality into \eqref{L1estimate1} and using $\rVert h^{(\sigma)}\rVert_{L^2(\mathbb{T})} \lesssim \rVert h \rVert_{L^{2}(\mathbb{T})} + \rVert h^{(k+\sigma)} \rVert_{L^{2}(\mathbb{T})}$ to estimate the second term on the right-hand side in \eqref{L1estimate1}, we obtain 
\[
L_1 \lesssim (1+ \rVert R \rVert_{H^{k+4+\sigma}(\mathbb{T})})\rVert h\rVert_{L^2(\mathbb{T})} + (1+\rVert R \rVert_{H^4(\mathbb{T})})\rVert h^{(k+\sigma)} \rVert_{L^2(\mathbb{T})}.
\]
Similarly, we can obtain
\[
L_2 \lesssim (1+ \rVert R \rVert_{H^{k+4+\sigma}(\mathbb{T})})\rVert h\rVert_{L^2(\mathbb{T})} + (1+\rVert R \rVert_{H^4(\mathbb{T})})\rVert h^{(k+\sigma)} \rVert_{L^2(\mathbb{T})}.
\]
Recalling \eqref{T5sigma}, \eqref{tame6} and \eqref{tame7}, we obtain \eqref{tame5}. This proves \eqref{tame4} and finishes the proof.
\end{proof}

\begin{proofprop}{growth_velocity}
If $i=j$, then the results follow immediately from Lemma~\ref{Acondition} and \ref{BCcondition}. For $i\ne j$,  the same results follow straightforwardly since there is not singularity in the integrands.
\end{proofprop}

In the next proposition, we estimate the derivative of the velocity with respect to the parameter $b$ in view of \eqref{derivative_matrix_3}. 

\begin{prop}\label{b_derivative}
Let  $R$, $z_i(R)$, $u_{i,j}$, be as in \eqref{R_2}, \eqref{z_2} and \eqref{u_def}. For $k\ge 2$, there exists $\epsilon=\epsilon(k,b_1,b_2) >0$ such that if $\rVert R \rVert_{(H^{3}(\mathbb{T}))^2} \le \epsilon$, then
\begin{align}
\rVert \partial_{b_1} u_{i,j}(R) \rVert_{(H^{k}(\mathbb{T}))^2} \lesssim 1 +\rVert R \rVert_{(H^{k+1}(\mathbb{T}))^2}.\label{growth_b10}
\end{align}
Also, if $\rVert R \rVert_{(H^{3}(\mathbb{T}))^2},\rVert r \rVert_{(H^{3}(\mathbb{T}))^2}  \le \epsilon$ and $\rVert R\rVert_{(H^{k+1}(\mathbb{T}))^2}, \rVert r \rVert_{(H^{k+1}(\mathbb{T}))^2} \le 1$, then,
\begin{align}\label{lip_b10}
\rVert \partial_{b_1} u_{i,j}(R)-\partial_{b_1} u_{i,j}(r) \rVert_{(H^{k}(\mathbb{T}))^2}  \lesssim \rVert R - r \rVert_{(H^{k+1}(\mathbb{T}))^2}.
\end{align}
\end{prop}
If $i,j\ne 1$, then \eqref{growth_b10} and \eqref{lip_b10} follow trivially, since $u_{2,2}$ is independent of $b_1$. As in Proposition~\ref{growth_velocity}, we only deal with the case where $i=j=1$.

\begin{lemma}\label{b_derivative1}
Let $R\in C^\infty(\mathbb{T})$ and $u(R)$ and $z(R)$ be as in \eqref{single1} and \eqref{single2}. For $k\ge 2$, there exists $\epsilon=\epsilon(k,b) >0$ such that if $\rVert R \rVert_{(H^{3}(\mathbb{T}))^2} \le \epsilon$, then
\begin{align}
\rVert \partial_{b} u(R) \rVert_{(H^{k}(\mathbb{T}))^2} \lesssim 1 +\rVert R \rVert_{H^{k+1}(\mathbb{T})}.\label{growth_b1}
\end{align}
Also, if $\rVert R \rVert_{H^{3}(\mathbb{T})},\rVert r \rVert_{H^{3}(\mathbb{T})}  \le \epsilon$ and $\rVert R\rVert_{(H^{k+1}(\mathbb{T}))^2}, \rVert r \rVert_{(H^{k+1}(\mathbb{T}))^2} \le 1$,  then,
\begin{align}\label{lip_b1}
\rVert \partial_{b_1} u(R)-\partial_{b_1} u(r) \rVert_{(H^{k}(\mathbb{T}))^2}  \lesssim \rVert R - r \rVert_{H^{k+1}(\mathbb{T})}.
\end{align}
\end{lemma}
\begin{proof}
From \eqref{f_def}, we have
\begin{align*}
\partial_bu(R)(x) & = \partial_bu_L(R)(x) + \partial_bu_N(R)(x)\\
& = \int_{\mathbb{T}}\log(2-2\cos(x-y))(-\cos y,-\sin y)dy  + \int_{\mathbb{T}}K_4(R)(x,y)z(R)'^\perp(y) dy \\
& \ + \int_{\mathbb{T}}K_5(R)(x,y)(-\cos y,-\sin y)dy\\
& =: B_1 + B_2(R) + B_3(R)
\end{align*}
where  
\begin{align*}
&K_4(R)(x,y) := G(R(x),R(y),J(R)(x,y)),\\
&K_5(R)(x,y) := F(R(x),R(y),J(R)(x,y)),\\
&G(u,v,w) =\partial_bF(u,v,w)= \frac{2b+u+v}{b^2+b(u+v)+uv+w^2},\\
&F(u,v,w)  =\log(b^2+b(u+v)+uv+w^2).
\end{align*} Again, using Lemma~\ref{composition}, we have that for $\rVert R\rVert_{H^{3}(\mathbb{T})},\rVert r \rVert_{H^{3}(\mathbb{T})}\le \epsilon$ for sufficiently small $\epsilon>0$,
\begin{align*}
&\rVert K_4(R) \rVert_{H^{k}(\mathbb{T}^2)}, \rVert K_5(R) \rVert_{H^{k}(\mathbb{T}^2)} \lesssim 1 + \rVert R\rVert_{H^{k+1}(\mathbb{T})},\\
&\rVert K_4(R) - K_4(r) \rVert_{H^{k}(\mathbb{T}^2)}, \rVert K_5(R) - K_5(r) \rVert_{H^{k}(\mathbb{T}^2)} \lesssim (1+\rVert R \rVert_{H^{k+1}(\mathbb{T})}+\rVert r\rVert_{H^{k+1}(\mathbb{T})})\rVert R-r \rVert_{H^{k+1}(\mathbb{T})}.
\end{align*}
Then \eqref{growth_b1} and \eqref{lip_b1} follow immediately from Lemma~\ref{GN}.
\end{proof}

\begin{proofprop}{b_derivative}
The case $i=j=1$ follows immediately from Lemma~\ref{b_derivative1}. If $i\ne j$ then the result follows straightforwardly in a similar manner.
\end{proofprop}

Regarding \eqref{D2G}, \eqref{dtDG} and \eqref{dttDG} (note that since $b_3$ in \eqref{def_b3} depends on $\Theta_3$ we need to estimate the second derivative wit respect to $b_3$ of $G$.), we need to estimate the higher derivatives of $u_{i,j}$. 
\begin{prop}\label{higer_derivative_estimates}
Let  $R$, $z_i(R)$, $u_{i,j}$, be as in \eqref{R_2}, \eqref{z_2} and \eqref{u_def}. For $k\ge 2$, there exists $\epsilon=\epsilon(k,b_1,b_2) >0$ such that if $\rVert R \rVert_{(H^{3}(\mathbb{T}))^2} \le \epsilon$, then
\begin{align}
&\rVert D^2u_{i,j}(R)[h,h] \rVert_{(H^{k}(\mathbb{T}))^2} \lesssim_k (1+ \rVert R \rVert_{(H^{k+2}(\mathbb{T}))^2})\rVert h \rVert_{(H^{k}(\mathbb{T}))^2}^2,\label{second_derivative_1}\\
&\rVert \partial_{b_3}D^2u_{i,j}(R)[h,h] \rVert_{(H^{k}(\mathbb{T}))^2} \lesssim_k (1+ \rVert R \rVert_{(H^{k+2}(\mathbb{T}))^2})\rVert h \rVert_{(H^{k}(\mathbb{T}))^2}^2,\nonumber\\
&\rVert \partial_{b_1}Du_{i,j}(R)[h] \rVert_{(H^{k}(\mathbb{T}))^2} \lesssim_k (1+ \rVert R \rVert_{(H^{k+2}(\mathbb{T}))^2})\rVert h \rVert_{(H^{k}(\mathbb{T}))^2},\nonumber\\
&\rVert \partial_{b_{1}b_1}Du_{i,j}(R)[h] \rVert_{(H^{k}(\mathbb{T}))^2} \lesssim_k (1+ \rVert R \rVert_{(H^{k+2}(\mathbb{T}))^2})\rVert h \rVert_{(H^{k}(\mathbb{T}))^2}.\nonumber
\end{align}
\end{prop}
\begin{proof}
We also consider the $i=j=1$ case only and briefly sketch the idea for \eqref{second_derivative_1}, since the proof is almost identical to Lemma~\ref{BCcondition} and Lemma~\ref{b_derivative1}. Adapting the setting in \eqref{single1} and \eqref{single2}, the Gateaux second derivative of $u$ can be written as a linear combination of the integral operators of the form
\[
\int_{\mathbb{T}}K(R)(x,y)L(x,y)dy,
\]
where $K(R)(x,y)=\tilde{K}(R(x),R(y),J(R)(x,y))$ for some smooth function $\tilde{K}:\RR^3 \mapsto \RR^2$, and $L(x,y)$ is a product of two of $h(x)$, $h(y)$, $h'(y)$ and $J(h)(x,y)$. As before, $K$ satisfies the estimates in \eqref{Kernal_norm2}, \eqref{Kernal_norm4} and \eqref{Kernal_norm6}. Let us consider $L(x,y) = J(h)(x,y)^2$ only, that is,
\[
T(R)[h,h]:=\int_{\mathbb{T}}K(R)(x,y)J(h)(x,y)J(h)(x,y)dy.
\]
Therefore, it follows from Lemma~\ref{twojs} that
\begin{align*}
\rVert T(R)[h,h]\rVert_{H^{k}(\mathbb{T})} \lesssim \rVert K \rVert_{H^{k+1}(\mathbb{T})}\rVert h \rVert_{H^{k}(\mathbb{T})}^2 \lesssim (1+\rVert R \rVert_{H^{k+2}(\mathbb{T})})\rVert h \rVert_{H^{k}(\mathbb{T})}^2,
\end{align*}
where the last inequality follows from the fact that $K$ satisfies \eqref{Kernal_norm2}. As mentioned, the other terms can be estimated in a similar manner and we omit the proofs.
\end{proof}

\subsection{Checking the hypotheses in Theorem~\ref{theorem1}}\label{checking_subsection}
In this subsection we will show the hypotheses in theorem \ref{theorem1} are satisfied. We will always assume that  $2\le k\in \mathbb{N}$ is fixed. We denote $R := (R_1,R_2,R_3) \in \left(X^{k+1}(\mathbb{T})\cap C^\infty(\mathbb{T})\right)^3$. Throughout the rest of the paper, we assume that 
\begin{align}\label{sizeassumption}
|\Theta_3^*-\Theta_3| + \| R \|_{(X^{3})^3} \le \epsilon,
\end{align}
for some $\epsilon$, sufficiently small. Note that from \eqref{parameter_3}, \eqref{def_b3} and  Lemma~\ref{kernel}, it follows that
\begin{align}\label{positive_b3}
0 < b_3(\Theta_3,R)<b_2,
\end{align} for sufficiently small $\epsilon$ in \eqref{sizeassumption}. 

We recall our functional from \eqref{stationary_equation} and \eqref{stationarR_equation4} can be written as (since $b_1,b_2,\Theta_1,\Theta_2$ are fixed, we omit their dependence but mark the dependence of  $\Theta_3$ in the notations)
\begin{align}\label{def_G2}
G(\Theta_3,R) = \colvec{G_1(\Theta_3,R) \\ G_2(\Theta_3,R) \\ G_3(\Theta_3,R)} =\frac{1}{4\pi}\colvec{\sum_{i=1}^3 \Theta_i\left( R_1'u^{\theta}_{i,1} -(b_1+R_1)u^r_{i,1}\right)\\\sum_{i=1}^3 \Theta_i\left(R_2'u^{\theta}_{i,2} -(b_2+R_2)u^r_{i,2}\right)\\ \sum_{i=1}^3\Theta_i\left(R_3' u^{\theta}_{i,3} -(b_3+R_3)u^r_{i,3} \right)} =: \colvec{R_1'u^{\theta}_1  - (b_1+R_1)u^{r}_1 \\ R_2'u^{\theta}_2  - (b_2+R_2)u^{r}_2 \\ R_3'u^{\theta}_3  - (b_3+R_3)u^{r}_3},
\end{align}
where for $i,j,k \in \left\{ 1,2,3 \right\}$, 
\begin{align*}
&u^{\theta}_k :=u^\theta_k(\Theta_3,R) = \frac{1}{4\pi}\sum_{i=1}^3 \Theta_iu^{\theta}_{i,k}, \quad u^{r}_k :=u^r_k(\Theta_3,R) = \frac{1}{4\pi}\sum_{i=1}^3 \Theta_iu^{r}_{i,k}, \\
&u^{\theta}_{i,j} = u^{\theta}_{i,j}(b_3(\Theta_3),\Theta_3,R) =u_{i,j}(b_3(\Theta_3),\Theta_3,R) \cdot \colvec{-\sin x \\ \cos x},\\
&u^{r}_{i,j} = u^{r}_{i,j}(b_3(\Theta_3),\Theta_3,R) =u_{i,j}(b_3(\Theta_3),\Theta_3,R) \cdot \colvec{\cos x \\ \sin x},
\end{align*}
and $u_{i,j}(b_3(\Theta_3),\Theta_3,R)$ is given by  \eqref{z_2} and \eqref{u_def} and $b_3(\Theta_3,R)$ is given by \eqref{def_b3}. We also denote by $\omega(\Theta_3,R), \Psi(\Theta_3,R):\RR^2\mapsto \RR$ the corresponding vorticity and its stream function, more precisely, 
\begin{align}&\omega(\Theta_3,R)(p):= \sum_{i=1}^3\Theta_i 1_{D_i}(p), \quad \Psi(\Theta_3,R)(p):=\frac{1}{2\pi}\left(\omega(\Theta_3,R)*\log|\cdot|\right)(p), \quad p\in \RR^2,\label{vorticity_def}
 \\
&D_i := \left\{ r(\cos x,\sin x)\in \RR^2 : r \le b_i+R_i(x), \ x\in \mathbb{T}\right\}.\label{patch_notation}\end{align}
With the stream function, we can write $G$ as
\begin{align}\label{def_G_stream}
{G(\Theta_3,R)(x)} = -\colvec{\partial_x\left(\Psi(\Theta_3,R)(z_1(R)(x))\right) \\ \partial_x\left(\Psi(\Theta_3,R)(z_2(R)(x))\right) \\ \partial_x\left(\Psi(\Theta_3,R)(z_3(R)(x)) \right)}.
\end{align}

The derivative of the functional will be given as
\begin{align}\label{derivative_matrix_3}
DG(\Theta_3,R)[h] & = \colvec{u_1^{\theta} & 0 & 0 \\ 0 & u_2^\theta & 0 \\ 0 & 0 & u_3^\theta}\colvec{h_1' \\ h_2' \\ h_3'} \nonumber\\
&  \  + \colvec{-h_1u_1^r(\Theta_3,R) + R_1'Du_1^\theta(\Theta_3,R)[h] - (b_1+R_1)Du_1^r(\Theta_3,R)[h]\\ -h_2u_2^r(\Theta_3,R) + R_2'Du_2^\theta(\Theta_3,R)[h] - (b_2+R_2)Du_2^r(\Theta_3,R)[h] \\ -h_3u_3^r(\Theta_3,R) + R_3'Du_3^\theta(\Theta_3,R)[h] - (b_3+R_3)Du_3^r(\Theta_3,R)[h]}\nonumber\\
& \ + Db_3(\Theta_3,R)[h]\colvec{R_1'\partial_{b_3}u^{\theta}_1  - (b_1+R_1)\partial_{b_3}u^{r}_1 \\ R_2'\partial_{b_3}u^{\theta}_2  - (b_2+R_2)\partial_{b_3}u^{r}_2 \\ R_3'\partial_{b_3}u^{\theta}_3  - (b_3+R_3)\partial_{b_3}u^{r}_3 - u^r_3},
\end{align}
where $Db_3(\Theta_3,R)[h]$  is given in \eqref{b_der_R}. We define the projection $P_0:C^{\infty}(\mathbb{T})\mapsto C^{\infty}(\mathbb{T})$ as
\[
P_0 f = f -\frac{1}{2\pi}\int_{\mathbb{T}}f(x)dx,
\]
and the linear maps,
\begin{align}\label{decomp_11}
a(\Theta_3,R)[h]:=\colvec{P_0\left( h_1' u^\theta_1 (\Theta_3,R) \right) \\ 0 \\ 0}, \quad A(\Theta_3,R)[h] := DG(\Theta_3,R)[h] - a(\Theta_3,R)[h].
\end{align}

Now, we start checking the hypotheses. The hypothesis (a) immediately follows from \eqref{trivial_one}, \eqref{stationarR_equation3} and \eqref{stationarR_equation4}.  

\subsubsection{Regularity of the functional $G$}
In this subsection, we use the estimates obtained in Subsection~\ref{velocity_estimates} to prove that the functional $G$ defined in \eqref{def_G2} satisfies the following proposition:
\begin{prop}\label{regularity_checking}
Let $2\le k\in \mathbb{N}$ and $\Theta_3$ and $R$ satisfy \eqref{sizeassumption}. Then, $G:\RR \times (X^{k+1})^{3} \mapsto (Y^k)^3$ is well-defined and \eqref{lineargrowth1}-\eqref{dttDG} hold.
\end{prop}
\begin{proof}
Note that $G_j$ is the tangential derivative of the stream function on the corresponding boundary $D_j$ (see \eqref{def_G_stream}). Thus, $P_0G_j(\Theta,R) = 0$, that is, each $G_j$ has zero mean. If the boundary of the patch $D_i$ is given by by $R_i\in X^{k+1}$, then $\omega(\Theta_3,R)$ is invariant under $\frac{2\pi}{m}$ rotation about the origin and reflection, therefore, $G_j$ is also $\frac{2\pi}{m}$ periodic and odd. Furthermore, it follows from (A) in Proposition~\ref{growth_velocity} and Lemma~\ref{ponce_kato} that
\begin{align*}
\rVert G(\Theta_3,R)\rVert_{(H^{k}(\mathbb{T}))^3} &\lesssim \sum_{i,j=1,2,3} \left( \bigg\rVert u_{i,j}(\Theta_3,R)\cdot \colvec{-\sin x \\ \cos x}R_j'\bigg\rVert_{H^{k}(\mathbb{T})} +  \bigg\rVert u_{i,j}(\Theta_3,R)\cdot \colvec{\cos x \\ \sin x}(b_j+R_j)\bigg\rVert_{H^{k}(\mathbb{T})}\right)\\
&  \lesssim 1+ \rVert R \rVert_{(H^{k+1}(\mathbb{T}))^3},
\end{align*}
which proves \eqref{lineargrowth1} and that $G:\RR \times (X^{k+1})^{3} \mapsto (Y^k)^3$ is well-defined. For \eqref{lineargrowth2} - \eqref{dttDG}, the results follow from (A) and (B) in Proposition~\ref{growth_velocity} and Proposition~\ref{b_derivative} and \ref{higer_derivative_estimates} straightforwardly.
\end{proof}
\subsubsection{The Dirichlet-Neumann operator}
Here, we aim to prove that the linear operator $a(\Theta_3,R)$ in \eqref{decomp_11} satisfies \eqref{approxinverse1}. Thanks to Lemma~\ref{zero mean}, we know that at a stationary solution, the corresponding velocity outside the support of the vorticity must vanish. In the following proposition, we will prove this quantitatively by using the main idea  of \cite{Craig-Schanz-Sulem:modulational-regime-3d-water-waves,Castro-Cordoba-Fefferman-Gancedo-LopezFernandez:rayleigh-taylor-breakdown}.

\begin{prop}\label{a_estimate11}
There exists $\eta>0$ such that if $\rVert R_1 \rVert_{H^{k+2}(\mathbb{T})} \le \eta$, then
 \[
 \| u^\theta_1 (\Theta_3,R) \|_{H^{k}(\mathbb{T})} \lesssim_{k} \|G_1(\Theta_3,R)\|_{H^{k}(\mathbb{T})}.
 \]

\end{prop}
\begin{proof}
We adapt the notations in \eqref{vorticity_def}, \eqref{patch_notation} and \eqref{def_G_stream} for $\omega$, $\Psi$ and $D_i$. Clearly we have (we omit the dependence of $\Theta_3$ and $R$ for simplicity)
\begin{align}\label{uandphi}
u^\theta_1(\alpha) = \nabla^{\perp}\Psi(z_1(\alpha))\cdot \colvec{-\sin \alpha \\ \cos \alpha}, \quad G_1(\alpha) = \nabla \Psi(z_1(\alpha))\cdot z_1'(\alpha), \quad \alpha\in \mathbb{T},
\end{align}
where $z_1$ is as given in \eqref{z_2}. Note that $\Psi$ is harmonic in $D_1^c$. In addition, it follows from \eqref{Theta b relation} that $\int_{\RR^2}\omega(y) dy = 0$, hence, for $x\in \RR^2$,
\begin{align*}
&|\Psi(x) | \lesssim \bigg| \int_{\RR^2} \omega(y) \log\frac{|x-y|}{|x|}dy \bigg| = O\left(\frac{1}{|x|}\right), \\
&| \nabla \Psi (x) | \lesssim \bigg| \int_{\RR^2} \omega(y) \left( \frac{(x-y)}{|x-y|^2}- \frac{x}{|x|^2} \right) dy\bigg| = O\left(\frac{1}{|x|^2}\right).
\end{align*}
With the above decay rates, we can use integration by parts to obtain that for  $x\in \partial D_1$, 
\begin{align}
 0 & = \int_{D_1^c} \log|x-y| \Delta \Psi(y)dy \nonumber\\
 & = -\frac{1}{2}\int_{\partial D_1} \log|x-y|^2 \nabla \Psi(y)\cdot \vec{n}(y)d\sigma(y) +\int_{\partial D_1}\nabla_y(\log|x-y|)\cdot\vec{n}(y)(\Psi(y)-\Psi(x))d\sigma(y)\nonumber\\
 & =:- \frac{1}{2}L_1 + L_2,\label{equation11}
\end{align}
where $\vec{n}$ denotes the outer normal vector on $\partial D_1$. By the change of variables, $x\mapsto z_1(\alpha)$, and $y\mapsto z_1(\beta)$,  we obtain
\begin{align*}
L_1 &=\int_\mathbb{T} \log |z_1(\alpha)-z_1(\beta)|^2 \nabla \Psi(z_1(\beta))\cdot (-z'_1(\beta)^\perp)d\beta \\
& = \int_{\mathbb{T}} \log (2-2\cos(\alpha-\beta)) \nabla\Psi(z_1(\beta))\cdot (-z'_1(\beta)^\perp)d\beta \\
& \  + \int_\mathbb{T} \log\left(\frac{|z_1(\alpha)-z_1(\beta)|^2}{2-2\cos(\alpha-\beta)} \right) \nabla \Psi(z_1(\beta))\cdot (-z'_1(\beta)^\perp)d\beta\\
& =: L_{11} + L_{12}.
\end{align*}
Similarly, we have
\begin{align*}
L_2 = \int_\mathbb{T} \frac{(z_1(\alpha)-z_1(\beta))\cdot z_1'(\beta)^\perp}{|z_1(\alpha)-z_1(\beta)|^2}(\Psi(z_1(\beta)) - \Psi(z_1(\alpha)) d\beta.
\end{align*}
Hence, we obtain
\begin{align}\label{lidentity}
L_{11} = -L_{12} + 2 L_2.
\end{align}
We claim that 
\begin{align}
&\| L_{11} \|_{H^{k+1}(\mathbb{T})} \gtrsim \| \nabla\Psi(z_1(\cdot))\cdot (-z'_1(\cdot)^\perp)\|_{H^{k}(\mathbb{T})} \label{Lestimate1} \\
& \| L_{12} \|_{H^{k+1}(\mathbb{T})} \lesssim \| R_1 \|_{H^{k+2}(\mathbb{T})} \| \nabla\Psi(z_1(\cdot))\cdot (-z'_1(\cdot)^\perp)\|_{L^2(\mathbb{T})} \label{Lestimate2} \\
& \|L_2\|_{H^{k+1}(\mathbb{T})} \lesssim \| \nabla\Psi(z_1(\cdot))\cdot (z'_1(\cdot))\|_{H^k(\mathbb{T})}. \label{Lestimate3}
\end{align}
Let us assume the above claims for a moment.  Then it follows from the claims, \eqref{sizeassumption} and \eqref{lidentity} that for sufficiently small $\eta>0$ in the hypothesis of the proposition, we have
\[
\| \nabla\Psi(z_1(\cdot))\cdot (-z'_1(\cdot)^\perp)\|_{H^k(\mathbb{T})}  \lesssim \| \nabla\Psi(z_1(\cdot))\cdot (z'_1(\cdot))\|_{H^k(\mathbb{T})}.
\]
With this inequality, we can obtain
\begin{align*}
\| \nabla \Psi(z_1(\cdot)) \|_{H^{k}(\mathbb{T})}  & = \bigg\rVert \nabla \Psi(z_1(\cdot))\cdot \frac{z'_1(\cdot)^\perp}{|z_1'(\cdot)|} \bigg\lVert_{H^{k}(\mathbb{T})} +  \bigg\rVert \nabla \Psi(z_1(\cdot))\cdot \frac{z'_1(\cdot)}{|z_1'(\cdot)|}  \bigg\rVert_{H^{k}(\mathbb{T})}\\
& \lesssim \rVert z_1'\rVert_{H^{k}(\mathbb{T})}\left(  \| \nabla \Psi(z_1(\cdot))\cdot {z'_1(\cdot)^\perp} \|_{H^{k}(\mathbb{T})} +  \| \nabla \Psi(z_1(\cdot))\cdot {z'_1(\cdot)} \|_{H^{k}(\mathbb{T})}\right) \\
& \lesssim   \| \nabla \Psi(z_1(\cdot))\cdot {z'_1(\cdot)} \|_{H^{k}(\mathbb{T})},
\end{align*}
where the first  inequalities follows from $\frac{1}{c} < |z_1'| < c$ for some $c>0$ and the second inequality follows from $\rVert z_1' \rVert_{H^{k}(\mathbb{T})}\lesssim \rVert R_1 \rVert_{H^{k+1}(\mathbb{T})}\le \eta$ for some small $\eta$. under the assumption~\eqref{sizeassumption}. Finally, recalling \eqref{uandphi}, we obtain the desired result.

 To finish the proof, we need to prove the claims \eqref{Lestimate1}-\eqref{Lestimate3}. \eqref{Lestimate1} follows from Lemma~\ref{T1estimate}. To see \eqref{Lestimate2}, we observe that
 \begin{align*}
 \log\left(\frac{|z_1(\alpha)-z_1(\beta)|^2}{2-2\cos(\alpha-\beta)} \right) = \log(1+\mathcal{K}(\alpha,\beta)),
 \end{align*}
 where $\mathcal{K}(\alpha,\beta):={2(R_1(\alpha) + R_1(\beta))} + R_1(\alpha)R_1(\beta) + \left( \frac{R_1(\alpha)-R_1(\beta)}{2\sin(\frac{\alpha-\beta}{2})} \right)^2$. Thus, it is straightforward  that (thanks to \eqref{sizeassumption}, $1+\mathcal{K} \ge c>0$ for some $c>0$)
 \[
\| \log(1+\mathcal{K}) \|_{H^{k+1}(\mathbb{T}^2)} \lesssim \| \mathcal{K} \|_{H^{k+1}(\mathbb{T}^2)}  \lesssim \| R_1 \|_{H^{k+2}(\mathbb{T})},
 \]
 where the last inequality follows from Lemma~\ref{ponce_kato} and Lemma~\ref{appendix_lem_1}. Thus, 
 \begin{align*}
 \| L_{12} \|_{H^{k+1}(\mathbb{T})} &\lesssim  \bigg\rVert \log\left(\frac{|z_1(\alpha)-z_1(\beta)|^2}{2-2\cos(\alpha-\beta)} \right)  \bigg\rVert_{H^{k+1}(\mathbb{T}^2)}  \| \nabla\Psi(z_1(\cdot))\cdot (-z'_1(\cdot)^\perp)\|_{L^2(\mathbb{T})} \\
 & \lesssim \| R_1 \|_{H^{k+2}(\mathbb{T}^2)} \| \nabla\Psi(z_1(\cdot))\cdot (-z'_1(\cdot)^\perp)\|_{L^2(\mathbb{T})},
 \end{align*}
 which yields \eqref{Lestimate2}. Lastly, in order to show \eqref{Lestimate3},  we rewrite $L_2$ as
 \[
 L_2 = \int_{\mathbb{T}}g(\alpha,\beta)\frac{((\Psi(z_1(\beta))-M_\Psi) - (\Psi(z_1(\alpha))-M_\Psi)}{2\sin(\frac{\alpha-\beta}{2})} d\beta,
 \]
 where 
 \[
 g(\alpha,\beta) = \frac{(z_1(\alpha)-z_1(\beta))\cdot z_1'(\beta)^\perp}{2\sin(\frac{\alpha-\beta}{2})} \cdot \frac{2-2\cos(\alpha-\beta)}{|z_1(\alpha)-z_1(\beta)|^2}, \quad M_\Psi:=\frac{1}{2\pi}\int_{\mathbb{T}}\Psi(z_1(\alpha))d\alpha.
\]
From Lemma~\ref{appendix_lem_1}, and \ref{ponce_kato}, we have 
\[
\| g \|_{H^{k+1}(\mathbb{T}^2)} \lesssim \| R_1 \|_{H^{k+2}(\mathbb{T})}.
\]
Therefore, it follows from Lemma~\ref{J_linear} that
\[
\| L_2 \|_{H^{k+1}(\mathbb{T})} \lesssim \|(\Psi(z_1(R,\cdot))-M_\Psi) \|_{H^{k+1}(\mathbb{T})} \lesssim \| \nabla\Psi(z_1(R,\cdot))\cdot (z'_1(R,\cdot))\|_{H^k(\mathbb{T})},
\]
where the last inequality follows from Poincar\'e inequality. Hence \eqref{Lestimate3} follows.  
\end{proof}

\subsubsection{Estimates on the linearized operator}
Our goal here is to prove that $G$ satisfies the hypotheses $(c)$, $(\tilde{c}-1)$ and $(\tilde{c}-2)$ in Theorem~\ref{theorem1}. More precisely, we will prove the following proposition:

\begin{prop}\label{linear_prop}
Let $a(\Theta_3,R)$ and $A(\Theta_3,R)$ be as in \eqref{decomp_11}. Then 
\begin{enumerate}
\item[\rom{1})] $a(\Theta_3,R)\in \mathcal{L}((X^{k+1})^3,(Y^{k})^3)$, $A(\Theta_3,R) \in \mathcal{L}((X^{k+1})^3,Y^{k+1}\times (Y^{k})^2)$ with the estimates \eqref{approxinverse1} and \eqref{approxinverse2}.
\item[\rom{2})]  For $(\Theta_3^1,R^1),\ (\Theta_3^2,R^2)\in I\times (V^3\cap (C^\infty)^3),$ \eqref{NM_lip} holds.
\item[\rom{3})]  For any even $\sigma \in \mathbb{N}\cup \left\{ 0 \right\}$, there exists $0<\eta<1$ such that if $\rVert R \rVert_{(H^{k+3}(\mathbb{T}))^3} \le \eta$, and $A(\Theta_3,R)[h] =z$ for some  $z\in (C^\infty)^3$ and $h \in \text{Ker}(A(\Theta^*_3,0))^{\perp}$, then \eqref{tame1} holds.
\end{enumerate}
\end{prop}
\begin{proof}
\textbf{Proof of \rom{1}).} By definition of $a(\Theta_3,R)$ in \eqref{decomp_11}, we have $\int_{\mathbb{T}}a(\Theta_3,R)[h](x)dx = 0 $. Furthermore, clearly, $h_1'$ is a $\frac{2\pi}{m}$-periodic function and odd. Using the  invariance under rotation/reflection of the stream function, it follows straightforwardly that $u^\theta_1$, which is the radial derivative of the stream function on the outmost boundary, is also $\frac{2\pi}{m}$-periodic and even. Therefore, $a(\Theta_3,R)[h]$ is $\frac{2\pi}{m}$ periodic and odd. \eqref{approxinverse1} follows immediately from \eqref{a_estimate11}, and $a(\Theta_3,R)\in \mathcal{L}((X^{k+1})^3,(Y^{k})^3)$. Similarly, $A(\Theta_3,R) \in \mathcal{L}((X^{k+1})^3,Y^{k+1}\times (Y^{k})^2)$ and \eqref{approxinverse2} follows from (B) in Proposition~\ref{growth_velocity}.

\textbf{Proof of \rom{2}).} Thanks to (C) in Proposition~\ref{growth_velocity} and Proposition~\ref{b_derivative}, each term in $A(\Theta_3,R)$ is Lipschitz continuous with respect to $R\in (H^{k+3}(\mathbb{T})^3)$. Furthermore, $A(\theta_3,R)$ and $b_3$ depend on $\Theta_3$ smoothly, therefore the result follows immediately.

\textbf{Proof of \rom{3}).} In order to prove $\rom{3}$, we first claim that for each even $\sigma \in \mathbb{N}\cup \left\{ 0 \right\}$, there exist $\eta$ and  a linear map $T(R):(C^\infty(\mathbb{T}))^3 \mapsto (C^{\infty}(\mathbb{T}))^3$ such that if $\rVert R \rVert_{(H^{k+3}(\mathbb{T}))^3}  \le \eta$, then 
\begin{align}
& \left(\frac{d}{dx}\right)^\sigma A(\Theta_3,R)[h] = A(\Theta_3,R)[h^{(\sigma)}] + L^\sigma(R)[h], \label{claim_Tsigma}\\
&\rVert L^\sigma(R)[h]\rVert_{H^{k+1}(\mathbb{T}) \times (H^{k}(\mathbb{T}))^2} \lesssim_{k,\sigma} \rVert h \rVert_{(H^{k+\sigma}(\mathbb{T}))^3} + (1 + \rVert R \rVert_{(H^{k+4+\sigma}(\mathbb{T}))^3})\rVert h \rVert_{(H^{k+1}(\mathbb{T}))^3}.\label{claim_Tsigma2}
\end{align}
Let us assume that the above claims are true for a moment and let us suppose  
\[
A(\Theta_3,R)[h] = z, \quad \text{ for some }z\in (C^\infty(\mathbb{T}))^3,\ h\in \text{Ker}(A(\Theta_3^*,0))^{\perp}\subset (H^{k+1}(\mathbb{T}))^3.
\]
From \eqref{sizeassumption}, $\rom{2})$, Lemma~\ref{functional_stability} and the assumption that $\rVert R \rVert_{(H^{k+3}(\mathbb{T}))^3}\le \eta$ for some small enough $\eta$, we have that $A(\Theta_3,R):\text{Ker}(A(\Theta_3^*,0))^\perp\subset ((H^{k+1}(\mathbb{T}))^3) \mapsto \text{Im}(A(\Theta,R))\subset H^{k+1}(\mathbb{T})\times (H^k(\mathbb{T}))^2$ is invertible and 
\begin{align}\label{invert_11}
\rVert A(\Theta_3,R)^{-1} \rVert_{\mathcal{L}(\text{Im}(A(\Theta_3,R)),\text{Ker}(A(\Theta_3^*,0))^\perp)} \lesssim1.
\end{align}
Therefore, we have
\begin{align}\label{low_invert}
 \rVert h \rVert_{(H^{k+1}(\mathbb{T}))^3} \lesssim \rVert z \rVert_{H^{k+1}(\mathbb{T})\times (H^{k}(\mathbb{T}))^2}.
\end{align}
Now for each even $\sigma\in \mathbb{N}\left\{ 0 \right\}$, it follows from \eqref{claim_Tsigma}  that
\begin{align*}
\left(\frac{d}{dx}\right)^\sigma A(\Theta_3,R)[h] = A(\Theta_3,R)[h^{(\sigma)}] + L^\sigma(R)[h] = z^{(\sigma)}.
\end{align*}
Thanks to Proposition~\ref{Fredholm}, we have $h^{(\sigma)}\in \text{Ker}(A(\Theta_3^*,0))^\perp$, thus, it follows from \eqref{invert_11} that
\[
h^{(\sigma)} = A(\Theta_3,R)^{-1}[z^{(\sigma)}- L^\sigma(R)[h]],
\]
and
\begin{align}\label{h_high_estimate}
\rVert h^{(\sigma)}\rVert_{(H^{k+1}(\mathbb{T}))^3} &\lesssim \rVert z^{(\sigma)}  - L^\sigma(R)[h] \rVert_{H^{k+1}(\mathbb{T})\times (H^{k}(\mathbb{T}))^2} \nonumber\\
& \lesssim \rVert z \rVert_{H^{k+1+\sigma}(\mathbb{T}) \times (H^{k+\sigma}(\mathbb{T}))^2} +  \rVert L^\sigma(R)[h] \rVert_{H^{k+1}(\mathbb{T}) \times (H^{k}(\mathbb{T}))^2}.
\end{align}
From \eqref{claim_Tsigma2}, we also have
\begin{align}\label{Lsigma_estimate}
\rVert L^\sigma(R)[h] \rVert_{H^{k+1}(\mathbb{T}) \times (H^{k}(\mathbb{T}))^2}  & \lesssim \rVert h \rVert_{(H^{k+\sigma}(\mathbb{T}))^3} + (1 + \rVert R \rVert_{(H^{k+4+\sigma}(\mathbb{T}))^3})\rVert h \rVert_{(H^{k+1}(\mathbb{T}))^3}.
\end{align}
Using Lemma~\ref{GNinterpolation}, we have 
\begin{align*}
\rVert h \rVert_{(H^{k+\sigma}(\mathbb{T}))^3}  & \lesssim \rVert h \rVert_{(H^{k+1}(\mathbb{T}))^3} + \bigg\rVert\left( h^{(k+1)}\right)^{(\sigma-1)} \bigg\rVert_{(L^2(\mathbb{T}))^3}\\
& \lesssim \rVert h \rVert_{(H^{k+1}(\mathbb{T}))^3}  + \left(\rVert  h^{(k+1)} \rVert_{(L^2(\mathbb{T}))^3}\right)^{\frac{1}\sigma} \bigg\rVert \left(h^{(k+1)}\right)^{(\sigma)} \bigg\rVert_{(L^2{(\mathbb{T}))^3}}^{\frac{\sigma-1}{\sigma}} \\
& \lesssim  (1+C(\epsilon))\rVert  \rVert h \rVert_{(H^{k+1}(\mathbb{T}))^3} + \epsilon \rVert h^{(\sigma)} \rVert_{(H^{k+1}(\mathbb{T}))^3},
\end{align*}
for any $\epsilon>0$, which follows from Young's inequality. Plugging this into \eqref{Lsigma_estimate} and using \eqref{low_invert}, we obtain
\[
\rVert L^\sigma(R)[h] \rVert_{H^{k+1}(\mathbb{T}) \times (H^{k}(\mathbb{T}))^2}  \lesssim (C(\epsilon) +  \rVert R \rVert_{(H^{k+4+\sigma}(\mathbb{T}))^3}) \rVert z \rVert_{H^{k+1}(\mathbb{T})\times (H^{k}(\mathbb{T}))^2} + \epsilon \rVert h^{(\sigma)} \rVert_{(H^{k+1}(\mathbb{T}))^3}
\]
Hence we choose sufficiently small $\epsilon$ depending on $k$ and $\sigma$ so that \eqref{h_high_estimate} yields that
\[
\rVert h^{(\sigma)}\rVert_{(H^{k+1}(\mathbb{T}))^3}  \lesssim \rVert z \rVert_{H^{k+1+\sigma}(\mathbb{T}) \times (H^{k+\sigma}(\mathbb{T}))^2} +  (1 +  \rVert R \rVert_{(H^{k+4+\sigma}(\mathbb{T}))^3}) \rVert z \rVert_{H^{k+1}(\mathbb{T})\times (H^{k}(\mathbb{T}))^2}.
\]
With this estimate and \eqref{low_invert}, \eqref{tame1} follows.

 In order to finish the proof, we need to prove the claim \eqref{claim_Tsigma} and \eqref{claim_Tsigma2}, we note that each component of $A(\Theta_3,R)[h]$ consists of linear combination of the following forms (see \eqref{b_der_R}, \eqref{derivative_matrix_3} and \eqref{decomp_11}):
\begin{align}
&u^\theta_i(R)h_i', \text{ for }i=2,3,\nonumber\\
&h_i u_i^r(\Theta_3,R),\ R_i'Du^\theta_i(\Theta_3,R)[h], \ (b_i+R_i)Du_1^r(\Theta_3,R)[h],  \ \text{ for }i=1,2,3.\label{terms1}\\
& \int_{\mathbb{T}}R_i(x)h_i(x)dx\left(R_i'\partial_{b_3}u^\theta_i - (b_i+R_i)\partial_{b_3}u_i^r \right), \ \text{ for i=1,2}\label{terms2}\\
& \int_{\mathbb{T}}R_3(x)h_3(x)dx\left(R_3'\partial_{b_3}u^\theta_3 - (b_3+R_3)\partial_{b_3}u_3^r-u_3^r \right).\label{terms3}
\end{align}
For $i=2,3$, it is clear that 
\begin{align}\label{L_1}
\left(\frac{d}{dx}\right)^{\sigma}\left( u^{\theta}_i(R)h'_i\right) = u^\theta_ih^{(\sigma+1)}_i + \sum_{p+q=\sigma,\ q\le \sigma-1} C_{p,q,\sigma}\underbrace{\left( u^\theta_i(R)\right)^{(p)}h_i^{(q+1)}}_{L^\sigma_1(R)[h_i]},
\end{align}
for some $C_{p,q,\sigma}$. From Lemma~\ref{pq_derivatives}, we have
\begin{align}
\rVert L^\sigma_1(R)[h_i]\rVert_{H^{k}(\mathbb{T})} &\lesssim \rVert u_i^{\theta}(R)\rVert_{H^2(\mathbb{T})}\rVert h_i\rVert_{(H^{k+\sigma}(\mathbb{T}))^3} + \rVert u_i^{\theta}(R) \rVert_{H^{k+\sigma}(\mathbb{T})}\rVert h_i \rVert_{(H^1(\mathbb{T}))^3}\nonumber\\
& \lesssim \rVert h_i \rVert_{H^{k+\sigma}(\mathbb{T})} + (1 + \rVert R \rVert_{(H^{k+1+\sigma}(\mathbb{T}))^3})\rVert h_i \rVert_{H^1(\mathbb{T})},\label{L_2}
\end{align}
where the second inequality follows from Proposition~\ref{growth_velocity}.
For the other terms, we only deal with $R_1'Du_1^{\theta}(\Theta_3,R)[h]$, since the other terms can be dealt with in the same way. Thus, we will show that there exists a linear operator $L_2^\sigma(R)[h]$ such that
\begin{align}
& \left(\frac{d}{dx}\right)^\sigma \left(R_1'Du_1^{\theta}(\Theta_3,R)[h] \right) = R_1'Du_1^{\theta}(\Theta_3,R)[h^{(\sigma)}] + L_2^\sigma(R)[h], \label{claim_Tsigma3}\\
&\rVert L_2^\sigma(R)[h]\rVert_{H^{k+1}(\mathbb{T})} \lesssim_{k,\sigma} \rVert h\rVert_{(H^{k+\sigma}(\mathbb{T}))^{3}} + \rVert R \rVert_{(H^{k+4+\sigma}(\mathbb{T}))^3}\rVert h \rVert_{(H^3(\mathbb{T}))^3}. \label{claim_Tsigma4}
\end{align}
It follows from (D) in Proposition~\ref{growth_velocity} that there exists $T^\sigma(R)[h]$ such that 
\begin{align*}
\left(\frac{d}{dx}\right)^{\sigma} \left(R_1'Du_1^{\theta}(\Theta_3,R)[h] \right) & = R_1' \left(Du_1^{\theta}(\Theta_3,R)[h]\right)^{(\sigma)} + \sum_{p+q=\sigma,\ q\le \sigma-1}(R_1')^{(p)}\left(Du_1^{\theta}(\Theta_3,R)[h]\right)^{(q)}\\
& = R_1'Du_1^{\theta}(\Theta_3,R)[h^{(\sigma)}] +\underbrace{ R_1'T^\sigma(R)[h] + \sum_{p+q=\sigma,\ q\le \sigma-1}C_{p,q,\sigma}(R_1')^{(p)}\left(Du_1^{\theta}(\Theta_3,R)[h]\right)^{(q)}}_{=:L^\sigma_2(R)[h]},
\end{align*}
and
\[
\rVert T^\sigma(R)[h]\rVert_{H^{k+1}(\mathbb{T})} \lesssim (1+\rVert R \rVert_{(H^4(\mathbb{T}))^3})\rVert h^{(k+\sigma)} \rVert_{(L^2(\mathbb{T}))^3} + (1+ \rVert R \rVert_{(H^{k+4+\sigma}(\mathbb{T}))^3})\rVert h\rVert_{(L^2(\mathbb{T}))^3} .
\]

Therefore, from (B) in Proposition~\ref{growth_velocity} and Lemma~\ref{pq_derivatives}, it follows that
\begin{align}\label{rprime}
\rVert R_1' T^{\sigma}(R)[h]\rVert_{H^{k+1}(\mathbb{T})} &\lesssim \rVert R_1' \rVert_{H^{k+1}(\mathbb{T})} \rVert T^{\sigma}(R)[h]\rVert_{H^{k+1}(\mathbb{T})} \nonumber\\
& \lesssim \rVert h^{(k+\sigma)} \rVert_{(L^2(\mathbb{T}))^3} + (1+ \rVert R \rVert_{(H^{k+4+\sigma}(\mathbb{T}))^3})\rVert h\rVert_{(L^2(\mathbb{T}))^3},
\end{align}
where we used that $H^{k+1}(\mathbb{T})$ is a Banach algebra and that $\rVert R \rVert_{(H^{4}(\mathbb{T}))^3}\lesssim \rVert R \rVert_{(H^{k+3}(\mathbb{T}))^3} \le \eta$ for some $0 <\eta < 1$. From Lemma~\ref{pq_derivatives}, we also have
\begin{align*}
\rVert (R_1')^{(p)}\left(Du_1^{\theta}(\Theta_3,R)[h]\right)^{(q)} \rVert_{H^{k+1}(\mathbb{T})}& \lesssim \rVert R_1 \rVert_{H^2(\mathbb{T})} \rVert Du_1^{\theta}(\Theta_3,R)[h] \rVert_{H^{k+\sigma}(\mathbb{T})} + \rVert R_1 \rVert_{H^{k+1+\sigma}(\mathbb{T})}\rVert Du_1^{\theta}(\Theta_3,R)[h] \rVert_{H^1(\mathbb{T})} \\
& \lesssim \rVert Du_1^{\theta}(\Theta_3,R)[h] \rVert_{H^{k+\sigma}(\mathbb{T})} + \rVert R_1 \rVert_{H^{k+1+\sigma}(\mathbb{T})}\rVert Du_1^{\theta}(\Theta_3,R)[h] \rVert_{H^3(\mathbb{T})} \\
& \lesssim \rVert h \rVert_{(H^{k+\sigma}(\mathbb{T}))^3} + \rVert R \rVert_{(H^{k+3 + \sigma}(\mathbb{T}))^3}\rVert h \rVert_{(H^1(\mathbb{T}))^3} \\
&  \ + \rVert R_1 \rVert_{H^{k+1+\sigma}(\mathbb{T})}\rVert Du_1^{\theta}(\Theta_3,R)[h] \rVert_{H^3(\mathbb{T})} 
\end{align*}
where the last inequality follows from (B) in Proposition~\ref{growth_velocity},which also tells us that
\begin{align*}
\rVert Du_1^{\theta}(\Theta_3,R)[h] \rVert_{H^3(\mathbb{T})} & \lesssim (1+\rVert R \rVert_{(H^{5}(\mathbb{T}))^3})\rVert h\rVert_{(H^{3}(\mathbb{T}))^3}\\
& \lesssim (1 + \rVert R \rVert_{(H^{k+3+\sigma}(\mathbb{T}))^3})\rVert h \rVert_{(H^{3}(\mathbb{T})^3)},
\end{align*}
where we used $k\ge 2$.
Therefore, we have
\[
\rVert (R_1')^{(p)}\left(Du_1^{\theta}(\Theta_3,R)[h]\right)^{(q)} \rVert_{H^{k+1}(\mathbb{T})} \lesssim  \rVert h \rVert_{(H^{k+\sigma}(\mathbb{T}))^3}  + (1 + \rVert R \rVert_{(H^{k+3 + \sigma}(\mathbb{T}))^3})\rVert h \rVert_{(H^{3}(\mathbb{T}))^3}.
\]
Combining with \eqref{rprime}, we obtain \eqref{claim_Tsigma4}. Since every term in \eqref{terms1}, \eqref{terms2} and \eqref{terms3} can be shown to satisfy the same property as in \eqref{claim_Tsigma3} and \eqref{claim_Tsigma4}, combining with \eqref{L_1} and \eqref{L_2},  we obtain \eqref{claim_Tsigma} and \eqref{claim_Tsigma2}.

\end{proof}

 \subsubsection{Spectral Study}\label{Spectral}
 
 In this subsection, we will verify the hypotheses $(d)$ and $(e)$. They will be proved in Proposition~\ref{Fredholm} and Lemma~\ref{transversality_11} respectively.

\begin{prop}\label{Fredholm}
Let $2\le m\in \mathbb{N}.$  We can choose $b_1=\Theta_1 = 1$ and $b_2$ and $\Theta_2$ be as in \eqref{parameter_3}. Also we set $b_3$ as in \eqref{def_b3}. Then, there exists $\Theta^*_3 > 0$ such that $A(\Theta^*_3,0)\in  \mathcal{L}((H^{k+1})^3,H^{k+1}\times (H^{k})^2)$ has one-dimensional kernel and codimension one. That is, 
\begin{align*}
\text{Ker}(A(\Theta^*_3,0)) = \text{span}\left\{ v\right\}, \quad \text{Im}(A(\Theta^*_3,0))^{\perp} = \text{span}\left\{w\right\}.
\end{align*}
Furthermore, $0< b_3(\Theta^*_3,0) < b_2 < b_1=1$ and $v$ and $w$ are supported on the $m$-th Fourier mode. 
\end{prop}

The proof of the above proposition will be achieved after proving several lemmas. 

Let us first compute the derivative of $G$ at $(\Theta_3,0)$.  Since we have $G(\Theta_3, 0)=0$ for any $\Theta_3$, $\rom{1}$ in Proposition~\ref{linear_prop} yields that $a(\Theta_3, 0)=0$ for any $\Theta_3$. Moreover, it follows from \eqref{derivative_matrix_3}, \eqref{b_der_R} and Lemma~\ref{linearpart2} that letting $h_1(x) = \sum_{n\geq 1}h_{1,n} \cos(mnx),$ $h_2(x) = \sum_{n\geq 1}h_{2,n}\cos(mnx)$, $h_3(x) = \sum_{n\geq 1}h_{3,n}\cos(mnx)$, 

\begin{align*}
A(\Theta_3,0)[h_1,h_2,h_3]  = \left(\begin{array}{c}Q_1(x) \\ Q_2(x) \\ Q_3(x)\end{array}\right),
\end{align*}

where

\begin{align*}
Q_1(x) = \sum_{n\ge 1} q_{1,n} \sin(mnx), \quad Q_2(x) = \sum_{n\ge 1}q_{2,n} \sin(mnx), \quad Q_3(x) = \sum_{n\ge 1}q_{3,n} \sin(mnx)
\end{align*}
where the coefficients satisfy, for any $n$:

\begin{align}\label{matrix_form1}
(-mn) M_n(\Theta_3,b_3)
\left(\begin{array}{c}h_{1,n} \\ h_{2,n} \\ h_{3,n} \end{array}\right)
= 
\left(\begin{array}{c}q_{1,n} \\ q_{2,n} \\ q_{3,n} \end{array}\right),
\end{align}
with
\begin{align}
 & M_n(\Theta_3,b_3) = \nonumber \\
& \left(
\begin{array}{ccc}
 \left(-\frac12 + \frac{1}{2mn}\right) - \frac12 \Theta_2 b_2^2 - \frac12  \Theta_3 b_3^2 & \Theta_2 b_2 \frac{b_2^{mn}}{2mn} & \Theta_3 b_3 \frac{b_3^{mn}}{2mn}\\
\frac{b_2^{mn}}{2mn} & \Theta_2 b_2\left(-\frac12 + \frac{1}{2mn}\right) + \left(-\frac12 b_2\right) + \left(-\frac12\right)\Theta_3 \frac{b_3^2}{b_2} & \Theta_3 b_3 \frac{1}{2mn}\left(\frac{b_3}{b_2}\right)^{mn} \\
\frac{b_3^{mn}}{2mn} & \Theta_2 b_2 \frac{1}{2mn}\left(\frac{b_3}{b_2}\right)^{mn} &  \Theta_3 b_3\left(-\frac12 + \frac{1}{2mn}\right) + \left(-\frac12 b_3\right) -\frac{\Theta_2 b_3}{2}
\end{array}
\right).
\label{mn_formula}
\end{align}


\begin{lemma}\label{kernel}
Let $2\le m\in \mathbb{N}$, and $b_2,$ $\Theta_2$ satisfy \eqref{parameter_3}. Then, there exists $\Theta^*_3 := \Theta^*_{3,m}> 0$ and $b_3$ such that  $0< b_3< b_2,  $ and 
\begin{align}
&\text{det}(M_1)=0, \  \Theta_1b_1^2 +\Theta_2b_2^2+\Theta_3^*b_3^2 = 1+\Theta_2b_2^2+\Theta_3^*b_3^2=0,\label{singular}\\
&\text{det}(M_n(\Theta^*_3,b_3))\neq0, \quad \text{ if $n \geq 2$.}\label{nonsingular1}
\end{align}
\end{lemma}
\begin{proof}
We first write $\text{det}(M_1)$ as a polynomial in $b_3$. Under the constraint $1+\Theta_2b_2^2+\Theta_3b_3^2=0$, we get
\begin{align}
\label{M_1}
  M_{1}(\Theta_3,b_3) & =  \small \left(
\begin{array}{ccc}
 \frac{1}{2m}& \Theta_2 b_2 \frac{b_2^{m}}{2m} & \Theta_3 b_3 \frac{b_3^{m}}{2m}\\
\frac{b_2^{m}}{2m} & \Theta_2 b_2\left(-\frac12 + \frac{1}{2m}\right) + \left(-\frac12 b_2\right) + \left(-\frac12\right)\Theta_3 \frac{b_3^2}{b_2} & \Theta_3 b_3 \frac{1}{2m}\left(\frac{b_3}{b_2}\right)^{m} \\
\frac{b_3^{m}}{2m} & \Theta_2 b_2 \frac{1}{2m}\left(\frac{b_3}{b_2}\right)^{m} &  \Theta_3 b_3\left(-\frac12 + \frac{1}{2m}\right) + \left(-\frac12 b_3\right) + \left(-\frac12\right)\Theta_2 b_3
\end{array}
\right).
\end{align}
Therefore, multiplying the first row by $2m$, multiplying the third column by $b_3$ and dividing the second column by $b_2$, we get
\begin{align*}
 &\frac{2mb_3}{b_2} \text{det}(M_{1}(\Theta_3,b_3))  =\text{det} \small \left(
\begin{array}{ccc}
  1& \Theta_2  b_2^{m} & \Theta_3b_3^2  b_3^{m}\\
\frac{b_2^{m}}{2m} & \Theta_2 \left(-\frac12 + \frac{1}{2m}\right) + \left(-\frac12\right) + \left(-\frac12\right)\Theta_3 \frac{b_3^2}{b_2^2} & \Theta_3  \frac{1}{2m}b_3^2\left(\frac{b_3}{b_2}\right)^{m} \\
\frac{b_3^{m}}{2m} & \Theta_2  \frac{1}{2m}\left(\frac{b_3}{b_2}\right)^{m} &  \Theta_3 b_3^2\left(-\frac12 + \frac{1}{2m}\right) + \left(-\frac12 b_3^2\right) + \left(-\frac12\right)\Theta_2 b_3^2
\end{array}
\right)\\
&=\text{det}\small \left(
\begin{array}{ccc}
  1& \Theta_2  b_2^{m} & (-1-\Theta_2b_2^2)  b_3^{m}\\
\frac{b_2^{m}}{2m} & \Theta_2 \left(-\frac12 + \frac{1}{2m}\right) + \left(-\frac12\right) + \left(-\frac12\right)\frac{(-1-\Theta_2 b_2^2)}{b_2^2} & (-1-\Theta_2b_2^2)  \frac{1}{2m}\left(\frac{b_3}{b_2}\right)^{m} \\
\frac{b_3^{m}}{2m} & \Theta_2  \frac{1}{2m}\left(\frac{b_3}{b_2}\right)^{m} &  (-1-\Theta_2b_2^2)\left(-\frac12 + \frac{1}{2m}\right) + \left(-\frac12 b_3^2\right) + \left(-\frac12\right)\Theta_2 b_3^2
\end{array}
\right)\\
&=\text{det}\small \left(
\begin{array}{ccc}
  1& \Theta_2  b_2^{m} & (-1-\Theta_2b_2^2)  b_3^{m}\\
\frac{b_2^{m}}{2m} & \frac{\Theta_2}{2m}-\frac{1}{2}+\frac{1}{2b_2^2} & (-1-\Theta_2b_2^2)  \frac{1}{2m}\left(\frac{b_3}{b_2}\right)^{m} \\
\frac{b_3^{m}}{2m} & \Theta_2  \frac{1}{2m}\left(\frac{b_3}{b_2}\right)^{m} &  (-1-\Theta_2b_2^2)\left(-\frac12 + \frac{1}{2m}\right) + \left(-\frac12 -\frac12\Theta_2\right) b_3^2
\end{array}
\right)\\
& =: \text{det}
\begin{pmatrix}
a_{11} & a_{12} & a_{13}\\ a_{21} & a_{22} & a_{23} \\ a_{31} & a_{32} & a_{33}
\end{pmatrix},
\end{align*}
where we used the condition $1+\Theta_2b_2^2+\Theta_3b_3^2=0$ in the second equality to get rid of $\Theta_3$.  We further compute
\begin{align*}
\frac{2mb_3}{b_2}& \text{det}(M_{1}(\Theta_3,b_3))  = a_{31} (a_{12}a_{23} - a_{13}a_{22}) - a_{32}(a_{11}a_{23} - a_{13}a_{21}) + a_{33}(a_{11}a_{22} - a_{12}a_{21}) \\
&=b_3^{2m}\left( \frac{1}{4m}(1-\frac{1}{b_2^2})(-1-\Theta_2b_2^2)+\frac{\Theta_2(b_2^m-\frac{1}{b_2^m})(-1-\Theta_2b_2^2)}{4m^2b_2^m}\right)+A_{3,3}\left( -\frac{1}{2}-\frac{1}{2}\Theta_2\right) b_3^2\\
& \ +A_{3,3}(-1-\Theta_2b_2^2)\left(-\frac{1}{2}+\frac{1}{2m}\right).
\end{align*}
where $A_{3,3} :=a_{11}a_{22} - a_{12}a_{21}=\frac{\Theta_2}{2m}(1-b_2^{2m})+\frac{1}{2b_2^2}(1-b_2^2)$ is the cofactor.
Let us write 
\begin{align}\label{determinant_f}
\frac{2mb_2}{b_3} \text{det}(M_{1}(\Theta_3,b_3))=B_0b_3^{2m}+B_1b_3^2+B_2 =:f(b_3), \end{align} where
\[
B_0=(-1-\Theta_2b_2^2)\left( \frac{1}{4m}(1-\frac{1}{b_2^2})+\frac{\Theta_2(b_2^m-\frac{1}{b_2^m})}{4m^2b_2^m}\right),
\]
\[
B_1=A_{3,3}\left( -\frac{1}{2}-\frac{1}{2}\Theta_2 \right),
\]
\[
B_2=A_{3,3}(-1-\Theta_2b_2^2)\left(-\frac{1}{2}+\frac{1}{2m} \right).
\]
Under the condition in \eqref{parameter_3}, we have
\begin{align}\label{inequality 1}
&-1-\Theta_2b_2^2>0,\\\label{inequality 2}
&\frac{1}{4m}(1-\frac{1}{b_2^2})+\frac{\Theta_2(b_2^m-\frac{1}{b_2^m})}{4m^2b_2^m}>0,\\\label{inequality 3}
&A_{3,3}>0,\\\label{inequality 4}
&-\frac{1}{2}-\frac{\Theta_2}{2}>0,\\\label{inequality 5}
&-\frac{1}{2}+\frac{1}{2m}<0.
\end{align} 
So $B_0> 0$, $B_1> 0$, $B_2<0$. We then can use the Descartes rule of signs and show there exists a unique solution $b_3\in (0,\infty)$.  Since $B_2<0$, the Descartes rule again tells us that we only need show that $f(b_2) > 0$ to show the unique solution $b_3$ such that $f(b_3) = 0$ satisfies $0<b_3<b_2$. Once this is done, then $\Theta_3^*$ can be uniquely determined by $1+\Theta_2b_2^2+\Theta_3b_3^2=0$ and this proves the existence of $b_3$ and $\Theta^*_3$ satisfying \eqref{singular}. Therefore we compute
\begin{align*}
&f(b_2) = B_0b_2^{2m}+B_1b_2^2+B_2\\
&=[\frac{(\Theta_2b_2^2+1)(1-b_2^2)}{4mb_2^2}+\Theta_2\frac{(\frac{1}{b_2^m}-b_2^m)(1+\Theta_2b_2^2)}{4m^2b_2^m}]b_2^{2m}+(\frac{\Theta_2}{2m}(1-b_2^{2m})+\frac{1}{2b_2^2}(1-b_2^2))(-\frac{b_2^2}{2}+\frac{1}{2}-\frac{1}{2m}-\frac{\Theta_2b_2^2}{2m})\\
&=(\frac{b_2^2(1-b_2^{2m})}{4m^2}-\frac{b_2^2}{2m}\frac{1-b_2^{2m}}{2m})\Theta_2^2\\
&+[\frac{(1-b_2^2)b_2^{2+2m}}{4m^2b_2^2}+\frac{(\frac{1}{b_2^{m}}-b_2^{m})b_2^{2m}}{4m^2b_2^m}+\frac{1-b_2^{2m}}{2m}(-\frac{b_2^2}{2}+\frac{1}{2}-\frac{1}{2m})+\frac{1}{2b_2^2}(1-b_2^2)(-\frac{b_2^2}{2m})]\Theta_2\\
&+(b_2^{2m}\frac{1-b_2^2}{4mb_2^2})+\frac{1}{2b_2^2}(1-b_2^2)(-\frac{b_2^2}{2}+\frac{1}{2}-\frac{1}{2m})\\
&=\frac{1-b_2^2}{4mb_2^2}(b_2^{2m}+m-mb_2^2-1)\\
&=\frac{(1-b_2^2)^2}{4mb_2^2}(\frac{b_2^{2m}-1}{1-b_2^{2}}+m)\\
&=\frac{(1-b_2^2)^2}{4mb_2^2}(m-1-b_2^2-b_2^4-...-b_2^{2m-2})>0.
\end{align*}
Now we are left to show \eqref{nonsingular1}, that is, $\text{det}(M_n(\Theta_3^*),b_3)\neq 0$.
As before, we have (replacing $m$ by $nm$)
\begin{align}\label{eq112}
\frac{2nmb_2}{b_3}\text{det}(M_n(\Theta^*_3,b_3))=B_{0,n}b_3^{2nm}+B_{1,n}b_3^2+B_{2,n}.
\end{align}
where
\[
B_{0,n}=\frac{1}{4mn}(1-\frac{1}{b_2^2})(-1-\Theta_2b_2^2)+\frac{\Theta_2(b_2^{nm}-\frac{1}{b_2^{nm}})(-1-\Theta_2b_2^2)}{4n^2 m^2b_2^{nm}},
\]
\[
B_{1,n}=A_{3,3,n}\left(-\frac{1}{2}-\frac{1}{2}\Theta_2\right),
\]
\[
B_{2,n}=A_{3,3,n}(-1-\Theta_2b_2^2)(-\frac{1}{2}+\frac{1}{2nm}),
\] 
\[
 A_{3,3,n}=\frac{\Theta_2}{2nm}(1-b_2^{2nm})+\frac{1}{2b_2^2}(1-b_2^2).
\]
Since $f(b_3)=0$, where $f$ is as in \eqref{determinant_f}, that is, $b_3^{2}=-\frac{B_2+B_0b_3^{2m}}{B_1}$. Thus,
\[
b_3^{2nm}B_{0,n}+b_3^2 B_{1,n}+B_{2,n}= b_3^{2nm}B_{0,n}-\frac{B_{1,n}}{B_1}B_0b_3^{2m}-\frac{B_{1,n}}{B_1}B_2+B_{2,n}.
\] 
Let $q(z)=z^nB_{0,n}-\frac{B_{1,n}}{B_1}B_0z-\frac{B_{1,n}}{B_1}B_2+B_{2,n}$. We will show that 
\begin{align}\label{qpolynomial}
q(0)< 0 \text{ and }q'(z)<0 \quad \text{ for $0\leq z \leq b_2^{2m}.$}
\end{align} Once we have the above inequalities, then $q(b_3^{2m}) < 0$, which implies $\text{det}(M_n(\Theta^*_3))$ in \eqref{eq112} cannot be zero. This clearly finishes the proof.

To show \eqref{qpolynomial}, let us first consider $q(0)$. Note that
\[
\frac{B_{1,n}}{B_1}=\frac{A_{3,3,n}}{A_{3,3}},\  \frac{B_{2,n}}{B_2}=\frac{A_{3,3,n}(-\frac{1}{2}+\frac{1}{2mn})}{A_{3,3}(-\frac{1}{2}+\frac{1}{2m})}
\]
\[
q(0)=-\frac{B_{1,n}}{B_1}B_2+B_{2,n}<0 \Leftrightarrow \frac{B_{2,n}}{B_2}>\frac{B_{1,n}}{B_1}\Leftrightarrow \frac{A_{3,3,n}}{A_{3,3}}>0
\]
Thanks to \eqref{inequality 4}, it suffices to show that $\frac{A_{3,3,n}}{A_{3,3}}>1$. Note that
\[
\frac{A_{3,3,n}}{A_{3,3}}>1\Leftrightarrow \frac{\Theta_2}{2mn}(1-b_2^{2mn})>\frac{\Theta_2}{2m}(1-b_2^{2m}).
\]
Clearly, the last inequality holds because $b_2<1$,  and $\Theta_2<0$, therefore we have 
\[
q(0) < 0.
\]
Now, we turn to $q'(z)<0$. We have
\[
q'(z)=nz^{n-1}B_{0,n}-\frac{B_{1,n}}{B_1}B_0.
\]
If $B_{0,n}\leq 0$, then $z\mapsto q'(z)$ is deacreasing and   $q'(0)=-\frac{B_{1,n}}{B_1}B_0=-\frac{A_{3,3,n}}{A_{3,3}}B_0<0$, which yields $q'(z)<0$ for all $ 0 \le z \le b_2^{2m}$. If $B_{0,n}>0$, it is sufficient to show 
\begin{align}\label{eq222}
q'(b_2^{2m})\leq n(b_2^{2mn-2m})B_{0,n}-\frac{B_{1,n}}{B_1}B_0<0.
\end{align}
From \eqref{parameter_3}, we  have $(1-\frac{1}{b_2^2})+\frac{\Theta_2}{m}(1-\frac{1}{b_2^{2m}})>\frac{1-b_2^2}{b_2^2}$. And the lower bound of $\Theta_2$ in \eqref{parameter_3} is equivalent to $\frac{\Theta_2}{m}(1-b_2^{2m})+\frac{1}{b_2^2}(1-b_2^2)>0$,  hence
\[
\eqref{eq222}\Leftrightarrow \frac{A_{3,3,n}}{A_{3,3}}=\frac{\frac{\Theta_2}{nm}(1-b_2^{2nm})+\frac{1}{b_2^2}(1-b_2^2)}{\frac{\Theta_2}{m}(1-b_2^{2m})+\frac{1}{b_2^2}(1-b_2^2)}> \frac{(1-\frac{1}{b_2^2})+\frac{\Theta_2}{mn}(1-\frac{1}{b_2^{2nm}})}{(1-\frac{1}{b_2^2})+\frac{\Theta_2}{m}(1-\frac{1}{b_2^{2m}})}b_2^{2nm-2m},
\]
\[
\Leftrightarrow [\frac{\Theta_2}{nm}(1-b_2^{2nm})+\frac{1}{b_2^2}(1-b_2^2)][(1-\frac{1}{b_2^2})+\frac{\Theta_2}{m}(1-\frac{1}{b_2^{2m}})]>[(1-\frac{1}{b_2^2})+\frac{\Theta_2}{mn}(1-\frac{1}{b_2^{2nm}})][\frac{\Theta_2}{m}(1-b_2^{2m})+\frac{1}{b_2^2}(1-b_2^2)]b_2^{2mn-2m},
\]
\[
\Leftrightarrow\Theta_2[\frac{(1-b_2^{2mn})(1-b_2^{-2})+n(b_2^{-2}-1)(1-b_2^{-2m})}{nm}-\frac{(b_2^{2mn}-1)b_2^{-2m}(b_2^{-2}-1)+n(1-b_2^{2m})b_2^{2mn-2m}(1-b_2^{-2})}{nm}]
\]
\[
+\frac{1}{b_2^2}(1-b_2^2)(1-\frac{1}{b_2^2})-(1-\frac{1}{b_2^2})\frac{1}{b_2^2}(1-b_2^2)b_2^{2nm-2m}>0
\]
\[
\Leftrightarrow \frac{\Theta_2(1-b_2^{-2})(1-b_2^{-2m})(1-n)(1-b_2^{2nm})}{mn}-(1-b_2^{-2})^2(1-b_2^{2nm-2m})>0,
\]
\[
\Leftrightarrow \Theta_2<\frac{m(b_2^2-1)b_2^{2m}}{b_2^2(1-b_2^{2m})}\frac{1-b_2^{2m(n-1)}}{(1-b_2^{2mn})(1-\frac{1}{n})}.
\]
Since $\frac{1-b_2^{2m(n-1)}}{(1-b_2^{2mn})(1-\frac{1}{n})}\leq \frac{n}{n-1} \leq 2$, the condition $ \Theta_2< 2\frac{{b_2}^{2m-2}(b_2^{2}-1)m}{1-b_2^{2m}}$ in \eqref{parameter_3} leads to the inequality we want. This shows \eqref{qpolynomial}, and thus \eqref{nonsingular1} holds.
\end{proof}
From Lemma~\ref{kernel}, it is clear that the kernel of $A(\Theta_3^*,0)$ is one dimensional. We denote by $v\in \text{Ker}(A(\Theta_3^*,0))$ and the one such that
 \begin{align}\label{ker_element}
 v= \cos(mx)
 \begin{pmatrix}
 v^1\\
 v^2\\
 1
 \end{pmatrix}, \quad  \begin{pmatrix}
 v^1\\
 v^2\\
 1
 \end{pmatrix} \in Ker(M_1(\Theta_3^*,b_3)).
 \end{align}

Since $b_3$ and $\Theta^*_3$ are tied in the relation $1+\Theta_2b_2^2 + \Theta^*_3b_3^2 = 0$, we now drop the dependence on $b_3$ of $M_n(\Theta_3^*):=M_n(\Theta^*_3,b_3)$.

We now characterize the image of $A(\Theta_3^*,0)$. 

\begin{lemma}\label{Image_space}
Let $b_2, \Theta_2$ satisfy the conditions in \eqref{parameter_3} and let $ b_3, \Theta_3^*$ be as in Lemma~\ref{kernel}. Then,
\begin{align}\label{def_Z}
Z = \left\{Q=(Q_1, Q_2, Q_3) \in Y^{k+1,m}\times Y^{k,m} \times Y^{k,m}, \begin{pmatrix}
Q_1\\
Q_2\\
Q_3
\end{pmatrix}
= \sum_{n=1}^{\infty} \begin{pmatrix}
q_{1,n}\\
q_{2,n}\\
q_{3,n}
\end{pmatrix}
\sin(mnx), 
\right. \\ \left.
\left(\begin{array}{c}q_{1,1} \\ q_{2,1} \\ q_{3,1} \end{array}\right) \in \text{span}\left\{ 
c_1, c_2\right\}
\right\}, 
\end{align}
where $c_i$ is the $i$th column of $M_1(\Theta_3^*,b_3)$. Then $Z = \text{Im}\left(A(\Theta_3^*,0)\right)$.
\end{lemma}
\begin{proof}
We  choose an element $w \in \text{Im}(A(\Theta_3^*,0))^{\perp}$ to be the one such that
\begin{align}\label{image_perp_element}
w=\frac{1}{|c_1\times c_2 |}\cos(mx)\left( c_1 \times c_2 \right) \in Im(M_1(\Theta^*))^{\perp}.
\end{align}
where $c_1$ and $c_2$ are the first two column vectors in $M_1(\Theta^*_3)$, which are linearly independent. Since $\text{det}(M_1(\Theta_3^*))=0$, and $\text{det}(M_{n}(\Theta_3^*,b_3))\ne 0$ for all $n>1$, and Proposition \ref{Fredholm}, it is clear that $\text{Im}\left( A(\Theta_3^*,0)\right)\subset Z$. 
In order to show the other direction, $Z \subset \text{Im}\left( A(\Theta_3^*,0)\right)$, we pick an element $Q\in (Q_1,Q_2,Q_3)\in Z$,$Q_i=\sum_{n=1}^{\infty} q_{1,n}\sin(mnx)$, and
we have $\lambda_Q^1$, $\lambda_Q^2\in \RR$ such that 
\begin{align*}
\begin{pmatrix}
q_{1,2}\\
q_{2,2}\\
q_{3,2}
\end{pmatrix}
=\lambda_Q^1c_1+\lambda_Q^2c_2.
\end{align*}
Thanks to \eqref{nonsingular1}, one can find a sequence of vectors $h^n:=(h_{1,n}, h_{2,n}, h_{3,n})$ such that  \begin{align}\label{matrix_form2}
(-mn)M_{mn}(\Theta^*) h^n = (-mn)M_{mn}(\Theta^*)\begin{pmatrix}
h_{1,n} \\ h_{2,n} \\ h_{3,n}
\end{pmatrix}
 = \begin{pmatrix}
q_{1,n} \\ q_{2,n} \\ q_{3,n}
\end{pmatrix} =: q^n,
\end{align}
Therefore it suffices to show that  
\begin{align}\label{regularity_image}
\quad |h_{1,n}| + |h_{2,n}| + |h_{3,n}| \le C \left(|q_{1.n}| + \frac{|q_{2,n}|}{|mn|}  +\frac{ |q_{3,n}|}{|mn|}\right), \quad \text{ for all sufficiently large $n$}.
\end{align}
where $C$ is independent of $n$. Once we have the above inequality, then it is clear that $Q = A(\Theta^*_3,0)[h]$ where
\begin{align*}
h_1(x)=\sum_{n=1}^\infty h_{1,n}\cos(mnx), \quad h_2(x)=\sum_{n=1}^\infty h_{2,n}\cos(mnx), \quad h_3(x)=\sum_{n=1}^\infty h_{3,n}\cos(mnx),
\end{align*}
and $h \in (H^{k+1})^3$, which follows from
\begin{align*}
\| h \|_{(H^{k+1})^3}^2 &= \sum_{n\ge 1}|h_{1,n}|^2|nm|^{2(k+1)} + |h_{2,n}|^2|nm|^{2(k+1)} +|h_{3,n}|^2|nm|^{2(k+1)}\\
& \lesssim \sum_{n\ge1}|mn|^{2(k+1)}\left(|q_{1,n}|^2 + \frac{|q_{2,n}|^2}{|mn|^2}+ \frac{|q_{2,n}|^2}{|mn|^2} \right)\\
& \lesssim \| Q_1 \|_{H^{k+1}}^2 + \|Q_2 \|_{H^{k}}^2 + \|Q_3 \|_{H^{k}}^2,\\
& < \infty.
\end{align*}
This proves that $Z \subset \text{Im}\left( A(\Theta_3^*,0)\right)$.

In order to prove \eqref{regularity_image}, we denote
\[
(-mn)M_{n}(\Theta^*_3) =: 
\begin{pmatrix}
a_{11} & a_{12} & a_{13}\\ a_{21} & a_{22} & a_{23} \\ a_{31} & a_{32} & a_{33}
\end{pmatrix} = \begin{pmatrix}  a_{11}  & 0 & 0 \\ 0 & a_{22} & 0 \\ 0 & 0 & a_{33}\end{pmatrix} + \begin{pmatrix}  0  & a_{12} & a_{13} \\ a_{21} & 0 & a_{23} \\ a_{31} & a_{32} & 0\end{pmatrix} =: D_n + S_n.
\]
Then \eqref{matrix_form2} is equivalent to 
\begin{align}\label{decomp_112}
D_n h^n = q^n - S_n h^n.
\end{align}
We claim that 
\begin{align}\label{claim_diagonal}
d_{2}:=\Theta_2b_2+b_2+\Theta_3\frac{b_3^2}{b_2}\neq 0, \quad \text{ and } \quad  d_{3}:=\Theta_3b_3+b_3+\Theta_2b_3\neq 0.
\end{align}
From \eqref{mn_formula}, and $0<b_3<b_2<1$, it follows that
\[
D_n = \begin{pmatrix} -\frac{1}{2} & 0 & 0 \\ 0 & \frac{mn}{2}\left( d_2  - \frac{\Theta_2b_2}{nm} \right) & 0  \\ 0 & 0 & \frac{nm}{2}\left(d_3 - \frac{\Theta_3^*b_3}{nm} \right) \end{pmatrix} \quad \text{ and } | S_n h^n | \lesssim \kappa^{nm}|h^n|, \quad \text{ for some $\kappa\in(0,1)$.}
\]
Therefore, for sufficiently large $n$, $D_n$ is a non-singular diagonal matrix, and it follows from \eqref{decomp_112}
\[
|h_{1,n}| \lesssim |q_{1,n}| + \kappa^{nm} |h^n|, \quad |h_{2,n}| \lesssim \frac{|q_{2,n}|}{|nm|} + \kappa^{nm} |h^n|, \quad  |h_{3,n}| \lesssim \frac{|q_{3,n}|}{|nm|} + \kappa^{nm} |h^n|,
\] 
which proves \eqref{regularity_image}.

We are left to show \eqref{claim_diagonal}. From \eqref{singular},  
\begin{align}\label{d2negative}
d_2 = \Theta_2b_2+b_2+\Theta_3\frac{b_3^2}{b_2}=(\Theta_2b_2^2+b_2^2+\Theta_3b_3^2)\frac{1}{b_2}=\frac{b_2^2-1}{b_2} <  0.
\end{align} For $d_3$, let us claim 
\begin{align}\label{non-zero}
B_0(\frac{1+\Theta_2b_2^2}{\Theta_2+1})^{m}+B_1(\frac{1+\Theta_2b_2^2}{\Theta_2+1})+B_2\neq 0, \quad \text{ where $B_0,$ $B_1$ and $B_2$ are as in \eqref{determinant_f}}.
\end{align}

Plugging in $B_0, B_1, B_2$ from \eqref{determinant_f}, we can find that  \eqref{non-zero} is equivalent to
\[
-[\frac{1}{4m}(1-b_2^{-2})+\frac{\Theta_2(b_2^{m}-\frac{1}{b_2^m})}{4m^2b_2^m}]\frac{(1+\Theta_2b_2^2)^m}{(\Theta_2+1)^m}-\frac{1}{2}A_{3,3}-A_{3,3}(-\frac{1}{2}+\frac{1}{2m})\neq 0
\]
\[
\Leftrightarrow-[(1-b_2^{-2})+\frac{\Theta_2}{m}(1-b_2^{-2m})](1+\Theta_2b_2^2)^m\neq [\frac{\Theta_2}{m}(1-b_2^{2m})+\frac{1}{b_2^2}(1-b_2^2)](\Theta_2+1)^m.
\]
According to \eqref{inequality 1}, \eqref{inequality 2} and \eqref{inequality 3},  the RHS is strictly positive and LHS is strictly negative.
From Lemma \ref{kernel} and \eqref{determinant_f}, $B_0b_3^{2m}+B_1b_3^2+B_2=0$. This gives $b_3^2\neq \frac{1+\Theta_2 b_2^2}{1+\Theta_2}=\frac{-\Theta_3b_3^2}{1+\Theta_2}$, then 
\begin{align}\label{non-zero2}
d_3 = (1+\Theta_2+\Theta_3)b_3\neq 0.
\end{align}
\end{proof}

\begin{proofprop}{Fredholm} The result follows immediately from Lemma~\ref{kernel} and \ref{Image_space}. Note that $v$ and $w$ are supported on the $m$th Fourier mode, thanks to \eqref{ker_element} and \eqref{def_Z}.
\end{proofprop}

\begin{lemma}\label{transversality_11}
$\partial_{\Theta_3} A(\Theta_3^*,0)[v] \notin \text{Im}(A(\Theta_3^*,0))$.
\end{lemma}
\begin{proof}
From \eqref{matrix_form1}, it follows that 
\[
\partial_{\Theta_3} A(\Theta^*_3,0)[h] = Q,
\]
where
\[
Q:=\sum_{n\ge 1}\begin{pmatrix} q_{1,n} \\ q_{2,n} \\ q_{3,n}\end{pmatrix}\sin(nm x), \quad h=\sum_{n\ge 1}\begin{pmatrix} h_{1,n} \\ h_{2,n} \\ h_{3,n}\end{pmatrix}\cos(nm x), \quad  \begin{pmatrix} q_{1,n} \\ q_{2,n} \\ q_{3,n}\end{pmatrix} = (-mn)\partial_{\Theta_3}M_{n}(\Theta^*_3)
 \begin{pmatrix}h_{1,n} \\ h_{2,n} \\ h_{3,n} \end{pmatrix} 
\]
Let us write 
\begin{align}\label{derivative matrix}
 M_1(\Theta_3^*) = 
\left(
\begin{array}{ccc}
 a_{11} & a_{12} & a_{13}\\
a_{21} &a_{22} &a_{23}\\
a_{31}&a_{32}& a_{33}
\end{array}
\right), \quad CM_{33}=\left(
\begin{array}{cccc}
a_{11} & a_{12} \\
a_{21} &a_{22}
\end{array}
\right),  \quad \partial_{\Theta_3} M_1(\Theta_3^*) =\left(
\begin{array}{cccc}
0 & 0 & b_{13}\\
0 &0 &b_{23}\\
b_{31}&b_{32}& b_{33}
\end{array}
\right),
\end{align}
where the vanishing elements in $\partial_{\Theta_3} M_1(\Theta_3^*)$ follows from  \eqref{mn_formula}. Note that
\[
\text{det}(CM_{33}) = a_{11}a_{22} -a_{12}a_{21} = \frac{1}{2m}\left(-\frac{1}{2}d_2 + \frac{\Theta_2b_2}{2m}(1-b_2^{2m}) \right), \quad \text{ where $d_2$ is as in \eqref{claim_diagonal}.}
\]
From \eqref{d2negative}, we have
\[
-\frac{1}{2}d_2 + \frac{\Theta_2b_2}{2m}(1-b_2^{2m})  = \frac{1-b_2^2}{2b_2} + \frac{\Theta_2b_2}{2m}(1-b_2^{2m}) > \frac{1-b_2^2}{2b_2} + \frac{b_2^2-1}{2} = \frac{1-b_2^2}{2}\left( \frac{1}{b_2} - 1\right) > 0,
\]
where the first inequality follows from \eqref{parameter_3}, and the last inequality follows from $0 < b_2 < 1$. Hence,
\begin{align}\label{cm33det}
\text{det}(CM_{33}) > 0,
\end{align} and  $CM_{33}$ is invertible.
Since $ \text{Im}(M_1(\Theta^*_3))^{\perp}=\text{Ker}(M_1(\Theta^*_3)^T)$, we can choose
\begin{align*}
\tilde{v}=\left(
\begin{array}{ccc}
-(CM_{33})^{-1}\left(\begin{array}{ccc}
a_{13}\\
a_{23}
\end{array}\right)\\
1
\end{array}
\right), \quad \tilde{w}=\left(
\begin{array}{ccc}
-(CM_{33}^{T})^{-1}\left(\begin{array}{ccc}
a_{31}\\
a_{32}
\end{array}\right)\\
1
\end{array}
\right),
\end{align*}
so that $\tilde{v}\in \text{Ker}A$, and $\tilde{w}\in (\text{Im} A)^{\perp}$.
Then $\partial_{\Theta_3} A(\Theta_3^*,0)[v] \notin \text{Im}(A(\Theta_3^*,0))$ is equivalent to have
\[
\partial_{\Theta_3}M_z(\Theta_3,0)\tilde{v} \cdot \tilde{w} \neq 0.
\]
According to \eqref{derivative matrix}, it is equivalent to have
\begin{align}\label{trans}
-(b_{13}, b_{23})
{CM_{33}^T}^{-1}\left(
\begin{array}{ccc}
a_{31}\\a_{32}
\end{array} 
\right)-(b_{31}, b_{32}){CM_{33}}^{-1}
\left(
\begin{array}{ccc}
a_{13}\\a_{23}
\end{array}
\right)
+b_{33} \ne 0.
\end{align}
Using $\frac{db_3}{d\Theta_3}=-\frac{b_3}{2\Theta_3}$, which comes from $\Theta_3b_3^2+\Theta_2b_2^2+1=0$, we compute

\begin{align*}
{\partial_{\Theta_3}M_1(\Theta_3^*)}=
\left(
\begin{array}{ccc}
0&0&\frac{-(m-1)b_3^{m+1}}{4m}\\
0&0&\frac{-b_3^{m+1}(m-1)}{4mb_2^m}\\
-\frac{b_3^m}{4\Theta^*_3}&\frac{-b_3^m\Theta_2}{4\Theta_3^* b_2^{m-1}}&\frac{b_3}{2}(-\frac{1}{2}+\frac{1}{2m})+\frac{b_3}{4\Theta^*_3}+\frac{\Theta_2 b_3}{4\Theta^*_3}
\end{array}
\right).
\end{align*}
Then
\begin{align*}
(b_{13}, b_{23}){CM_{33}^T}^{-1}\left(
\begin{array}{ccc}
a_{31}\\a_{32}
\end{array} 
\right)=\frac{1}{\text{det}(CM_{33})}(\frac{-(m-1)}{4m}\frac{1}{2m})b_3^{2m+1}\left(1,\frac{1}{b_2^m}\right)
\left(
\begin{array}{ccc}
\frac{1}{2m}\Theta_2 b_2-\frac{1}{2}b_2+\frac{1}{2b_2}&-\frac{b_2^m}{2m}\\
-\frac{\Theta_2 b_2^{m+1}}{2m}&\frac{1}{2m}
\end{array}
\right)
\left(
\begin{array}{ccc}
1\\
\frac{\Theta_2}{b_2^{m-1}}
\end{array}
\right),
\end{align*}
\begin{align*}
\left(b_{31}, b_{32}\right)
{CM_{33}}^{-1}\left(
\begin{array}{ccc}
a_{13}\\a_{23}
\end{array} 
\right)=\frac{1}{\text{det}(CM_{33})}\frac{\Theta^*_3}{2m}(-\frac{b_3^{2m+1}}{4\Theta^*_3})\left(1,\frac{\Theta_2}{b_2^{m-1}}\right)
\left(
\begin{array}{ccc}
\frac{1}{2m}\Theta_2 b_2-\frac{1}{2}b_2+\frac{1}{2b_2}&-\frac{\Theta_2 b_2^{m+1}}{2m}\\
-\frac{b_2^m}{2m}&\frac{1}{2m}
\end{array}
\right)
\left(
\begin{array}{ccc}
1\\
\frac{1}{b_2^{m}}
\end{array}
\right).
\end{align*}
So \eqref{trans} is equivalent to
\begin{align*}
\frac{1}{\text{det}(CM_{33})}(\frac{m-1}{8m^2}+\frac{1}{8m})b_3^{2m+1}\left(1,\frac{1}{b_2^m}\right)
\left(
\begin{array}{ccc}
\frac{1}{2m}\Theta_2 b_2-\frac{1}{2}b_2+\frac{1}{2b_2}&-\frac{b_2^m}{2m}\\
-\frac{\Theta_2 b_2^{m+1}}{2m}&\frac{1}{2m}
\end{array}
\right)
\left(
\begin{array}{ccc}
1\\
\frac{\Theta_2}{b_2^{m-1}}
\end{array}
\right)+b_{33}\neq 0.
\end{align*}
\begin{align*}
\Leftrightarrow(\frac{m-1}{8m^2}+\frac{1}{8m})b_3^{2m+1}\left(1,\frac{1}{b_2^m}\right)
\left(
\begin{array}{ccc}
\frac{1}{2m}\Theta_2 b_2-\frac{1}{2}b_2+\frac{1}{2b_2}&-\frac{b_2^m}{2m}\\
-\frac{\Theta_2 b_2^{m+1}}{2m}&\frac{1}{2m}
\end{array}
\right)
\left(
\begin{array}{ccc}
1\\
\frac{\Theta_2}{b_2^{m-1}}
\end{array}
\right)+\text{det}(CM_{33})b_{33}\neq 0.
\end{align*}

\begin{align*}
& \Leftrightarrow \text{det}(CM_{33})\left( (-\frac{1}{4}+\frac{1}{4m})b_3+\frac{b_3^3}{4(-1-\Theta_2 b_2^2)}+\frac{\Theta_2 b_3^3}{4(-1-\Theta_2 b_2^2)}\right) \\ & \ +(\frac{m-1}{8 m^2}+ \frac{1}{8m})b_3^{2m+1}(\frac{1}{2 b_2}-\frac{1}{2}b_2+\frac{\Theta_2}{2m b_2^{2m-1}}-\frac{b_2 \Theta_2}{2m})\neq 0,
\end{align*}
\begin{align}\label{nonzero_11}
\Leftrightarrow (\frac{2m-1}{8m^2})b_3^{2m}(\frac{1}{2b_2}-\frac{1}{2}b_2+\frac{\Theta_2}{2m b_2^{2m-1}}-\frac{b_2\Theta_2}{2m})+\text{det}(CM_{33})\frac{1+\Theta_2}{4(-1-\Theta_2 b_2^2)}b_3^2+\text{det}(CM_{33})(-\frac{1}{4}+\frac{1}{4m})\neq 0.
\end{align}
To show \eqref{nonzero_11}, let us write 
\begin{align*}
&C_0 := (\frac{2m-1}{8m^2})(\frac{1}{2b_2}-\frac{1}{2}b_2+\frac{\Theta_2}{2m b_2^{2m-1}}-\frac{b_2\Theta_2}{2m})\\
& C_1 :=\text{det}(CM_{33})\frac{1+\Theta_2}{4(-1-\Theta_2 b_2^2)} \\
& C_2 :=\text{det}(CM_{33})(-\frac{1}{4}+\frac{1}{4m}),
\end{align*}
then, it is enough to show that $g(b_3) : = C_0 b_3^{2m} + C_1 b_3^2 + C_2 \ne 0$. 
Note that we have
\begin{align*}
& \text{det}(CM_{33}) > 0, \quad \text{ from \eqref{cm33det}}\\
& (1+\Theta_2) < 0,  \quad \text { from \eqref{parameter_3} and $0<b_2<1$}\\
& 1 + \Theta_2b_2^2 < 0, \quad \text{ from \eqref{singular} and $\Theta^*_3 >0$},
\end{align*} thus, $C_1 < 0$. we also have $C_2<0$ since $m>1$. For $C_0$, we have
\begin{align*}
(\frac{1}{2 b_2}-\frac{1}{2}b_2+\frac{\Theta_2}{2m}(\frac{1}{b_2^{2m-1}}-b_2))&=\frac{b_2}{2}[\frac{1}{b_2^2}-1+\frac{\Theta_2}{m}(\frac{1}{b_2^{2m}}-1)]\\
&=\frac{b_2}{2}(\frac{1}{b_2^2}-1)[1+\frac{\Theta_2}{m}(1+\frac{1}{b_2^2}+...+\frac{1}{b_2^{2(m-1)}})]\\
&<\frac{b_2}{2}(\frac{1}{b_2^2}-1)(1+\frac{\Theta_2}{m}m)\\
& < 0,
\end{align*}
as $b_2<1$, $\Theta_2\leq \frac{-1}{b_2^2}<-1$, which follows from \eqref{parameter_3}. This yields that  $C_0<0$. 
Therefore  the Descartes rule of signs implies that there is no positive root of $g(x)=0$, which implies $g(b_3) \ne 0$, that is \eqref{nonzero_11}. 
\end{proof}

\subsection{Analytic regularity of the boundary}
\begin{theorem}\label{ofc}
The solutions $R(s)\in (X^{k,m})^3$ constructed in Theorem~\ref{teoremaestacionarias2} are analytic when s is sufficiently small.
\end{theorem}
\begin{proof}
 We use \cite[Theorem 1]{Kinderlehrer-Nirenberg:regularity-free-boundary} and \cite[Theorem 3.1]{Kinderlehrer-Nirenberg-Spruck:regularity-elliptic-free-boundary} to prove the analyticity.
 Let the stream function
 \[\varphi(s)=\sum_{i=1}^{3}\Theta_i1_{D_i}*\frac{1}{2\pi}\log(\cdot).
 \]
 We first consider the outermost boundary $\partial$$D_1$. For any $x_1\in\partial{D}_1$, we have
 \[
 \Delta{\varphi}=\Theta_1,\ in \ D_1\backslash \overline{D_2},
 \]
 \[
 \varphi=c_1,\ on \ \partial D_1.
 \]
 From Lemma \ref{maximum}, we have
 \[
 \nabla\varphi=0, \ on \ \partial D_1.
 \]
Moreover, from Theorem \ref{teoremaestacionarias2}, we know $\partial D_1$ is in $C^2$. Thus $\varphi$ is also in $C^2(\overline{D_1}\backslash \overline{D_2})$. Then by \cite[Theorem 1]{Kinderlehrer-Nirenberg:regularity-free-boundary}, we have the analyticity of $\partial D_1$.

Now we consider $\partial D_2$. We have
\begin{align*}
 \Delta{\varphi}=\begin{cases}\Theta_2+\Theta_1,& in \ D_2\backslash \overline{D_3},\\
 \Theta_1,& in \ D_1\backslash\overline{D_2}
 \end{cases}
 \end{align*}
 \[
 \varphi=c_2,\ on \ \partial D_2.
 \]
 From Corollary \ref{integral_a0} and \eqref{d2negative}, at the bifurcation point, we have
 \begin{align*}
 &\partial_{n}\varphi(0)=\frac{1}{2\pi}\sum_{i=1}^{3}\Theta_i\int_{0}^{2\pi}\cos(x-y)b_i\log(b_i^2+b_j^2-2b_ib_j\cos(x-y))dy\\
 &=\frac{1}{2\pi}(-2\pi)(\Theta_1b_1\frac{b_2}{b_1}+\Theta_2b_2+\Theta_3b_3\frac{b_3}{b_2})=-(\Theta_1b_2+\Theta_2b_2+\Theta_3\frac{b_3^2}{b_2})\neq 0.
 \end{align*}
 Thus when s is sufficiently small, $\partial_{n}\varphi(x)\neq 0$.
 
 From Theorem \ref{teoremaestacionarias2}, we have $\partial {D_2}$ is in $C^2$.  So $\varphi$ is also in $C^2((\overline{D_2}\backslash D_3)\cap(\overline{D_1}\backslash D_2))\cap C^1(\overline{D_1})$.
 Then we can use \cite[Theorem 3.1]{Kinderlehrer-Nirenberg-Spruck:regularity-elliptic-free-boundary} to get the result.
 The analyticity of $\partial D_3$ can be shown in the same way by using \eqref{non-zero2} instead of \eqref{d2negative}.
 \end{proof}
 \appendix
\section{Appendix}
\subsection{Crandall--Rabinowitz Theorem}
We recall the Crandall-Rabinowitz theorem \cite{Crandall-Rabinowitz:bifurcation-simple-eigenvalues}, that plays a crucial role throughout this paper.

\begin{theorem}{\cite[Theorem 1.7]{Crandall-Rabinowitz:bifurcation-simple-eigenvalues}}\label{CRtheorem}
Let $X,Y$ be Banach spaces, $V$ a neighborhood of $0$ in $X$ and 
\[
F:(-1,1) \times V \mapsto Y
\]
have the properties:
\begin{enumerate}
\item $F(t,0) = 0$ for $|t|<1$,
\item The Gateaux derivatives of $F$, $\partial_t F$ $DF$, and $\partial_tDF$ exist and are continuous,
\item $\text{Ker}(DF(0,0))$ and $\text{Im}(DF(0,0))^{\perp}$ are one-dimensional,
\item $\partial_t DF(0,0)[x_0] \notin \text{Im}(DF(0,0))$, where
\[
\text{Ker}(DF(0,0)) = \text{span}\left\{ x_0 \right\}.
\] 
\end{enumerate}

 Then there is a neighborhood $U$ of $(0,0) \subset \RR \times X$, an interval $(-a,a)$ and continuous functions $\varphi:(-a,a)\mapsto \RR$, $\psi:(-a,a)\mapsto z$ such that $\varphi(0) = \psi(0)$ and 
 \[
 F(\varphi(\alpha),\alpha x_0 + \alpha \psi(\alpha)) = 0, \quad |\alpha| < a.
 \]
\end{theorem}

\subsection{Useful integrals}

The following lemmas will deal with the integrals that appear throughout the calculation of the linear operator:

\begin{lemma}
\label{lemaexpansionlog}
Consider the expansion of the function $\log(A_0 + A_1 t + A_2 t^2)$. We have that:

\begin{align*}
\log(A_0 + A_1 t + A_2 t^2) \sim \log(A_0) + t \frac{A_1}{A_0} + t^{2}\left(\frac{2A_0A_2 - A_1^2}{2A_0^{2}}\right) + t^{3}\left(\frac{A_1^{3} - 3A_0A_1A_2}{3A_0^{3}}\right)
\end{align*}
\end{lemma}
\begin{proof}
Follows from the Taylor expansion.
\end{proof}

%
%

\begin{lemma}\label{integral_a1}
Let $0 < r < 1$ and $\mathcal{A}_{1}(r,m)$ be

\begin{align*}
\mathcal{A}_{1}(r,m) = \int \frac{\cos(my)}{1+r^2 - 2r\cos(y)} dy
= \frac{2\pi}{1-r^2}r^{m}
\end{align*}
\end{lemma}
\begin{proof}
We use the residue theorem. We set $\gamma$ to be the boundary of the unit disk and let $z=e^{iy}$
\begin{align*}
\mathcal{A}_{1}(r,m) &= \int \frac{\cos(my)}{1+r^2 - 2r\cos(y)} dy =\frac{1}{ri}\int_{\gamma}\frac{z^m}{\frac{1+r^2}{r}-z-z^{-1}}\frac{1}{z}dz\\
&=-\frac{1}{ri}\int_{\gamma}\frac{z^m}{(z-r)(z-\frac{1}{r})}dz=-\frac{2\pi i}{ri}\frac{r^m}{r-\frac{1}{r}}=\frac{2\pi r^m}{1-r^2}.
\end{align*}\end{proof}

\begin{corollary}\label{integral_a0}
Let $0 < r \leq 1$ and $\mathcal{A}_{0}(r,m)$ be

\begin{align*}
\mathcal{A}_{0}(r,m) = \int \log(1+r^2 - 2r\cos(y))\cos(my) dy
= -\frac{2\pi}{m}r^{m}
\end{align*}
\end{corollary}
\begin{proof}
For $r\leq 1$, we integrate by parts and we have
\[
\mathcal{A}_{0}(r,m)=-\frac{r}{m}\int \frac{\cos((m-1)y)-\cos((m+1)y)}{1+r^2 - 2r\cos(y)} dy=-\frac{2\pi}{m}r^{m}
\]
For $r=1$, we use the continuity.
\end{proof}

\subsection{Integral estimates}\label{integral_estimates}

\begin{lemma}{\cite[Lemma 3.4]{Majda-Bertozzi:vorticity-incompressible-flow}}\label{ponce_kato} Let $f,g\in C^{\infty}(\mathbb{T}^d)$. Then,
\begin{align}\label{kato_ponce}
\rVert f g \rVert_{H^k(\mathbb{T}^d)} \lesssim_{k} \rVert f \rVert_{L^\infty(\mathbb{T}^d)} \rVert g \rVert_{H^{k}(\mathbb{T}^d)} + \rVert f \rVert_{H^k(\mathbb{T}^d)} \rVert g \rVert_{L^{\infty}(\mathbb{T}^d)} \quad \text{ for $k\in \mathbb{N}\cup \left\{ 0 \right\}$}.
\end{align}
\end{lemma}

\begin{lemma}{\cite{Adams-Fournier:sobolev-spaces-book}}\label{GNinterpolation}
Let $f\in C^{\infty}(\mathbb{T}^d)$. Then for  $0\le k\le n\in \mathbb{N}\cup\left\{ 0 \right\}$, it holds that
\[
\rVert \nabla^{k} f \rVert_{L^2(\mathbb{T}^d)} \lesssim_{k,n} \rVert f \rVert_{L^2(\mathbb{T}^d)}^{1-\frac{k}{n}}\rVert \nabla^n f \rVert_{L^2(\mathbb{T}^d)}^{\frac{k}{n}}.
\]
\end{lemma}

\begin{lemma}\label{pq_derivatives}
Let $f,g\in C^\infty(\mathbb{T})$ and $\sigma,p,q\in \mathbb{N}\cup \left\{ 0 \right\}$ satisfy $p+q=\sigma$ with $q\le \sigma-1$. Then, for each $k\ge 0$, it holds that
\begin{align*}
\|f^{(p)}g^{(q)}\|_{H^k(\mathbb{T})}
&\lesssim_{k,\sigma}\|f\|_{H^2(\mathbb{T})}\|g\|_{H^{k+\sigma-1}(\mathbb{T})}+\|f\|_{H^{k+\sigma}(\mathbb{T})}\|g\|_{H^1(\mathbb{T})}.
\end{align*}
\end{lemma}
\begin{proof}
Using Lemma \ref{ponce_kato} and the Sobolev embedding theorem, we have \begin{align*}
\|f^{(p)}g^{(q)}\|_{H^k(\mathbb{T})}&\leq \|f'g\|_{H^{k+\sigma-1}(\mathbb{T})}\\
&\lesssim_{k,\sigma}\|f'\|_{L^{\infty}(\mathbb{T})}\|g\|_{H^{k+\sigma-1}(\mathbb{T})}+\|f\|_{H^{k+\sigma}(\mathbb{T})}\|g\|_{L^{\infty}(\mathbb{T})}\\
&\lesssim_{k,\sigma}\|f\|_{H^2(\mathbb{T})}\|g\|_{H^{k+\sigma-1}(\mathbb{T})}+\|f\|_{H^{k+\sigma}(\mathbb{T})}\|g\|_{H^1(\mathbb{T})}.
\end{align*}
\end{proof}

\begin{lemma}\label{composition}
Let $F\in C^\infty_0(\RR^n)$ and $u\in (C^{\infty}(\mathbb{T}^2))^n$. If $\rVert u \rVert_{(H^2(\mathbb{T}^2))^n} \lesssim 1$, then for $k\ge 2$,
\begin{align}
\rVert F(u)\rVert_{H^{k}(\mathbb{T}^2)} \lesssim_{k} 1 + \rVert u \rVert_{(H^{k}(\mathbb{T}^2))^n}.\label{norm_estimate}
\end{align}
For $u_1,u_2\in (C^{\infty}(\mathbb{T}^2))^n$, if $\rVert u_1 \rVert_{(H^2(\mathbb{T}^2))^n} ,\rVert u_2 \rVert_{(H^2(\mathbb{T}^2))^n} \lesssim 1$ we have that for $k\ge 2$,
\begin{align}\label{lip_estimate_F}
\rVert F(u_1)-F(u_2)\rVert_{H^{k}(\mathbb{T}^2)} \lesssim_{k} (1+ \rVert u_1 \rVert_{(H^k(\mathbb{T}^2))^n}  + \rVert u_2 \rVert_{(H^k(\mathbb{T}^2))^n} )\rVert u_1 - u_2 \rVert_{(H^{k}(\mathbb{T}^2))^n}.
\end{align}
\end{lemma}

\begin{proof}
Since $\rVert u \rVert_{(L^{\infty}(\mathbb{T}^2))^n} \lesssim \rVert u \rVert_{(H^2(\mathbb{T}^2))^n}\lesssim 1$. Then \eqref{norm_estimate} follows straightforwardly from Lemma~\ref{ponce_kato} and Fa\'a di Bruno's formula. For \eqref{lip_estimate_F}, we have
\[
F(u_1(x)) - F(u_2(x) ) = \int_{0}^{1}\frac{d}{dt}F((u_1(x)-u_2(x))t + u_2(x))dt = \int_0^1 \nabla F((u_1(x)-u_2(x))t + u_2(x))\cdot (u_1(x)-u_2(x))dt.
\]
Thus equation \eqref{lip_estimate_F} follows from Lemma \ref{ponce_kato} and \eqref{norm_estimate}.
\end{proof}

\begin{lemma}\label{appendix_lem_2}
Let $K(x,y)\in C^{\infty}(\mathbb{T}^2)$ and $\tilde{K}(x,y):= \frac{K(x,y)-K(x,x)}{2\sin\left( \frac{x-y}{2}\right)}$. Then for $k \in \mathbb{N}\cup \left\{ 0 \right\}$, 
\[
\rVert \tilde{K} \rVert_{H^{k}(\mathbb{T}^2)} \lesssim_k \rVert K \rVert_{H^{k+1}(\mathbb{T}^2)}.
\]
\end{lemma}
\begin{proof}
We first divide the domain of the integration into three regions:
\begin{align*} 
&D_1:=\left\{ (x,y)\in \mathbb{T}^2 : |x-y| < 1/2 \right\},\\
&D_2:=\left\{ (x,y)\in \mathbb{T}^2 : |x-y| > 1/2, \quad ||x-y|-2\pi | <\frac{1}{2}\right\},\\
&D_3:=\left\{ (x,y)\in \mathbb{T}^2 : |x-y| > 1/2, \quad ||x-y|-2\pi | >\frac{1}{2}\right\}.
\end{align*}
Clearly, $\mathbb{T}^2 = D_1 \cup D_2 \cup D_3$. We will show only that
\begin{align}\label{appendix_claim2}
\int_{D_1}\left|D^{k} \tilde{K}(x,y)\right|^2dxdy \lesssim_{k} \rVert K \rVert_{H^{k+1}(\mathbb{T}^2)}^2,
\end{align}
since the other contributions from $D_2$ and $D_3$ can be easily proved in the same or easier way.

For $x,y\in D_1$, we can write
\[
\tilde{K}(x,y) = \int_{0}^1\partial_{u}\left(K(x,(x-y)u+y)\right)du\frac{1}{\sin\left( \frac{x-y}{2}\right)} = \int_0^{1}\partial_yK(x,(x-y)u+y)du\frac{(x-y)}{\sin(\frac{x-y}{2})}.
\]
Since $\frac{(x-y)}{\sin\left( \frac{x-y}{2}\right)}$ can be extended to a smooth function in $D_1$, we use Lemma~\ref{ponce_kato} and obtain
\begin{align*}
\int_{D_1}\left|D^{k} \tilde{K}(x,y)\right|^2dxdy  \lesssim \bigg\rVert  \int_0^{1}\partial_yK(x,(x-y)u+y)du\bigg\rVert_{H^{k}(\mathbb{T}^2)}^2.
\end{align*}
 The right-hand side can be estimated as
\begin{align*}
\int_{\mathbb{T}^2} \left| D^{k}  \int_0^{1}\partial_yK(x,(x-y)u+y)du|\right|^2 dxdy  &\lesssim \int_{\mathbb{T}^2} \left( \int_0^1 \left| D^{k} \partial_yK(x,(x-y)u + y) (1+u^k) \right| du \right)^2 dxdy \\
&  = \int_{\mathbb{T}^2} \left|\int_{y}^{x} \left| D^{k} \partial_yK(x,z)\right| \left( 1 + \left( \frac{(z-y)}{x-y}\right)^k \right) \frac{1}{x-y}dz \right|^2 dxdy \\
& \lesssim \int_{\mathbb{T}^2} \left| \int_y^x D^{k} \partial_yK(x,z)\frac{1}{|x-y|}dz \right|^2 dxdy,
\end{align*}
where the equality follows from the change of variables, $(x-y)u + u \mapsto z$. Using the Cauchy-Schwarz inequality, we obtain
\begin{align*}
\int_{\mathbb{T}^2} \left| \int_y^x D^{k} \partial_yK(x,z)\frac{1}{|x-y|}dz \right|^2 dxdy  & \lesssim \int_{\mathbb{T}^2}\int_{y}^x \left| D^{k} \partial_yK(x,z)\right|^2 dz \frac{1}{|x-y|}dxdy \\
& \lesssim \int_{(x,y,z)\in \mathbb{T}^3, y<z<x} \left|D^{k} \partial_yK(x,z)\right|^2\frac{1}{|x-y|} dxdydz \\
& \lesssim \int_{\mathbb{T}^2} \left|  D^{k} \partial_yK(x,z)\right|^2 \int_z^{2\pi} -\log \left( 1 - \frac{z}{x}\right) dx dz\\
& \lesssim  \int_{\mathbb{T}^2} \left| D^{k} \partial_yK(x,z)\right|^2 dz\\
& \lesssim \rVert K \rVert_{H^{k+1}(\mathbb{T}^2)}^2,
\end{align*}
where the second last inequality follows from the fact that $z\mapsto \int_z^{2\pi} -\log\left( 1 - \frac{z}{x}\right) dx$ is uniformly bounded. This finishes the proof.\end{proof}

For $R\in C^{\infty}(\mathbb{T})$, we denote
\begin{align}
& J(R)(x,y) = \frac{R(x) - R(y)}{2\sin(\frac{x-y}{2})}. \label{Jdef}
\end{align}

\begin{lemma}\label{appendix_lem_1}
Let $R\in C^{\infty}(\mathbb{T})$. Then, for $k\in \mathbb{N}\cup \left\{ 0 \right\}$, 
\begin{align}\label{J_estimate1}
 \int_{\mathbb{T}^2}\left|D^{k} J(R)(x,y)\right|^2dxdy \lesssim_{k} \rVert R \rVert_{H^{k+1}(\mathbb{T})}^2.
\end{align}
Consequently, $ \rVert J(R) \rVert_{H^{k}(\mathbb{T}^2)} \lesssim_{k} \rVert R \rVert_{H^{k+1}(\mathbb{T})}$.
Furthermore, we have
\begin{align}\label{Hilbert_transform}
\bigg\rVert \int_{\mathbb{T}}J(R)(\cdot,y) dy \bigg\rVert_{H^{k}(\mathbb{T})} \lesssim \rVert R \rVert_{H^{k}(\mathbb{T})}.
\end{align}
\end{lemma}

\begin{proof}
For \eqref{J_estimate1}, it follows from Lemma~\ref{appendix_lem_2} by putting $K(x,y)=R(y)$.  \eqref{Hilbert_transform} follows from classical singular integral theory, thus we omit the proof.
 \end{proof}

\begin{lemma}\label{GN}
Let $f,g\in C^{\infty}(\mathbb{T}^2)$ and let 
\[
T(x):=\int_{\mathbb{T}}f(x,y)g(x,y)dy.
\]
Then for $k\in \mathbb{N}\cup\left\{ 0\right\}$,
\begin{align*}
\rVert T \rVert_{H^{k}(\mathbb{T})} \lesssim \rVert f \rVert_{L^\infty(\mathbb{T}^2)}\rVert g \rVert_{H^k(\mathbb{T}^2)}+\rVert f \rVert_{H^k(\mathbb{T}^2)}\rVert g \rVert_{L^\infty(\mathbb{T}^2)}.
\end{align*}
\end{lemma}
\begin{proof}
By Minkowski's integral inequality and lemma \ref{ponce_kato}, We have
\begin{align*}
    \bigg\|\int_{\mathbb{T}}f(x,y)g(x,y)dy\bigg\|_{H^{k}(\mathbb{T})} & \lesssim\int_{\mathbb{T}}\|f(\cdot, y)g(\cdot,y)\|_{H^{k}(\mathbb{T})}dy\\
    & \lesssim\int_{\mathbb{T}}\|f(\cdot,y)\|_{L^{\infty}(\mathbb{T})}\|g(\cdot,y)\|_{H^{k}(\mathbb{T})}dy+\int_{\mathbb{T}}\|g(\cdot,y)\|_{L^{\infty}(\mathbb{T})}\|f(\cdot,y)\|_{H^{k}(\mathbb{T})}dy\\
    & \lesssim \|f\|_{L^{\infty}(\mathbb{T}^2)}\left(\int_{\mathbb{T}}\|g(\cdot,y)\|_{H^{k}(\mathbb{T})}^2 dy\right)^{\frac{1}{2}}+\|g\|_{L^{\infty}(\mathbb{T}^2)}\left(\int_{\mathbb{T}}\|f(\cdot,y)\|_{H^{k}(\mathbb{T})}^2 dy\right)^{\frac{1}{2}}\\
    & \lesssim \|f\|_{L^{\infty}(\mathbb{T}^2)}\|g\|_{H^k(\mathbb{T}^2)}+\|g\|_{L^{\infty}(\mathbb{T}^2)}\|f\|_{H^k(\mathbb{T}^2)}
\end{align*}
\end{proof}

\begin{lemma}\label{J_linear}
Let $K\in C^{\infty}(\mathbb{T}^2)$, ${K}^\#(x):=K(x,x)$ and $R\in C^\infty(\mathbb{T})$. Let
\[
T(R)(x):=\int_{\mathbb{T}} K(x,y)J(R)(x,y) dy.
\]
Then
for $k\ge 2$, 

\[
\rVert T(R)\rVert_{H^{k}(\mathbb{T})} \lesssim_k (\rVert K \rVert_{H^{k+1}(\mathbb{T}^2)} + \rVert {K}^\# \rVert_{H^{k}(\mathbb{T})})\rVert R \rVert_{H^1(\mathbb{T})} + (\rVert {K}^\# \rVert_{L^\infty(\mathbb{T})} + \rVert \nabla K \rVert_{L^\infty(\mathbb{T}^2)})\rVert R \rVert_{H^{k}(\mathbb{T})}. 
\]
\end{lemma}

\begin{proof}
Let $\tilde{K}(x,y):=\frac{K(x,y)-{K}^\#(x)}{2\sin\left(\frac{x-y}{2}\right)}$. Clearly, we have
\begin{align}\label{tilde_K}
\rVert \tilde{K}\rVert_{L^{\infty}(\mathbb{T}^2)} \lesssim \rVert \nabla K \rVert_{L^\infty(\mathbb{T})}.
\end{align} 
We write
\[
T(R)(x) = {K}^\#(x)\int_{\mathbb{T}}J(R)(x,y)dy + \int_{\mathbb{T}}\tilde{K}(x,y)(R(x)-R(y))dy=:T_1(R)(x)+T_2(R)(x).
\]
We can then get the $L^2$ estimate directly by H\"older's inequality. From Lemma~\ref{ponce_kato}, \ref{appendix_lem_1} and $H^1(\mathbb{T})\hookrightarrow L^\infty(\mathbb{T})$, we have
\begin{align}\label{T_1estimate}
\rVert T_1(R) \rVert_{H^{k}(\mathbb{T})}& \lesssim \rVert {K}^\#(x) \rVert_{H^k(\mathbb{T})}\bigg\rVert \int_{\mathbb{T}}J(R)(\cdot,y)dy\bigg\rVert_{H^{1}(\mathbb{T})} + \rVert {K}^\#\rVert_{L^{\infty}(\mathbb{T})}\bigg\rVert \int_{\mathbb{T}}J(R)(\cdot,y)dy\bigg\rVert_{H^{k}(\mathbb{T})}\nonumber\\
& \lesssim \rVert {K}^\#(x) \rVert_{H^k(\mathbb{T})} \rVert R \rVert_{H^1(\mathbb{T})} +\rVert {K}^\#\rVert_{L^{\infty}(\mathbb{T})}\rVert R \rVert_{H^{k}(\mathbb{T})}.
\end{align}
For $T_2(R)$, it follows from Lemma~\ref{GN} that
\begin{align*}
\rVert T_2(R) \rVert_{H^{k}(\mathbb{T})} &\lesssim \rVert \tilde{K} \rVert_{H^{k}(\mathbb{T}^2)}\rVert R \rVert_{L^{\infty}(\mathbb{T})} + \rVert \tilde{K} \rVert_{L^{\infty}(\mathbb{T}^2)}\rVert R \rVert_{H^{k}(\mathbb{T})}\\
&\lesssim \rVert K \rVert_{H^{k+1}(\mathbb{T}^2)}\rVert R \rVert_{H^{1}(\mathbb{T})} + \rVert \nabla K \rVert_{L^\infty(\mathbb{T})}\rVert R \rVert_{H^{k}(\mathbb{T})},
\end{align*}
where the last inequality follows from \eqref{tilde_K}. With \eqref{T_1estimate}, we obtain the desired result.

\end{proof}

\begin{lemma}\label{twojs}
Let $K\in C^{\infty}(\mathbb{T}^2)$, ${K}^\#(x):=K(x,x)$ and $R\in C^\infty(\mathbb{T})$. Let
\[
T(R)(x):=\int_{\mathbb{T}} K(x,y)J(R)(x,y)J(R)(x,y)dy.
\]
Then for $k\ge 3$, 
\[
\rVert T(R)\rVert_{H^{k}(\mathbb{T})} \lesssim_k \rVert K \rVert_{H^{k+1}(\mathbb{T}^2)}\rVert R\rVert_{H^k(\mathbb{T}^2)}^2
\]

\end{lemma}

\begin{proof}
Clearly, 
\begin{align}\label{L2case}
\rVert T(R)\rVert_{L^2{(\mathbb{T})}} & \lesssim \rVert K \rVert_{L^{\infty}(\mathbb{T}^2)}\rVert J(R) \rVert_{L^{\infty}(\mathbb{T})}\rVert J(R) \rVert_{L^2(\mathbb{T}^2)}\nonumber\\
&  \lesssim  \rVert K \rVert_{H^{2}(\mathbb{T}^2)}\rVert J(R)\rVert_{H^{2}(\mathbb{T}^2)}\rVert R \rVert_{H^1(\mathbb{T})} \nonumber\\
& \lesssim \rVert K \rVert_{H^{2}(\mathbb{T}^2)}\rVert R\rVert_{H^{3}(\mathbb{T})}\rVert R \rVert_{H^1(\mathbb{T})}\nonumber\\
&\lesssim \rVert K \rVert_{H^{2}(\mathbb{T}^2)}\rVert R\rVert_{H^{3}(\mathbb{T})}^2,
\end{align}
where we used  $H^{2}(\mathbb{T}^2)\hookrightarrow L^{\infty}(\mathbb{T}^2)$ and Lemma~\ref{appendix_lem_1}.

 Now we compute the higher derivatives. Using the change of variables, $y\mapsto x-y$, we have
\[
T(R)(x) = \int_{\mathbb{T}}K(x,x-y)J(R)(x,x-y)J(R)(x,x-y)dy.
\]
Since $J(R)(x,x-y)=\frac{R(x)-R(x-y)}{2\sin\left(\frac{y}{2}\right)}$, we have
\begin{align*}
\left(\frac{d}{dx} \right)^{k}T(R) &= \sum_{p_1+p_2+p_3=k}\int_{\mathbb{T}}((\partial_x)^{p_1}K(x,x-y)) J(R^{(p_2)})(x,x-y)J(R^{(p_3)})(x,x-y)dy\\
& = \sum_{p_1+p_2+p_3=k}\int_{\mathbb{T}}((\partial_x)^{p_1}\tilde{K}(x,x-y))J(R^{(p_2)})(x,y)J(R^{(p_3)})(x,y)dy\\
&= \sum_{p_1+p_2+p_3=k} A_{p_1,p_2,p_3},
\end{align*}
where $\tilde{K}(x,y) := K(x,x-y)$. If $p_2=k$ or $p_3=k$, then we can write (assuming $p_3=k$ without loss of generality),
\begin{align*}
A_{0,0,k} & = \int_{\mathbb{T}}\frac{{K}(x,y)J(R)(x,y)-{K}(x,x)J(R)(x,x)}{2\sin\left(\frac{x-y}{2}\right)}(R^{(k)}(x)-R^{(k)}(y))dy + {K}(x,x)R'(x)\int_{\mathbb{T}}J(R^{(k)})(x,y)dy
\\
&= \int_{\mathbb{T}}\frac{{K}(x,y)J(R)(x,y)-{K}(x,x)J(R)(x,x)}{2\sin\left(\frac{x-y}{2}\right)}R^{(k)}(x)dy -\int_{\mathbb{T}}\frac{{K}(x,y)J(R)(x,y)-{K}(x,x)J(R)(x,x)}{2\sin\left(\frac{x-y}{2}\right)}R^{(k)}(y)dy\\
& \ + {K}(x,x)R'(x)\int_{\mathbb{T}}J(R^{(k)})(x,y)dy\\
&=A_{0,0,k}^{1}+A_{0,0,k}^{2}+A_{0,0,k}^{3}.
\end{align*}
Using Lemma~\ref{appendix_lem_2} for the first two integrals, Lemma~\ref{appendix_lem_1} for the second integral, and Cauchy-Schwarz inequality, we obtain 
\begin{align*}
    \|A_{0,0,k}^{1}\|_{L^2(\mathbb{T})}&\lesssim\bigg\|\int_{\mathbb{T}}\frac{{K}(x,y)J(R)(x,y)-{K}(x,x)J(R)(x,x)}{2\sin\left(\frac{x-y}{2}\right)}dy\bigg\|_{L^{\infty}(\mathbb{T})}\|R\|_{H^{k}(\mathbb{T})}\\
    &\lesssim\bigg\|\int_{\mathbb{T}}\frac{{K}(x,y)J(R)(x,y)-{K}(x,x)J(R)(x,x)}{2\sin\left(\frac{x-y}{2}\right)}dy\bigg\|_{H^{1}(\mathbb{T})}\|R\|_{H^{k}(\mathbb{T})}\\
    &\lesssim\|KJ(R)\|_{H^2(\mathbb{T}^2)}\|R\|_{H^k(\mathbb{T})}\\
    &\lesssim(\|K\|_{L^{\infty}(\mathbb{T}^2)}\|J(R)\|_{H^2(\mathbb{T}^2)}+\|J(R)\|_{L^{\infty}(\mathbb{T}^2)}\|K\|_{H^2(\mathbb{T}^2)})\|R\|_{H^k(\mathbb{T})},
\end{align*}
\begin{align*}
    \|A_{0,0,k}^{2}\|_{L^2(\mathbb{T})}&\lesssim \int_{\mathbb{T}}\bigg\|\frac{{K}(\cdot,y)J(R)(\cdot,y)-{K}(\cdot,\cdot)J(R)(\cdot,\cdot)}{2\sin\left(\frac{\cdot-y}{2}\right)}\bigg\|_{L^{\infty}(\mathbb{T})}|R^{(k)}(y)|dy\\
    &\lesssim\int_{\mathbb{T}}\bigg\|\frac{{K}(\cdot,y)J(R)(\cdot,y)-{K}(\cdot,\cdot)J(R)(\cdot,\cdot)}{2\sin\left(\frac{\cdot-y}{2}\right)}\bigg\|_{H^1(\mathbb{T})}|R^{(k)}(y)|dy\\
     &\lesssim\bigg\|\frac{{K}(x,y)J(R)(x,y)-{K}(x,x)J(R)(x,x)}{2\sin\left(\frac{x-y}{2}\right)}\bigg\|_{H^1(\mathbb{T}^2)}\|R\|_{H^k(\mathbb{T})}\\
     &\lesssim\|KJ(R)\|_{H^2(\mathbb{T}^2)}\|R\|_{H^k(\mathbb{T})}\\
    &\lesssim(\|K\|_{L^{\infty}(\mathbb{T}^2)}\|J(R)\|_{H^2(\mathbb{T}^2)}+\|J(R)\|_{L^{\infty}(\mathbb{T}^2)}\|K\|_{H^2(\mathbb{T}^2)})\|R\|_{H^k(\mathbb{T})},
\end{align*}
and
\begin{align*}
    \|A_{0,0,k}^{3}\|_{L^2(\mathbb{T})}\lesssim \|K\|_{L^{\infty}(\mathbb{T}^2)}\|R\|_{H^2(\mathbb{T})}\|R\|_{H^k(\mathbb{T})}.
\end{align*}
Using the Sobolev embedding theorem, we have $\rVert K\rVert_{L^{\infty}(\mathbb{T}^2)}\lesssim \rVert K \rVert_{H^2(\mathbb{T}^2)}$ and $\rVert J(R)\rVert_{L^\infty(\mathbb{T}^2)}\lesssim \rVert J(R)\rVert_{H^2(\mathbb{T}^2)}$, thus
\begin{align}\label{00kcase}
\rVert A_{0,0,k}\rVert_{L^2(\mathbb{T})} & \lesssim \rVert K \rVert_{H^2(\mathbb{T}^2)}\left( \rVert J(R)\rVert_{H^2(\mathbb{T}^2)} + \rVert R \rVert_{H^2(\mathbb{T})}\right)\rVert R\rVert_{H^k(\mathbb{T})}\nonumber\\
&\lesssim \rVert K \rVert_{H^2(\mathbb{T}^2)} \rVert R\rVert_{H^3(\mathbb{T})}\rVert R \rVert_{H^k(\mathbb{T})}\nonumber \\
& \lesssim \rVert K \rVert_{H^2(\mathbb{T}^2)} \rVert R \rVert_{H^k(\mathbb{T})}^2,
\end{align}
where the second inequality follows from \eqref{appendix_lem_1} and the last inequality follows from $k\ge 3$.
If $p_1,p_2,p_3 < k$, We use Lemma \ref{GN} to obtain
\begin{align}\label{k1case}
\rVert A_{p_1,p_2,p_3}\rVert_{L^2(\mathbb{T})} & \lesssim
\bigg \|\int_{\mathbb{T}}\partial_{x}(K(x,x-y))J(R)(x,x-y)J(R)(x,x-y)dy\bigg\|_{H^{k-1}(\mathbb{T}^2)}
\\\nonumber
&=\bigg\|\int_{\mathbb{T}}((\partial_{x}K)(x,y)+(\partial_{y}K)(x,y))J(R)(x,y)J(R)(x,y)dy\bigg\|_{H^{k-1}(\mathbb{T}^2)}\\\nonumber
&\lesssim\|\nabla K\|_{H^{k-1}(\mathbb{T}^2)}\|J(R)J(R)\|_{L^{\infty}(\mathbb{T}^2)}+\|J(R)J(R)\|_{H^{k-1}(\mathbb{T}^2)}\|\nabla K\|_{L^{\infty}(\mathbb{T}^2)}\\\nonumber
&\lesssim \|K\|_{H^{k}(\mathbb{T}^2)}\|J(R)\|_{H^2(\mathbb{T}^2)}\|J(R)\|_{H^2(\mathbb{T}^2)}+\|J(R)\|_{H^{k-1}(\mathbb{T}^2)}\|J(R)\|_{L^\infty(\mathbb{T}^2)}\|K\|_{H^{3}(\mathbb{T}^2)}\\\nonumber
&\lesssim\|K\|_{H^{k}(\mathbb{T}^2)}\|R\|_{H^{k}(\mathbb{T}^2)}^2
\end{align}
 Similarly, if $p_1 = k$, then
\begin{align}\label{k00case}
\rVert A_{k,0,0}\rVert_{L^2(\mathbb{T})} & \lesssim \rVert \tilde{K}\rVert_{H^{k}(\mathbb{T}^2)}\rVert J(R)\rVert_{L^{\infty}(\mathbb{T}^2)}^2 \nonumber\\
& \lesssim \rVert K \rVert_{H^{k}(\mathbb{T}^2)}\rVert J(R)\rVert_{H^2(\mathbb{T}^2)}^2\nonumber\\
& \lesssim \rVert K \rVert_{H^{k}(\mathbb{T}^2)}\rVert R\rVert_{H^3(\mathbb{T}^2)}^2\nonumber\\
& \lesssim \rVert K \rVert_{H^{k}(\mathbb{T}^2)}\rVert R\rVert_{H^k(\mathbb{T}^2)}^2.
\end{align}
Therefore it follows from \eqref{00kcase}, \eqref{k1case} and \eqref{k00case} that 
\begin{align*}
\bigg\rVert \left(\frac{d}{dx} \right)^{k}T(R) \bigg\rVert_{H^{k}(\mathbb{T})} \lesssim \rVert K \rVert_{H^{k+1}(\mathbb{T}^2)}\rVert R\rVert_{H^k(\mathbb{T}^2)}^2.
\end{align*}
With \eqref{L2case}, the desired result follows
\end{proof}

\begin{lemma}\label{T1estimate}
Define a linear map
\[
T(R)(x) := \int_{\mathbb{T}}\log(2-2\cos(x-y))R(y)dy.
\]
Then for $k\in \mathbb{N}$, $T:H^{k-1}(\mathbb{T}) \mapsto H^{k}(\mathbb{T})$ is an isomorphism. 
\end{lemma}

\begin{proof}
The result follows immediately from Corollary~\ref{integral_a0}.
\end{proof}

\subsection{Stability of linear operators}

\begin{lemma}\label{functional_stability}
For Hilbert spaces $X,Y,Z$, assume that $A:X\mapsto \mathcal{L}(Y,Z)$ satisfies:
\begin{itemize}
\item[(1)]$X\ni x \mapsto A(x) \in \mathcal{L}(Y,Z) \text{ is Lipschitz continuous near $0\in X$},$
\item[(2)]$A(0) \text{ is injective and has codimension one, that is, there exists $0\ne z\in Z$ such that} $ \begin{align*}\text{Im}(A(0))^\perp = \text{span}\left\{ z \right\}.\end{align*}
Then there exists $\epsilon_0>0$ such that if $|x|_X \le \epsilon_0$, then $A(x)$ is also injective and has codimension one. Furthermore, there exist $c>0$ and  $z_1\in Z$ such that 
\begin{align}
&\text{Im}(A(x))^{\perp}=\text{span}\left\{z_1\right\} \label{Atilde11}\\
&|z - z_1|_{Z} \le c|x|_X, \label{Atilde_21} \\
&|A(x)|_{\mathcal{L}(Y, \text{Im}(A(x)))} \le c, \label{Atilde_31} \\
&|A(x)^{-1}|_{\mathcal{L}( \text{Im}(A(x)),Y)} \le c.\label{Atilde_41}
\end{align}
\end{itemize}
\end{lemma}
\begin{proof}
 Let $Q:Z\mapsto Im(A(0))$ be the projection. Then $QA(0):Y\mapsto Im(A(0))$ is an isomorphism. Therefore we can choose $\ep>0$ small enough so that $QA(x):Y\mapsto Im(A(0))$ is an isomorphism whenever $|x| \leq \ep$. Therefore, if $A(x)[y] = 0$ for some $y\in Y$, then $QA(x)[y] = 0$, which implies $y=0$. This shows $A(x)$ is injective. In addition, the Fredholm index of $A(0)$ is $\text{dim(Ker$(A(0))$)} - \text{dim(coker($A(0)$))} = -1$, therefore the Fredholm index of $A(x)$ is also $-1$ (\cite[Theorem 4.4.2]{Buhler-Salamon:functional-analysis-book}) for small $\epsilon$, which implies $A(x)$ has codimension one. This proves \eqref{Atilde11}. Furthermore, $A(x):Y\mapsto Im(A(x))=\left(\text{span}\left\{ z_1\right\} \right)^\perp$ is an isomorphism, therefore \eqref{Atilde_31} and \eqref{Atilde_41} follow immediately. We remark that the constant $c$ can be chosen uniformly in $x$.

 In order to prove \eqref{Atilde_21}, let us assume without loss of generality that $|z|_Z = |z_1|_Z = 1$. For any $y\in Y$, we have
 \[
 \langle A(0)[y],z_1\rangle_Z =    \langle (A(0)-A(x))[y],z_1\rangle_Z + \langle A(x)[y],z_1\rangle_Z  =  \langle (A(0)-A(x))[y],z_1\rangle_Z,
\]
where the last equality follows from $A(x)[y]\in \text{Im}(A(x))$ and $z_1\in \text{Im}(A(x))^\perp$. Therefore, we have
\begin{align}\label{functional_11}
|\langle A(0)[y],z_1\rangle_Z| \le |A(0)-A(x)|_{\mathcal{L}(Y;Z)}|y|_Y.
\end{align}
Note that  we can decompose $z_1$ and  find $y\in Y$ such that
\[
z_1 = \langle z,z_1 \rangle_Z z + \left(\underbrace{ z_1 - \langle z,z_1 \rangle_Z z }_{=:w}\right) =  \langle z,z_1 \rangle_Z z + A(0)[y],
\]
where the last equality follows from $w \in \text{Im}(A(0))$, hence $y=A(0)^{-1}[w]$. Then, we have
\begin{align*}
|w|_Z^2 & = \langle w,w\rangle_Z \\
&=\langle w, z_1 \rangle_Z \\
& = \langle A(0)[y],z_1\rangle_Z  \\
& \le |A(0)-A(x)|_{\mathcal{L}(Y;Z)}|y|_Y \\
& \le |A(0)-A(x)|_{\mathcal{L}(Y;Z)}|A(0)^{-1}|_{\mathcal{L}(\text{Im}(A(0));Y)}|w|_Z,
\end{align*}
where the second equality follows from $\langle w,z \rangle_Z = 0$ and the first inequality follows from \eqref{functional_11}. 
Hence 
\begin{align}\label{functional_22}
|w|_Z \le  |A(0)-A(x)|_{\mathcal{L}(Y;Z)}|A(0)^{-1}|_{\mathcal{L}(\text{Im}(A(0));Y)}.
\end{align}
Now, using $z_1 = \langle z,z_1\rangle_Zz + w,$ we have
\[
1 = \langle z_1,z_1 \rangle_Z =\langle z,z_1 \rangle_Z^2 + \langle w,z_1\rangle_Z = \langle z,z_1 \rangle_Z^2 + \langle w,z_1-z\rangle_Z,
\]
where the last equality follows from $\langle w,z\rangle = 0$. Hence, we have
\[
\langle z,z_1 \rangle_Z = \sqrt{1-\langle w,z_1-z\rangle_Z}.
\]
Therefore we have
\[
|z-z_1|_Z^2 = 2-2\langle z,z_1\rangle_Z = 2-2\sqrt{1-\langle w,z_1-z\rangle_Z}\le |\langle w, z_1-z\rangle_z| \le |w|_Z|z_1-z|_Z,
\]
which implies 
\[
|z-z_1| \le |w|_Z \le |A(0)-A(x)|_{\mathcal{L}(Y;Z)}|A(0)^{-1}|_{\mathcal{L}(\text{Im}(A(0));Y)} \lesssim |x|_X,
\]
where the first inequality follows from \eqref{functional_22} and the last inequality follows from the Lipschitz continuity of $x\mapsto A(x)$. This proves \eqref{Atilde_21}.
\end{proof}

\section*{Acknowledgements}

JGS was partially supported by NSF through Grant NSF DMS-1763356, and by the European Research Council through ERC-StG-852741-CAPA. JP was partially supported by NSF through Grants NSF DMS-1715418, NSF CAREER Grant DMS-1846745 and by the European Research Council through ERC-StG-852741-CAPA. JS was partially supported by NSF through Grant NSF DMS-1700180 and by the European Research Council through ERC-StG-852741-CAPA. JGS would like to thank Francisco Gancedo for useful discussions. We would like to thank Claudia Garc\'ia for mentioning the references \cite{Kinderlehrer-Nirenberg:regularity-free-boundary,Kinderlehrer-Nirenberg-Spruck:regularity-elliptic-free-boundary} to us and Yao Yao for all the conversations on this project over the last years.
This material is based upon work supported by the National Science Foundation under Grant No. DMS-1929284 while the authors were in residence at the Institute for Computational and Experimental Research in Mathematics in Providence, RI, during the program `Hamiltonian Methods in Dispersive and Wave Evolution Equations''.

\bibliographystyle{abbrv}
\bibliography{references}

\begin{tabular}{l}
\textbf{Javier G\'omez-Serrano}\\
{Department of Mathematics} \\
{Brown University} \\
{Kassar House, 151 Thayer St.} \\
{Providence, RI 02912, USA} \\ \\
{and} \\ \\ 
{Departament de Matem\`atiques i Inform\`atica} \\
{Universitat de Barcelona} \\
{Gran Via de les Corts Catalanes, 585} \\
{08007, Barcelona, Spain} \\
{Email: javier\_gomez\_serrano@brown.edu, jgomezserrano@ub.edu} \\ \\
\textbf{Jaemin Park} \\
{Departament de Matem\`atiques i Inform\`atica} \\
{Universitat de Barcelona} \\
{Gran Via de les Corts Catalanes, 585} \\
{08007, Barcelona, Spain} \\
{Email: jpark@ub.edu} \\ \\
\textbf{Jia Shi}\\
{Department of Mathematics}\\
{Princeton University} \\
{409 Fine Hall, Washington Rd,}\\
{Princeton, NJ 08544, USA}\\
{e-mail: jiashi@math.princeton.edu}\\ \\
\end{tabular}

\end{document}